\documentclass{amsart}
\usepackage{amsmath,amssymb,amsthm,bm,bbm}
\usepackage{amscd}
\usepackage{graphicx,enumerate}
\usepackage{multirow}
\usepackage{layout}
\usepackage[OT2,T1]{fontenc}
\DeclareSymbolFont{cyrletters}{OT2}{wncyr}{m}{n}
\DeclareMathSymbol{\Sha}{\mathalpha}{cyrletters}{"58}
\usepackage[OT2,OT1]{fontenc}
\newcommand\cyr{\renewcommand\rmdefault{wncyr}
\renewcommand\sfdefault{wncyss}
\renewcommand\encodingdefault{OT2}
\normalfont\selectfont}
\DeclareTextFontCommand{\textcyr}{\cyr}

\usepackage{url}

\theoremstyle{plain}
\newtheorem{theorem}{Theorem}[section]
\newtheorem*{theorem-nn}{Theorem}
\newtheorem{lemma}[theorem]{Lemma}
\newtheorem{proposition}[theorem]{Proposition}
\newtheorem*{proposition-nn}{Proposition}
\newtheorem{corollary}[theorem]{Corollary}

\theoremstyle{definition}
\newtheorem{definition}[theorem]{Definition}
\newtheorem{example}[theorem]{Example}
\newtheorem{remark}[theorem]{Remark}

\newtheorem*{acknowledgment}{Acknowledgment}
\theoremstyle{remark}

\newcommand{\bZ}{\mathbbm{Z}}\newcommand{\bQ}{\mathbbm{Q}}
\newcommand{\bC}{\mathbbm{C}}
\newcommand{\bG}{\mathbbm{G}}\newcommand{\bF}{\mathbbm{F}}
\newcommand{\bA}{\mathbbm{A}}

\newcommand{\GL}{{\rm GL}}
\newcommand{\PSL}{{\rm PSL}}

\title{Norm one tori and Hasse norm principle, III: Degree $16$ case}

\author[A. Hoshi]{Akinari Hoshi}
\address{Department of Mathematics, Niigata University, Niigata 950-2181, Japan}
\email{hoshi@math.sc.niigata-u.ac.jp}

\author[K. Kanai]{Kazuki Kanai}
\address{General Education Program, National Institute of Technology, Kure College, Hiroshima 737-8506, Japan}
\email{k-kanai@kure-nct.ac.jp}

\author[A. Yamasaki]{Aiichi Yamasaki}
\address{Department of Mathematics, Kyoto University, Kyoto 606-8502, Japan}
\email{aiichi.yamasaki@gmail.com}

\thanks{
This work was partially supported by JSPS KAKENHI Grant Numbers 
19K03418, 20H00115, 20K0351, 24K00519, 24K06647.
}


\begin{document}
\maketitle
\begin{abstract}
Let $k$ be a field, $T$ be an algebraic $k$-torus, 
$X$ be a smooth $k$-compactification of $T$ 
and ${\rm Pic}\,\overline{X}$ be the Picard group of 
$\overline{X}=X\times_k\overline{k}$ where $\overline{k}$ 
is a fixed separable closure of $k$.  
Hoshi, Kanai and Yamasaki \cite{HKY22}, \cite{HKY23} determined 
$H^1(k,{\rm Pic}\, \overline{X})$ 
for norm one tori $T=R^{(1)}_{K/k}(\bG_m)$ 
and gave 
a necessary and sufficient condition for the Hasse norm principle for 
extensions $K/k$ of number fields with $[K:k]\leq 15$. 
In this paper, we treat the case where $[K:k]=16$. 
Among $1954$ transitive subgroups 
$G=16Tm\leq S_{16}$ $(1\leq m\leq 1954)$ up to conjugacy, 
we determine $1101$ (resp. $774$, $31$, $37$, $1$, $1$, $9$) cases 
with $H^1(k,{\rm Pic}\, \overline{X})=0$ (resp. $\bZ/2\bZ$, 
$(\bZ/2\bZ)^{\oplus 2}$, $(\bZ/2\bZ)^{\oplus 3}$, $(\bZ/2\bZ)^{\oplus 4}$, 
$(\bZ/2\bZ)^{\oplus 6}$, $\bZ/4\bZ$) 
where $G$ is the Galois group of the Galois closure $L/k$ of $K/k$. 
We see that $H^1(k,{\rm Pic}\, \overline{X})=0$ 
implies that the Hasse norm principle holds for $K/k$. 
In particular, among $22$ primitive $G=16Tm$ cases, i.e. 
$H\leq G=16Tm$ is maximal with $[G:H]=16$, 
we determine exactly 
$6$ cases $(m=178, 708, 1080, 1329, 1654, 1753)$ 
with $H^1(k,{\rm Pic}\, \overline{X})\neq 0$ $($$(\bZ/2\bZ)^{\oplus 2}$, 
$\bZ/2\bZ$, $(\bZ/2\bZ)^{\oplus 2}$, $\bZ/2\bZ$, $\bZ/2\bZ$,  $\bZ/2\bZ$). 
Moreover, we give a necessary and sufficient condition for the Hasse norm principle 
for $K/k$ with $[K:k]=16$ for $22$ primitive $G=16Tm$ cases. 
As a consequence of the $22$ primitive $G$ cases,  
we get the Tamagawa number $\tau(T)=1$, $1/2$, $1/4$ of 
$T=R^{(1)}_{K/k}(\bG_m)$ over a number field $k$ via Ono's formula 
$\tau(T)=1/|\Sha(T)|$ where $\Sha(T)$ is the Shafarevich-Tate group of $T$. 
\end{abstract}
\tableofcontents
%
\section{Introduction}\label{S1}
Let $k$ be a field, 
$K/k$ be a separable field extension of degree $n$ 
and $L/k$ be the Galois closure of $K/k$ with $G={\rm Gal}(L/k)$. 
Let $H={\rm Gal}(L/K)\leq G$ with $[G:H]=n$. 
Let $nTm$ be the $m$-th transitive subgroup of the symmetric group 
$S_n$ of degree $n$ up to conjugacy 
(see Butler and McKay \cite{BM83}, \cite{GAP}). 
The Galois group $G$ may be regarded as a transitive subgroup 
$nTm\leq S_n$ via an injection $G\to S_n$ 
which is derived from the action of $G$ on the left cosets 
$\{g_1H,\ldots,g_nH\}$ by $g(g_iH)=(gg_i)H$ for any $g\in G$. 
We may assume that 
$H$ is the stabilizer of one of the letters in $G$, 
i.e. $L=k(\theta_1,\ldots,\theta_n)$ and $K=k(\theta_i)$ for some 
$1\leq i\leq n$. 
We just take $K=k(\theta_1)$ for the simplicity. 
Let $S_n$ (resp. $A_n$, $D_n$, $C_n$) be the symmetric 
(resp. the alternating, the dihedral, the cyclic) group 
of degree $n$ of order $n!$ (resp. $n!/2$, $2n$, $n$). 
Let $V_4\simeq C_2\times C_2$ be the Klein four group. 

Voskresenskii \cite{Vos67} proved that 
all the $2$-dimensional algebraic $k$-tori are $k$-rational. 
This implies that 
$H^1(k,{\rm Pic}\,\overline{X})\simeq A(T)\simeq \Sha(T)=0$ 
where $X$ is a smooth $k$-compactification of an algebraic $k$-torus $T$, 
i.e. smooth projective $k$-variety $X$ 
containing $T$ as a dense open subvariety, 
${\rm Pic}\,\overline{X}$ is the Picard group of 
$\overline{X}=X\times_k\overline{k}$ over a fixed separable closure $\overline{k}$ of $k$,  
$A(T)$ is the kernel of the weak approximation of $T$ 
and $\Sha(T)$ is the Shafarevich-Tate group of $T$ 
(see Section \ref{S3} and also Manin \cite[\S 30]{Man86}). 
Kunyavskii \cite{Kun84} showed that, among $73$ cases of 
$3$-dimensional $k$-tori $T$, 
there exist exactly $2$ cases with 
$H^1(k,{\rm Pic}\,\overline{X})\neq 0$ 
which are of special type called norm one tori. 
He also proved that 
there exist $58$ (resp. $15$) stably $k$-rational 
(resp. not retract $k$-rational) cases (see also Hoshi and Yamasaki \cite[Theorem 1.2, Example 5.3]{HY17}). 
For the stably/retract rational classification of algebraic $k$-tori of dimensions $4$ and $5$, 
see Hoshi and Yamasaki \cite{HY17}. 
Hoshi, Kanai and Yamasaki \cite[Theorem 1.15]{HKY22}, \cite[Theorem 1.1]{HKY23} determined 
$H^1(k,{\rm Pic}\, \overline{X})$ 
for norm one tori $T=R^{(1)}_{K/k}(\bG_m)$ 
with $n=[K:k]\leq 15$. 

For $n=16$, there exist exactly $1954$ transitive subgroups $16Tm\leq S_{16}$ 
$(1\leq m\leq 1954)$ up to conjugacy (see Hulpke \cite[Tabelle 1]{Hul96}, \cite[Table 1]{Hul05}, \cite{GAP}) 
and we get $H^1(k,{\rm Pic}\, \overline{X})$ as follows: 
\begin{theorem}\label{thmain1}
Let $k$ be a field, 
$K/k$ be a separable field extension of degree $16$ 
and $L/k$ be the Galois closure of $K/k$. 
Assume that $G={\rm Gal}(L/k)=16Tm$ $(1\leq m\leq 1954)$ 
is a transitive subgroup of $S_{16}$ 
and $H={\rm Gal}(L/K)$ with $[G:H]=16$. 
Let $T=R^{(1)}_{K/k}(\bG_m)$ be the norm one torus of $K/k$ of dimension $15$ 
and $X$ be a smooth $k$-compactification of $T$. 
Then 
\begin{align*}
H^1(k,{\rm Pic}\, \overline{X})=
\begin{cases}
\bZ/2\bZ & {\rm if}\ m\ {\rm is\ given\ as\ in\ Table}\ 1\textrm{-}1\ (774\ {\rm cases)},\\
(\bZ/2\bZ)^{\oplus 2} & {\rm if}\  m=7, 10, 11, 46, 58, 61, 73, 76, 82, 87, 89, 107, 113, 118, 120, 128, 129, 138, \\ 
& \hspace*{11mm} 142, 162, 164, 165, 178, 183, 206, 297, 308, 319, 414, 731, 1080\ (31\ {\rm cases)},\\
(\bZ/2\bZ)^{\oplus 3} & {\rm if}\ m=2, 9, 18, 20, 23, 25, 67, 69, 83, 92, 98, 101, 127, 173, 197, 202, 212, 241,\\ 
& \hspace*{11mm}  246, 270, 295, 301, 313, 358, 372, 440, 463, 466, 604, 632, 649, 656,\\
& \hspace*{11mm}  794, 801, 1082, 1187, 1378\ (37\ {\rm cases)},\\
(\bZ/2\bZ)^{\oplus 4} & {\rm if}\ m=64\ (1\ {\rm case),}\\
(\bZ/2\bZ)^{\oplus 6} & {\rm if}\ m=3\ (1\ {\rm case)},\\
\bZ/4\bZ & {\rm if}\ m=4, 51, 63, 143, 185, 323, 375, 430, 769\  (9\ {\rm cases)},\\
0 & {\rm otherwise}\ (1101\ {\rm cases)}. 
\end{cases}
\end{align*} 
In particular, if $H^1(k,{\rm Pic}\, \overline{X})=0$, then 
we get the vanishing 
$H^1(k,{\rm Pic}\, \overline{X})\simeq H^1(G,{\rm Pic}\, X_L)\simeq 
\Sha^2_\omega(G,J_{G/H})\simeq {\rm Br}(X)/{\rm Br}(k)\simeq 
{\rm Br}_{\rm nr}(k(X)/k)/{\rm Br}(k)=0$. 
This implies that, 
when $k$ is a global field, i.e. a finite extension of $\bQ$ or $\bF_q(t)$, 
$A(T)=0$ and $\Sha(T)=0$, i.e. $T$ has the weak approximation property, 
Hasse principle holds for all torsors $E$ under $T$ and 
the Hasse norm principle holds for $K/k$ 
$($see Section \ref{S3}, 
Hoshi, Kanai and Yamasaki \cite[Section 1]{HKY22}, \cite[Section 1]{HKY23},  
Hoshi and Yamasaki \cite[Section 4]{HY1}$)$. 
Moreover, if $k$ is a global field and 
$L/k$ is an unramified extension, then $A(T)=0$ and 
$H^1(k,{\rm Pic}\,\overline{X})\simeq \Sha(T)$. 
\end{theorem}
%
%
\begin{center}
\vspace*{1mm}
Table $1$-$1$: $H^1(k,{\rm Pic}\, \overline{X})\simeq H^1(G,[J_{G/H}]^{fl})\simeq\bZ/2\bZ$ 
with $G=16Tm$ $(1\leq m\leq 1954)$\vspace*{2mm}\\
{\footnotesize 
\begin{tabular}{lc} 
$m:  G=16Tm$ with $H^1(k,{\rm Pic}\, \overline{X})\simeq H^1(G,[J_{G/H}]^{fl})\simeq\bZ/2\bZ$ (774 cases)\\\hline
  5, 8, 13, 15, 16, 17, 19, 21, 27, 29, 30, 32, 33, 34, 35, 36, 37, 38, 39, 40, 
  42, 43, 44, 45, 47, 48, 50, 52, 53, 54, 57, 62, 68,\\ 
  70, 71, 72, 74, 75, 78, 80, 81, 86, 88, 90, 91, 93, 94, 96, 99, 100, 102, 105, 106, 
  109, 110, 111, 112, 114, 115, 116, 117, 119,\\ 
  122, 123, 130, 131, 132, 133, 134, 137, 140, 141, 148, 149, 150, 151, 152, 153, 154, 156, 
  159, 160, 166, 167, 169, 170, 171,\\ 
  172, 174, 175, 176, 177, 179, 180, 182, 184, 192, 193, 194, 198, 199, 200, 201, 203, 
  204, 205, 207, 208, 210, 214, 215, 217,\\ 
  218, 221, 223, 225, 229, 230, 231, 232, 237, 239, 244, 245, 248, 249, 250, 251, 
  254, 255, 256, 259,  260, 261, 262, 264, 265,\\ 
  266, 267, 268, 269, 271, 272, 274, 275, 276, 277, 278, 279, 280, 281, 282, 
  283, 284, 287, 288, 291, 292, 293, 294, 296, 298,\\ 
  303, 304, 305, 307, 309, 310, 311, 312, 315, 316, 317, 318, 320, 321, 
  324, 326, 328, 329, 331, 334, 335, 336, 340, 341, 342,\\ 
  343, 344, 345, 347, 350, 351, 352, 353, 357, 359, 360, 361, 362, 
  363, 364, 365, 368, 371, 373, 376, 377, 379, 380, 381, 383,\\ 
  384, 385, 386, 387, 389, 393, 394, 396, 397, 398, 400, 401, 402, 404, 406, 407, 411, 412, 416, 417, 
  419, 420, 423, 424, 426,\\ 
  427, 429, 433, 435, 437, 441, 442, 445, 446, 448, 449, 450, 452, 453, 454, 455, 456, 457, 459, 460, 
  462, 464, 465, 469, 470,\\ 
  473, 475, 477, 482, 487, 489, 494, 496, 497, 498, 499, 501, 504, 506, 507, 508, 509, 510, 512, 513, 
  514, 523, 528, 531, 532,\\ 
  535, 540, 543, 544, 545, 546, 548, 549, 555, 556, 557, 558, 559, 562, 566, 567, 570, 571, 574, 575, 
  576, 579, 582, 583, 585,\\ 
  586, 587, 588, 589, 592, 594, 595, 596, 597, 598, 601, 603, 605, 606, 607, 608, 609, 613, 616, 618, 
  619, 620, 622, 624, 627,\\ 
  628, 633, 635, 638, 639, 640, 641, 642, 643, 644, 645, 647, 651, 652, 653, 654, 655, 657, 658, 659, 
  660, 661, 665, 666, 667,\\ 
  668, 669, 676, 677, 682, 683, 684, 686, 689, 690, 693, 695, 700, 703, 706, 708, 709, 713, 714, 716, 
  717, 718, 720, 721, 722,\\ 
  723, 724, 733, 734, 737, 738, 750, 751, 752, 755, 758, 766, 767, 771, 772, 776, 779, 782, 783, 784, 
  787, 793, 796, 797, 799,\\ 
  802, 803, 804, 805, 806, 807, 809, 810, 812, 813, 815, 820, 821, 822, 823, 824, 825, 829, 831, 833, 
  834, 835, 842, 845, 846,\\ 
  848, 850, 853, 854, 855, 856, 857, 860, 862, 869, 872, 874, 875, 878, 881, 884, 886, 887, 893, 894, 
  895, 896, 903, 905, 913,\\ 
  915, 916, 918, 920, 921, 922, 926, 927, 928, 930, 931, 935, 948, 953, 963, 966, 967, 970, 976, 977, 
  978, 981, 982, 984, 985,\\ 
  988, 991, 996, 1000, 1002, 1005, 1008, 1009, 1010, 1013, 1014, 
  1015, 1020, 1021, 1023, 1024, 1028, 1029, 1032, 1037, 1039,\\ 
  1042, 1043, 1050, 1058, 1059, 1063, 1068, 1069, 1070, 1071, 
  1072, 1083, 1084, 1088, 1089, 1092, 1095, 1097, 1098, 1099,\\ 
  1101, 1102, 1103, 1104, 1110, 1111, 1114, 1115, 1117, 1118, 
  1121, 1122, 1123, 1125, 1130, 1133, 1136, 1139, 1141, 1142,\\ 
  1143, 1144, 1145, 1146, 1148, 1151, 1152, 1154, 1157, 1159, 
  1163, 1167, 1168, 1171, 1172, 1175, 1182, 1183, 1188, 1190,\\ 
  1196, 1201, 1202, 1203, 1206, 1207, 1208, 1209, 1211, 1212, 
  1215, 1219 1220, 1221, 1228, 1229, 1230, 1231, 1232, 1234,\\
  1238, 1241, 1242, 1245, 1246, 1255, 1268, 1273, 1279, 1282, 
  1283, 1284, 1285, 1288, 1290, 1295, 1301, 1302, 1303, 1305,\\
  1306, 1308, 1309, 1327, 1329, 1330, 1331, 1333, 1334, 1335, 
  1336, 1339, 1341, 1349, 1352, 1353, 1355, 1356, 1358, 1360,\\
  1362, 1367, 1369, 1370, 1372, 1375, 1377, 1379, 1381, 1383, 
  1384, 1386, 1389, 1393, 1398, 1403, 1404, 1406, 1407, 1412,\\
  1414, 1415, 1418, 1419, 1424, 1431, 1435, 1436, 1437, 1438, 
  1439, 1443, 1449, 1453, 1457, 1461, 1463, 1468, 1470, 1472,\\
  1479, 1481, 1485, 1492, 1494, 1506, 1509, 1510, 1511, 1512, 
  1513, 1515, 1516, 1517, 1518, 1522, 1525, 1526, 1543, 1544,\\
  1545, 1546, 1547, 1548, 1549, 1550, 1551, 1555, 1557, 1560, 
  1562, 1564, 1566, 1568, 1577, 1580, 1586, 1587, 1588, 1590,\\
  1598, 1599, 1600, 1619, 1620, 1625, 1626, 1639, 1645, 1650, 
  1654, 1656, 1657, 1658, 1668, 1669, 1672, 1676, 1695, 1696,\\
  1697, 1703, 1705, 1707, 1716, 1717, 1747, 1749, 1752, 1753, 
  1767, 1769, 1785, 1786, 1790, 1792, 1794, 1803, 1809, 1810,\\
  1811, 1819, 1822, 1831, 1844, 1845, 1854, 1859, 1874, 1876, 
  1877, 1891, 1892, 1895, 1896, 1898, 1904, 1908, 1911, 1917,\\
  1919, 1921, 1927, 1936, 1937, 1941\\\hline
\end{tabular}
}
\end{center}
\newpage
\begin{remark}
For the reader's convenience, we give a list of some groups $G=16Tm$ in Theorem \ref{thmain1}:\\ 
$16T64\simeq (C_2)^4\rtimes C_3$ with $H^1(k,{\rm Pic}\, \overline{X})\simeq (\bZ/2\bZ)^4$,\\ 
$16T3\simeq (C_2)^4$ with $H^1(k,{\rm Pic}\, \overline{X})\simeq (\bZ/2\bZ)^6$,\\ 
$16T4\simeq (C_4)^2$, 
$16T51\simeq (C_4)^2\rtimes C_2$, 
$16T63\simeq (C_4)^2\rtimes C_3$, 
$16T143\simeq (C_4)^2\rtimes C_4$, 
$16T185\simeq (C_4)^2\rtimes C_6$, 
$16T430\simeq (C_4)^2\rtimes (C_3\rtimes C_4)\simeq (C_4)^2\rtimes Q_{12}$ 
with $H^1(k,{\rm Pic}\, \overline{X})\simeq \bZ/4\bZ$
(see Example \ref{ex4.3}). 
\end{remark}
\begin{remark}
Theorem \ref{thmain1} enables us 
to obtain the group $T(k)/R$ of $R$-equivalence classes 
over a local field $k$ via 
$T(k)/R\simeq H^1(k,{\rm Pic}\,\overline{X})\simeq 
H^1(G,[J_{G/H}]^{fl})$ for norm one tori $T=R^{(1)}_{K/k}(\bG_m)$ 
with $[K:k]=16$ and $G={\rm Gal}(L/k)=16Tm$ 
(see Colliot-Th\'{e}l\`{e}ne and Sansuc \cite[Corollary 5, page 201]{CTS77}, 
Voskresenskii \cite[Section 17.2]{Vos98} and Hoshi, Kanai and Yamasaki \cite[Section 7, Application 1]{HKY22}). 
We also see that 
$H^1(k,{\rm Pic}\, \overline{X})\simeq\Sha^2_\omega(G,J_{G/H})\simeq 
{\rm Br}(X)/{\rm Br}(k)\simeq {\rm Br}_{\rm nr}(k(X)/k)/{\rm Br}(k)$ (see Section \ref{S2}). 
\end{remark}

Kunyavskii \cite{Kun84} gave a necessary and sufficient condition 
for the Hasse norm principle for $K/k$ with $[K:k]=4$ $(G=4Tm\ (1\leq m\leq 5))$. 
Drakokhrust and Platonov \cite{DP87} gave a necessary and sufficient condition 
for the Hasse norm principle for $K/k$ with $[K:k]=6$ $(G=6Tm\ (1\leq m\leq 16))$. 
Hoshi, Kanai and Yamasaki \cite[Theorem 1.18]{HKY22}, \cite[Theorem 1.3]{HKY23} gave 
a necessary and sufficient condition for the Hasse norm principle 
for $K/k$ with $[K:k]=n\leq 15$ (see also \cite[Section 1]{HKY22}). 

Ono's theorem (Theorem \ref{thOno}) claims that 
the Hasse norm principle holds for $K/k$ if and only if 
$\Sha(T)=0$ where $T=R^{(1)}_{K/k}(\bG_m)$ is the norm one torus of $K/k$. 
Because $H^1(k,{\rm Pic}\, \overline{X})=0$ implies $\Sha(T)=0$ (see Theorem \ref{thV}), 
the Hasse norm principle holds for $K/k$ when $H^1(k,{\rm Pic}\, \overline{X})=0$, 
e.g. the $1101$ cases of $G=16Tm$ with $H^1(k,{\rm Pic}\, \overline{X})=0$ as in Theorem \ref{thmain1}. 
Voskresenskii and Kunyavskii \cite{VK84} proved that 
$H^1(k,{\rm Pic}\, \overline{X})=0$ for $G\simeq S_n$ and $[G:H]=n$ 
(see also Voskresenskii \cite[Theorem 4, Corollary]{Vos88}) 
and 
Macedo \cite{Mac20} showed that 
$H^1(k,{\rm Pic}\, \overline{X})=0$ for $G\simeq A_n$ and $[G:H]=n\geq 5$.  

For the case where $n=16$, because we have $774+31+37+1+1+9=853$ cases 
with $H^1(k,{\rm Pic}\, \overline{X})\neq 0$ as in Theorem \ref{thmain1}, 
we will consider the Hasse norm principle for $K/k$ with $[K:k]=16$  
only for the maximal $H\leq G=16Tm$ cases with $[G:H]=16$, i.e. $G$ is primitive. 
Recall that a transitive subgroup $G=nTm\leq S_n$ is called {\it primitive} if the only $G$-invariant partitions are 
$\{\{i\}\mid i\in\varOmega\}$ and $\{\varOmega\}$ where $\varOmega=\{1,\ldots,n\}$. 
We see that $G=nTm$ is primitive if and only if $H\leq G=nTm$ is maximal with $[G:H]=n$. 

Among $1954$ transitive groups $G=16Tm\leq S_{16}$, only one $16T1953\simeq A_{16}$ is simple.  
Including simple $A_{16}$, we find that there exist $22$ primitive subgroups $16Tm\leq S_{16}$ 
up to conjugacy (see Hulpke \cite[Tabelle 1]{Hul96}, \cite[Table 1]{Hul05}, \cite{GAP}). 
By Theorem \ref{thmain1}, we obtain $H^1(k,{\rm Pic}\, \overline{X})$ for $22$ primitive groups $G=16Tm$ 
as in Table $1$-$2$. 
In Table $1$-$2$, 
$C_n$ is the cyclic group of order $n$, 
$A_n$ is the alternating group of degree $n$, 
$S_n$ is the symmetric group of degree $n$, 
$\GL_n(\bF_q)$ is the general linear group of degree $n$ 
over the finite field $\bF_q$ of $q$ elements. 
%
\begin{center}
\vspace*{1mm}
Table $1$-$2$: $H^1(k,{\rm Pic}\, \overline{X})\simeq H^1(G,[J_{G/H}]^{fl})$ 
for $22$ primitive $G=16Tm\leq S_{16}$ ($H\leq G$: maximal)\vspace*{2mm}\\
\begin{tabular}{lc} 
$G$ & $H^1(k,{\rm Pic}\, \overline{X})$ 
$\simeq H^1(G,[J_{G/H}]^{fl})$\\\hline
$16T178\simeq (C_2)^4\rtimes C_5$ & $(\bZ/2\bZ)^{\oplus 2}$\\
$16T415\simeq ((C_2)^4\rtimes C_5)\rtimes C_2$ & $0$\\
$16T447\simeq (C_2)^4\rtimes C_{15}$ & $0$\\
$16T708\simeq (((C_2)^4\rtimes C_3)\rtimes C_2)\rtimes C_3$ & $\bZ/2\bZ$\\
$16T711\simeq ((C_2)^4\rtimes C_5)\rtimes C_4$ & $0$\\
$16T777\simeq (((C_2)^4\rtimes C_5)\rtimes C_2)\rtimes C_3$ & $0$\\
$16T1030\simeq (A_4)^2\rtimes C_4$ & $0$\\
$16T1034\simeq ((((C_2)^4\rtimes C_3)\rtimes C_2)\rtimes C_3)\rtimes C_2$ & $0$\\
$16T1079\simeq ((C_2)^4\rtimes C_{15})\rtimes C_4$ & $0$\\
$16T1080\simeq (C_2)^4\rtimes A_5$ & $(\bZ/2\bZ)^{\oplus 2}$\\
$16T1081\simeq (C_2)^4\rtimes A_5$ & $0$\\
$16T1294\simeq (S_4)^2\rtimes C_2$ & $0$\\
$16T1328\simeq (C_2)^4\rtimes S_5$ & $0$\\
$16T1329\simeq (C_2)^4\rtimes S_5$ & $\bZ/2\bZ$\\
$16T1508\simeq (C_2)^4\rtimes \GL_2(\bF_4)$ & $0$\\
$16T1653\simeq (C_2)^4\rtimes (A_5\rtimes S_3)$ & $0$\\
$16T1654\simeq (C_2)^4\rtimes A_6$ & $\bZ/2\bZ$\\
$16T1753\simeq (C_2)^4\rtimes S_6$ & $\bZ/2\bZ$\\
$16T1840\simeq (C_2)^4\rtimes A_7$ & $0$\\
$16T1906\simeq (C_2)^4\rtimes A_8\simeq (C_2)^4\rtimes \GL_4(\bF_2)$ & $0$\\
$16T1953\simeq A_{16}$ & $0$\\
$16T1954\simeq S_{16}$ & $0$\\\hline
\end{tabular}
\end{center}\vspace*{2mm}

By using Drakokhrust and Platonov's method 
(see Section \ref{S5} for details) 
and GAP functions given in \cite[Section 6]{HKY22}, \cite[Section 6]{HKY23} 
with the new function {\tt StemExtensions} of GAP \cite{GAP} provided in 
Section \ref{S7}, 
we will prove the second main theorem 
which gives a necessary and sufficient condition 
for the Hasse norm principle for $K/k$, 
i.e. a computation of $\Sha(T)$ via Ono's theorem (Theorem \ref{thOno}) where $T=R^{(1)}_{K/k}(\bG_m)$, 
with $[K:k]=16$ and primitive 
$G={\rm Gal}(L/k)=16Tm$ as in Table $1$-$2$. 
For the definitions of $\Sha(T)$ and $A(T)$, see Section \ref{S3}.  
%
\begin{theorem}\label{thmain2}
Let $k$ be a number field, 
$K/k$ be a field extension of degree $16$ 
and $L/k$ be the Galois closure of $K/k$. 
Let $G={\rm Gal}(L/k)=16Tm$ be a transitive subgroup of $S_{16}$, 
$H={\rm Gal}(L/K)$ with $[G:H]=16$ 
and $G_v$ be the decomposition group of $G$ at a place $v$ of $k$. 
Let $T=R^{(1)}_{K/k}(\bG_m)$ be the norm one torus of $K/k$ 
of dimension $15$ and $X$ be a smooth $k$-compactification of $T$. 
Assume that $G=16Tm\leq S_{16}$ is one of the $22$ primitive groups, 
i.e. $H\leq G$ is maximal.  
Then $A(T)\simeq \Sha(T)=0$ except for the $6$ cases $G=16Tm$ 
with $H^1(k,{\rm Pic}\, \overline{X})\neq 0$ as in Table $1$-$2$. 
For the cases 
$G=16T708$, $16T1329$, $16T1654$, $16T1753$ with 
$H^1(k,{\rm Pic}\, \overline{X})\simeq \bZ/2\bZ$,  
we have either {\rm (a)} 
$A(T)=0$ and $\Sha(T)\simeq\bZ/2\bZ$ 
or {\rm (b)} 
$A(T)\simeq\bZ/2\bZ$ and $\Sha(T)=0$. 
We assume that $H$ is the stabilizer of one of the letters in $G$. 
Then $\Sha(T)$ for 
the $6$ primitive cases $G=16Tm$ with 
$H^1(k,{\rm Pic}\, \overline{X})\neq 0$ 
is given as in Table $2$-$1$ and Table $2$-$2$. 
Moreover, for the case $G=16T708$, 
$\Sha(T)$ is also given for general $H\leq G$ with $[G:H]=16$. 
\end{theorem}

In Table $2$-$1$ and Table $2$-$2$, 
$N_G(G^\prime)$ is the normalizer of 
a subgroup $G^\prime\leq G$, 
$D(G)$ is the derived subgroup (commutator subgroup) of $G$, 
${\rm Orb}_{G}(i)$ is the orbit of $1\leq i\leq 16$ 
under the action of $G\leq S_{16}$, 
${\rm Syl}_p(G)$ is a $p$-Sylow subgroup of $G$ 
where $p$ is a prime number, 
$G_{16,3}:=\langle x,y,z\mid x^4=y^2=z^2=1, xy=yx, yz=zy, zxz^{-1}=xy\rangle\simeq (C_4\times C_2)\rtimes C_2$ 
is the group of order $16$ with SmallGroup ID $[16,3]$ in GAP \cite{GAP}. 
Note that we should distinguish $G_{16,3}$ from 
$G_{16,13}:=\langle x,y,z\mid x^4=y^2=z^4=1, x^2=z^2, yxy=x^{-1}, xz=zx, yz=zy\rangle\simeq (C_4\times C_2)\rtimes C_2$  
with SmallGroup ID $[16,13]$ in GAP 
which is the central product of $D_4\simeq\langle x,y\rangle$ and $C_4\simeq\langle z\rangle$ 
over $C_2\simeq\langle x^2\rangle=\langle z^2\rangle$. 

Note also that a place $v$ of $k$ with non-cyclic decomposition group $G_v$ 
as in Table $2$-$1$ and Table $2$-$2$ must be ramified in $L$ because 
if $v$ is unramified, then $G_v$ is cyclic. 

\newpage
\begin{center} 
Table 2-1: $\Sha(T)=0$ for $T=R^{(1)}_{K/k}(\bG_m)$ and primitive groups $G={\rm Gal}(L/k)=16Tm$\\
\hspace*{-12mm}
with $H^1(k,{\rm Pic}\, \overline{X})\simeq \bZ/2\bZ$ as in Table $1$-$1$ and Table $1$-$2$\vspace*{2mm}\\
\renewcommand{\arraystretch}{1.05}
\begin{tabular}{ll} 
 & $\Sha(T)\leq \bZ/2\bZ$, and\\
$G$ & $\Sha(T)=0$ if and only if there exists a place $v$ of $k$ such that\\\hline\vspace*{-2mm}\\
$16T708\simeq (((C_2)^4\rtimes C_3)\rtimes C_2)\rtimes C_3$ & \multirow{3}{*}{
\begin{minipage}{11cm} 
{\rm (i)} $V_4\leq G_v$ where 
{\rm (i-1)} $[G:N_G(V_4)]=18$ and $V_4 \subset D(G)$ with $|D(G)|=48$, or 
{\rm (i-2)} $[G:N_G(V_4)]=2$, \\
{\rm (ii)} $(C_2)^3\leq G_v$, or\\
{\rm (iii)} $(C_4\times C_2)\leq G_v$
\end{minipage}}\\
 & \\
 & \\
 & \vspace*{1mm}\\\hline\vspace*{-2mm}\\
$16T1329\simeq (C_2)^4\rtimes S_5$ & \multirow{3}{*}{
\begin{minipage}{11cm} 
{\rm (i)} $V_4\leq G_v$ where 
$[G:N_G(V_4)]=40$, \\
{\rm (ii)} $D_4\leq G_v$ where 
$|{\rm Orb}_{D_4}(i)|=4$ for $12$ of $1\leq i\leq 16$ and 
$|{\rm Orb}_{D_4}(j)|=2$ for $4$ of $1\leq j\leq 16$, or\\
{\rm (iii)} $C_4\times C_4\leq G_v$ where
$[G:N_G(C_4\times C_4)]=10$
\end{minipage}}\\
 & \\
 & \\
 & \vspace*{1mm}\\\hline\vspace*{-2mm}\\
$16T1654\simeq (C_2)^4\rtimes A_6$ & \multirow{3}{*}{
\begin{minipage}{11cm} 
{\rm (i)} $V_4\leq G_v$ where 
$|N_{N_{S_{16}}(G)}(V_4)|=96$ and
${\rm Syl}_2(N_{N_{S_{16}}(G)}(V_4))=(C_2)^5$, \\
{\rm (ii)} $D_4\leq G_v$ where 
$|{\rm Orb}_{D_4}(i)|=4$ for $12$ of $1\leq i\leq 16$ and 
$|{\rm Orb}_{D_4}(j)|=2$ for $4$ of $1\leq j\leq 16$, or\\
{\rm (iii)} $C_4\times C_4\leq G_v$ where
$[G:N_G(C_4\times C_4)]=30$
\end{minipage}}\\
 & \\
 & \\
 & \vspace*{1mm}\\\hline\vspace*{-2mm}\\
$16T1753\simeq (C_2)^4\rtimes S_6$ & \multirow{3}{*}{
\begin{minipage}{11cm} 
{\rm (i)} $V_4\leq G_v$ where
{\rm (i-1)}
$|{\rm Orb}_{V_4}(i)|=2$ for $1\leq i\leq 16$ and
$|N_{N_{S_{16}}(G)}(V_4)|=64$, or 
{\rm (i-2)}
$|{\rm Orb}_{V_4}(i)|=2$ for $12$ of $1\leq i\leq 16$ and
$|{\rm Orb}_{V_4}(j)|=4$ for $4$ of $1\leq j\leq 16$ and
$|N_{N_{S_{16}}(G)}(V_4)|=96$, \\
{\rm (ii)} $C_4\times C_2\leq G_v$ where
$|N_{N_{S_{16}}(G)}(C_4\times C_2)|=32$ and
{\rm (ii-1)} 
$|{\rm Orb}_{C_4\times C_2}(i)|=4$ for $12$ of $1\leq i\leq 16$ and
$|{\rm Orb}_{C_4 \times C_2}(j)|=2$ for $4$ of $1\leq j\leq 16$, or 
{\rm (ii-2)}
$|{\rm Orb}_{V_4}(i)|=8$ for $1\leq i\leq 16$, \\
{\rm (iii)} $D_4\leq G_v$ where
$|{\rm Orb}_{D_4}(i)|=4$ for $12$ of $1\leq i\leq 16$ and
$|{\rm Orb}_{D_4}(j)|=2$ for $4$ of $1\leq j\leq 16$, \\
{\rm (iv)} $C_4\times C_4\leq G_v$ where
$|G:N_G(C_4\times C_4)|=30$, \\
{\rm (v)} $G_{16,3}\leq G_v$ where
$|G:N_G(G_{16,3})|=180$ and
$|G_{16,3} \cap G_{16,3}^\prime|=1$
for 160 of 180 conjugate subgroups $G_{16,3}^\prime$ of $G_{16,3}$, \\
{\rm (vi)} $C_8\times C_2\leq G_v$, or\\
{\rm (vii)} $C_2\times Q_8\leq G_v$ where
$|G \cap E_{16}|=2$ ($E_{16} \simeq (C_2)^4$ is the unique characteristic subgroup
of $G$ of order $16$)
\end{minipage}}\\
 & \\
 & \\
 & \\
 & \\
 & \\
 & \\
 & \\
 & \\
 & \\
 & \\
 & \\
 & \vspace*{4mm}\\\hline
\end{tabular}
\end{center}~\\

\newpage
\begin{center}
Table 2-2: $\Sha(T)$ for $T=R^{(1)}_{K/k}(\bG_m)$ 
and primitive groups $G={\rm Gal}(L/k)=16Tm$\\
\hspace*{5mm}
with $H^1(k,{\rm Pic}\, \overline{X})\simeq (\bZ/2\bZ)^{\oplus 2}$ 
as in Theorem \ref{thmain1} and Table $1$-$2$\vspace*{2mm}\\
\renewcommand{\arraystretch}{1.05}
\begin{tabular}{ll} 
$G$ & $\Sha(T)\leq (\bZ/2\bZ)^{\oplus 2}$\\\hline\vspace*{-2mm}\\
$16T178\simeq (C_2)^4\rtimes C_5$ & \multirow{8}{*}{
\begin{minipage}{12cm}
~\vspace*{-1mm}\\
{\rm (I)} $\Sha(T)=0$ if and only if\\ 
{\rm (I-i)} 
there exists a place $v$ of $k$ such that 
$(C_2)^3\leq G_v$, or\\
{\rm (I-ii)} 
there exist places $v_1,v_2$ of $k$ such that 
$V_4\simeq V_4^{(i)}\leq G_{v_1}$, $V_4\simeq V_4^{(j)}\leq G_{v_2}$ 
$(1\leq i<j\leq 7)$ with $(i,j)\in\{(1,3)$, $(1,4)$, $(1,5)$, $(1,6)$, $(2,3)$, 
$(2,4)$, $(2,5)$, $(2,6)$, $(3,5)$, $(3,6)$, $(4,5)$, $(4,6)\}$;\\
{\rm (II)} $\Sha(T)\simeq \bZ/2\bZ$ if and only if 
there exists a place $v$ of $k$ such that 
$V_4^{(i)}\leq G_v$ $(1\leq i\leq 6)$
and {\rm (I)} does not hold\\
where there exist $7$ non-conjugate $V_4\simeq V_4^{(i)}\leq G$ $(1\leq i\leq 7)$ 
which can be distinguished by the following subtable:\vspace*{2mm}\\
Subtable: the number of the $t_i$'s with ${\rm Orb}_{V_4}{(1)}=\{1,t_1,t_2,t_3\}$ 
which belong to ${\rm Orb}_H(s_j)$ with 
$\{1,\ldots,16\}={\rm Orb}_H(1)$ $\cup$ ${\rm Orb}_H(s_1)$ $\cup$ 
${\rm Orb}_H(s_2)$ $\cup$ ${\rm Orb}_H(s_3)$ 
where $H={\rm Stab}_1(G)\simeq C_5$ with  
${\rm Orb}_H(1)=\{1\}$ and $|{\rm Orb}_H(s_j)|=5$ $(1\leq j\leq 3)$.\vspace*{2mm}\\
\begin{tabular}{cccc}
 & ${\rm Orb}_H(s_1)$ & ${\rm Orb}_H(s_2)$ & ${\rm Orb}_H(s_3)$\\\hline
$V_4^{(1)}$ & $1$ & $0$ & $2$\\
$V_4^{(2)}$ & $1$ & $2$ & $0$\\
$V_4^{(3)}$ & $2$ & $0$ & $1$\\
$V_4^{(4)}$ & $0$ & $2$ & $1$\\
$V_4^{(5)}$ & $2$ & $1$ & $0$\\
$V_4^{(6)}$ & $0$ & $1$ & $2$\\
$V_4^{(7)}$ & $1$ & $1$ & $1$\\
\end{tabular}
\end{minipage}}\\
 & \\
 & \\
 & \\
 & \\
 & \\
 & \\
 & \\
 & \\
 & \\
 & \\
 & \\
 & \\
 & \\
 & \\
 & \\
 & \\
 & \\
 & \\
 & \\
 & \\
 & \\
 & \vspace*{1mm}\\\hline\\~\vspace*{-4mm}
$16T1080\simeq (C_2)^4\rtimes A_5$ & \multirow{8}{*}{
\begin{minipage}{12cm}
{\rm (I)} $\Sha(T)=0$ if and only if\\ 
{\rm (I-i)} 
there exists a place $v$ of $k$ such that 
$(C_2)^3\leq G_v$ where 
$(C_2)^3\cap E_{16}=C_2$ for the unique characteristic subgroup 
$E_{16}\simeq (C_2)^4\lhd G$ of order $16$, or\\
{\rm (I-ii)} 
there exist places $v_1,v_2$ of $k$ such that 
$V_4\simeq V_4^{(i)}\leq G_{v_1}$, $V_4\simeq V_4^{(j)}\leq G_{v_2}$ 
$(1\leq i<j\leq 3)$ 
where there exist $3$ non-conjugate $V_4\simeq V_4^{(i)}\leq G$ 
$(1\leq i\leq 3)$ which satisfy 
$|{\rm Orb}_{V_4^{(i)}}(s_1)|=2$ for $12$ of $1\leq s_1\leq 16$ 
and $|{\rm Orb}_{V_4^{(i)}}(s_2)|=4$ for $4$ of $1\leq s_2\leq 16$, or\\ 
{\rm (I-iii)} there exist places $v_1,v_2$ of $k$ such that 
$(C_4)^2\simeq ((C_4)^2)^{(i)}\leq G_{v_1}$, 
$(C_4)^2\simeq ((C_4)^2)^{(j)}\leq G_{v_2}$ 
$(1\leq i<j\leq 3)$ 
where there exist $3$ non-conjugate 
$(C_4)^2\simeq ((C_4)^2)^{(i)}\leq G$ $(1\leq i\leq 3)$, or\\
{\rm (I-iv)} 
there exist places $v_1,v_2$ of $k$ such that 
$V_4\simeq V_4^{(i)}\leq G_{v_1}$ $(1\leq i\leq 3)$ 
where there exist $3$ non-conjugate $V_4\simeq V_4^{(i)}\leq G$ 
$(1\leq i\leq 3)$ which satisfy 
$|{\rm Orb}_{V_4^{(i)}}(s_1)|=2$ for $12$ of $1\leq s_1\leq 16$ 
and $|{\rm Orb}_{V_4^{(i)}}(s_2)|=4$ for $4$ of $1\leq s_2\leq 16$, 
$(C_4)^2\simeq ((C_4)^2)^{(j)}\leq G_{v_2}$ $(1\leq j\leq 3)$ 
where there exist $3$ non-conjugate 
$(C_4)^2\simeq ((C_4)^2)^{(j)}\leq G$ $(1\leq j\leq 3)$ 
and $|N_{N_{S_{16}}}(x)\cap N_{N_{S_{16}}}(y)|=12$ or $48$ 
for $x\in V_4^{(i)}$ and $y\in ((C_4)^2)^{(j)}$ 
(the last condition is satisfied by $6$ $(i,j)$'s out of $9$);\\
{\rm (II)} $\Sha(T)\simeq \bZ/2\bZ$ if and only if 
there exists a place $v$ of $k$ such that 
$V_4\leq G_v$ 
which satisfies 
$|{\rm Orb}_{V_4}(s_1)|=2$ for $12$ of $1\leq s_1\leq 16$ 
and $|{\rm Orb}_{V_4}(s_2)|=4$ for $4$ of $1\leq s_2\leq 16$, 
or $(C_4)^2\leq G_v$, 
and {\rm (I)} does not hold
\end{minipage}}\\
 & \\
 & \\
 & \\
 & \\
 & \\
 & \\
 & \\
 & \\
 & \\
 & \\
 & \\
 & \\
 & \\
 & \\
 & \\
 & \\
 & \\
 & \\
 & \vspace*{4mm}\\\hline
\end{tabular}
\end{center}~\\
\newpage
As a consequence of Theorem \ref{thmain2}, 
we get the 
Tamagawa number $\tau(T)$ of 
$T=R^{(1)}_{k/k}(\bG_m)$ 
over a number field $k$ 
via Ono's formula $\tau(T)=|H^1(G,J_{G/H})|/|\Sha(T)|$ 
where $J_{G/H}\simeq \widehat{T}={\rm Hom}(T,\bG_m)$ 
(see Ono \cite[Main theorem, page 68]{Ono63}, \cite{Ono65}, 
Voskresenskii \cite[Theorem 2, page 146]{Vos98} and 
Hoshi, Kanai and Yamasaki \cite[Section 8, Application 2]{HKY22}). 
%
%
\begin{corollary}\label{cor1.5}
Let the notation be as in Theorem \ref{thmain2}. 
Then the Tamagawa number 
$\tau(T)=1/|\Sha(T)|=1$, $1/2$ or $1/4$ where $\Sha(T)$ is given as in 
Table $2$-$1$ and Table $2$-$2$. 
\end{corollary}
%

We organize this paper as follows. 
In Section \ref{S2}, 
we prepare basic definitions and known results 
about algebraic $k$-tori and norm one tori.  
In particular, 
we recall the definition of a flabby resolution of a $G$-lattice  
which is a basic tool to investigate algebraic $k$-tori. 
In Section \ref{S3}, 
we recall some known results about the Hasse norm principle. 
In Section \ref{S4}, we give the proof of Theorem \ref{thmain1}. 
In Section \ref{S5}, we recall 
Drakokhrust and Platonov's method for 
the Hasse norm principle for $K/k$ 
and results in Hoshi, Kanai and Yamasaki \cite[Section 6]{HKY22}. 
In Section \ref{S6}, we prove Theorem \ref{thmain2} 
by using Drakokhrust and Platonov's method and 
some new useful functions of GAP \cite{GAP}. 
A proof of Corollary \ref{cor1.5} is also given. 
In Section \ref{S7}, 
GAP algorithms are given 
which are also available 
as in \cite{Norm1ToriHNP}. 
%
\begin{acknowledgment}
We thank the referee for very careful reading of the manuscript 
and also pointing out an erratum in the last line of Theorem \ref{thDP2}. 
\end{acknowledgment}

\section{Rationality problem for algebraic tori and norm one tori}\label{S2}

Let $k$ be a field and $K$ 
be a finitely generated field extension of $k$. 
A field $K$ is called {\it rational over $k$} 
(or {\it $k$-rational} for short) 
if $K$ is purely transcendental over $k$, 
i.e. $K$ is isomorphic to $k(x_1,\ldots,x_n)$, 
the rational function field over $k$ with $n$ variables $x_1,\ldots,x_n$ 
for some integer $n$. 
$K$ is called {\it stably $k$-rational} 
if $K(y_1,\ldots,y_m)$ is $k$-rational for some algebraically 
independent elements $y_1,\ldots,y_m$ over $K$. 
Two fields 
$K$ and $K^\prime$ are called {\it stably $k$-isomorphic} if 
$K(y_1,\ldots,y_m)\simeq K^\prime(z_1,\ldots,z_n)$ over $k$ 
for some algebraically independent elements $y_1,\ldots,y_m$ over $K$ 
and $z_1,\ldots,z_n$ over $K^\prime$. 
When $k$ is an infinite field, 
$K$ is called {\it retract $k$-rational} 
if there is a $k$-algebra $R$ contained in $K$ such that 
(i) $K$ is the quotient field of $R$, and (ii) 
the identity map $1_R : R\rightarrow R$ factors through a localized 
polynomial ring over $k$, i.e. there is an element $f\in k[x_1,\ldots,x_n]$, 
which is the polynomial ring over $k$, and there are $k$-algebra 
homomorphisms $\varphi : R\rightarrow k[x_1,\ldots,x_n][1/f]$ 
and $\psi : k[x_1,\ldots,x_n][1/f]\rightarrow R$ satisfying 
$\psi\circ\varphi=1_R$ (cf. Saltman \cite{Sal84}). 
$K$ is called {\it $k$-unirational} 
if $k\subset K\subset k(x_1,\ldots,x_n)$ for some integer $n$. 
It is not difficult to see that
\begin{center}
$k$-rational \ \ $\Rightarrow$\ \ 
stably $k$-rational\ \ $\Rightarrow$ \ \ 
retract $k$-rational\ \ $\Rightarrow$ \ \ 
$k$-unirational.
\end{center} 

Let $\overline{k}$ be a fixed separable closure of the base field $k$. 
Let $T$ be an algebraic $k$-torus, 
i.e. a group $k$-scheme with fiber product (base change) 
$T\times_k \overline{k}=
T\times_{{\rm Spec}\, k}\,{\rm Spec}\, \overline{k}
\simeq (\bG_{m,\overline{k}})^n$; 
$k$-form of the split torus $(\bG_m)^n$. 
Then there exists a finite Galois extension $K/k$ 
with Galois group $G={\rm Gal}(K/k)$ such that 
$T$ splits over $K$: $T\times_k K\simeq (\bG_{m,K})^n$. 
It is also well-known that 
there is the duality between the category of $G$-lattices, 
i.e. finitely generated $\bZ[G]$-modules which are $\bZ$-free 
as abelian groups, 
and the category of algebraic $k$-tori which split over $K$ 
(see Ono \cite[Section 1.2]{Ono61}, 
Voskresenskii \cite[page 27, Example 6]{Vos98} and 
Knus, Merkurjev, Rost and Tignol \cite[page 333, Proposition 20.17]{KMRT98}). 
Indeed, if $T$ is an algebraic $k$-torus, then the character 
module $\widehat{T}={\rm Hom}(T,\bG_m)$ of $T$ 
may be regarded as a $G$-lattice. 
Let $X$ be a smooth $k$-compactification of $T$, 
i.e. smooth projective $k$-variety $X$ 
containing $T$ as a dense open subvariety, 
and $\overline{X}=X\times_k\overline{k}$. 
There exists such a smooth $k$-compactification of an algebraic $k$-torus $T$ 
over any field $k$ (due to Hironaka \cite{Hir64} for ${\rm char}\, k=0$, 
see Colliot-Th\'{e}l\`{e}ne, Harari and Skorobogatov 
\cite[Corollaire 1]{CTHS05} for any field $k$). 

An algebraic $k$-torus 
$T$ is said to be {\it $k$-rational} (resp. {\it stably $k$-rational}, 
{\it retract $k$-rational}, {\it $k$-unirational}) 
if the function field $k(T)$ of $T$ over $k$ is $k$-rational 
(resp. stably $k$-rational, retract $k$-rational, $k$-unirational). 
Two algebraic $k$-tori $T$ and $T^\prime$ 
are said to be 
{\it birationally $k$-equivalent $(k$-isomorphic$)$} 
if their function fields $k(T)$ and $k(T^\prime)$ are 
$k$-isomorphic. 
For an equivalent definition in the language of algebraic geometry, 
see e.g. 
Manin \cite{Man86}, 
Manin and Tsfasman, \cite{MT86}, 
Colliot-Th\'{e}l\`{e}ne and Sansuc \cite[Section 1]{CTS07}, 
Beauville \cite{Bea16}, 
Merkurjev \cite[Section 3]{Mer17}. 

Let $L/k$ be a finite Galois extension with Galois group $G={\rm Gal}(L/k)$. 
Let $M=\bigoplus_{1\leq i\leq n}\bZ\cdot u_i$ be a $G$-lattice with 
a $\bZ$-basis $\{u_1,\ldots,u_n\}$, 
i.e. finitely generated $\bZ[G]$-module 
which is $\bZ$-free as an abelian group. 
Let $G$ act on the rational function field $L(x_1,\ldots,x_n)$ 
over $L$ with $n$ variables $x_1,\ldots,x_n$ by 
\begin{align*}
\sigma(x_i)=\prod_{j=1}^n x_j^{a_{i,j}},\quad 1\leq i\leq n
\end{align*}
for any $\sigma\in G$, when $\sigma (u_i)=\sum_{j=1}^n a_{i,j} u_j$, 
$a_{i,j}\in\bZ$. 
The field $L(x_1,\ldots,x_n)$ with this action of $G$ will be denoted 
by $L(M)$.
There is the duality between the category of $G$-lattices 
and the category of algebraic $k$-tori which split over $L$ 
(see \cite[Section 1.2]{Ono61}, \cite[page 27, Example 6]{Vos98}). 
In fact, if $T$ is an algebraic $k$-torus, then the character 
group $\widehat{T}={\rm Hom}(T,\bG_m)$ of $T$ 
may be regarded as a $G$-lattice. 
Conversely, for a given $G$-lattice $M$, there exists an algebraic 
$k$-torus $T={\rm Spec}(L[M]^G)$ which splits over $L$ 
such that $\widehat{T}\simeq M$ as $G$-lattices 
where $L[M]^G$ is the invariant ring of $L[M]=L[x_1,\ldots,x_n]$ under the action of $G$. 

The invariant field $L(M)^G$ of $L(M)$ under the action of $G$ 
may be identified with the function field $k(T)$ 
of the algebraic $k$-torus $T$ where $\widehat{T}\simeq M$. 
Note that the field $L(M)^G$ is always $k$-unirational 
(see \cite[page 40, Example 21]{Vos98}). 
Isomorphism classes of tori of dimension $n$ over $k$ correspond bijectively 
to the elements of the set $H^1(k,\GL_n(\bZ)):=H^1(\mathcal{G},\GL_n(\bZ))$ 
where $\mathcal{G}={\rm Gal}(\overline{k}/k)$ since 
${\rm Aut}(\bG_m^n)=\GL_n(\bZ)$. 
The $k$-torus $T$ of dimension $n$ is determined uniquely by the integral 
representation $h : \mathcal{G}\rightarrow \GL_n(\bZ)$ up to conjugacy, 
and the group $G=h(\mathcal{G})$ is a finite subgroup of $\GL_n(\bZ)$ 
(see \cite[page 57, Section 4.9]{Vos98})). 
The minimal splitting field $L$ of $T$ corresponds to 
$\mathcal{H}$ where 
$\rho: \mathcal{G}\rightarrow G$ and $\mathcal{H}={\rm Ker}(\rho)$ 
with $G\simeq \mathcal{G}/\mathcal{H}$.

A $G$-lattice $M$ is called {\it permutation} $G$-lattice 
if $M$ has a $\bZ$-basis permuted by $G$, 
i.e. $M\simeq \oplus_{1\leq i\leq m}\bZ[G/H_i]$ 
for some subgroups $H_1,\ldots,H_m$ of $G$. 
$M$ is called {\it stably permutation} 
$G$-lattice if $M\oplus P\simeq P^\prime$ 
for some permutation $G$-lattices $P$ and $P^\prime$. 
$M$ is called {\it invertible} 
if it is a direct summand of a permutation $G$-lattice, 
i.e. $P\simeq M\oplus M^\prime$ for some permutation $G$-lattice 
$P$ and a $G$-lattice $M^\prime$. 
$M$ is called {\it flabby} (or {\it flasque}) if $\widehat H^{-1}(H,M)=0$ 
for any subgroup $H$ of $G$ where $\widehat H$ is the Tate cohomology. 
$M$ is called {\it coflabby} (or {\it coflasque}) if $H^1(H,M)=0$
for any subgroup $H$ of $G$. 
We say that two $G$-lattices $M_1$ and $M_2$ are {\it similar} 
if there exist permutation $G$-lattices $P_1$ and $P_2$ such that 
$M_1\oplus P_1\simeq M_2\oplus P_2$. 
The set of similarity classes becomes a commutative monoid 
with respect to the sum $[M_1]+[M_2]:=[M_1\oplus M_2]$ 
and the zero $0=[P]$ where $P$ is a permutation $G$-lattice. 
For a $G$-lattice $M$, there exists a short exact sequence of $G$-lattices 
$0 \rightarrow M \rightarrow P \rightarrow F \rightarrow 0$
where $P$ is permutation and $F$ is flabby which is called a 
{\it flabby resolution} of $M$ 
(see Endo and Miyata \cite[Lemma 1.1]{EM75}, 
Colliot-Th\'el\`ene and Sansuc \cite[Lemma 3]{CTS77}, 
Manin \cite[Appendix, page 286]{Man86}). 
The similarity class $[F]$ of $F$ is determined uniquely 
and is called {\it the flabby class} of $M$. 
We denote the flabby class $[F]$ of $M$ by $[M]^{fl}$. 
We say that $[M]^{fl}$ is invertible if $[M]^{fl}=[E]$ for some 
invertible $G$-lattice $E$. 
For a $G$-lattice $M$, 
it is not difficult to see that 
\begin{align*}
\textrm{permutation}\ \ 
\Rightarrow\ \ 
&\textrm{stably\ permutation}\ \ 
\Rightarrow\ \ 
\textrm{invertible}\ \ 
\Rightarrow\ \ 
\textrm{flabby\ and\ coflabby}\\
&\hspace*{8mm}\Downarrow\hspace*{34mm} \Downarrow\\
&\hspace*{7mm}[M]^{fl}=0\hspace*{10mm}\Rightarrow\hspace*{5mm}[M]^{fl}\ 
\textrm{is\ invertible}.
\end{align*}
The above implications in each step cannot be reversed 
(see, for example, \cite[Section 1]{HY17}). 
For some basic facts of flabby (flasque) $G$-lattices, 
see \cite{CTS77}, \cite{Swa83}, \cite[Chapter 2]{Vos98}, \cite[Chapter 2]{Lor05}, \cite{Swa10}. 
%

The flabby class $[M]^{fl}=[\widehat{T}]^{fl}$ 
plays a crucial role in the rationality problem 
for $L(M)^G\simeq k(T)$ 
as follows (see 
Colliot-Th\'el\`ene and Sansuc \cite[Section 2]{CTS77}, 
\cite[Proposition 7.4]{CTS87}, 
Voskresenskii \cite[Section 4.6]{Vos98}, 
Kunyavskii \cite[Theorem 1.7]{Kun07}, 
Colliot-Th\'el\`ene \cite[Theorem 5.4]{CT07}, 
Hoshi and Yamasaki \cite[Section 1]{HY17}, 
see also e.g. Swan \cite{Swa83}, 
Kunyavskii \cite[Section 2]{Kun90}, 
Lemire, Popov and Reichstein \cite[Section 2]{LPR06}, 
Kang \cite{Kan12}, Yamasaki \cite{Yam12}). 

%
\begin{theorem}[{Voskresenskii \cite[Section 4, page 1213]{Vos69}, \cite[Section 3, page 7]{Vos70}, see also 
\cite{Vos74}, \cite[Section 4.6]{Vos98}, Kunyavskii \cite[Theorem 1.9]{Kun07} and 
Colliot-Th\'el\`ene \cite[Theorem 5.1, page 19]{CT07} for any field $k$}]\label{thVos69}
Let $k$ be a field 
and $\mathcal{G}={\rm Gal}(\overline{k}/k)$. 
Let $T$ be an algebraic $k$-torus, 
$X$ be a smooth $k$-compactification of $T$ 
and $\overline{X}=X\times_k\overline{k}$. 
Then there exists an exact sequence of $\mathcal{G}$-lattices 
\begin{align*}
0\to \widehat{T}\to \widehat{Q}\to {\rm Pic}\,\overline{X}\to 0
\end{align*}
where $\widehat{Q}$ is permutation 
and ${\rm Pic}\ \overline{X}$ is flabby. 
\end{theorem}
We have 
$H^1(k,{\rm Pic}\,\overline{X})\simeq H^1(G,{\rm Pic}\, X_K)$ 
where $K$ is the splitting field of $T$, $G={\rm Gal}(K/k)$ and 
$X_K=X\times_k K$. 
Hence Theorem \ref{thVos69} says that 
for $G$-lattices $M=\widehat{T}$ and $P=\widehat{Q}$, 
the exact sequence 
$0\to M\to P\to {\rm Pic}\, X_K\to 0$ 
gives a flabby resolution of 
$M$ and the flabby class of $M$ is 
$[M]^{fl}=[{\rm Pic}\ X_K]$ as $G$-lattices. 

%
\begin{theorem}
\label{th2-1}
Let $L/k$ be a finite Galois extension with Galois group $G={\rm Gal}(L/k)$ 
and $M$ and $M^\prime$ be $G$-lattices. 
Let $T$ and $T^\prime$ be algebraic $k$-tori with $\widehat{T}\simeq M$ 
and $\widehat{T}^\prime\simeq M^\prime$, 
i.e. $L(M)^G\simeq k(T)$ and $L(M^\prime)^G\simeq k(T^\prime)$.\\
{\rm (i)} $(${\rm Endo and Miyata} \cite[Theorem 1.6]{EM73}$)$ 
$[M]^{fl}=0$ if and only if $k(T)$ is stably $k$-rational.\\
{\rm (ii)} $(${\rm Voskresenskii} \cite[Theorem 2]{Vos74}$)$ 
$[M]^{fl}=[M^\prime]^{fl}$ if and only if $k(T)$ and $k(T^\prime)$ 
are stably $k$-isomorphic.\\
{\rm (iii)} $(${\rm Saltman} \cite[Theorem 3.14]{Sal84}$)$ 
$[M]^{fl}$ is invertible if and only if $k(T)$ is 
retract $k$-rational.
\end{theorem}

Let $G$ be a finite group and $M$ be a $G$-lattice. We define 
\begin{align*}
\Sha^i_\omega(G,M):={\rm Ker}\left\{H^i(G,M)\xrightarrow{{\rm res}}\bigoplus_{g\in G}H^i(\langle g\rangle,M)\right\}\quad (i\geq 1) .
\end{align*}
The following is a theorem of Colliot-Th\'{e}l\`{e}ne and Sansuc \cite{CTS87}: 
\begin{theorem}[{Colliot-Th\'{e}l\`{e}ne and Sansuc \cite[Proposition 9.5 (ii)]{CTS87}, see also \cite[Proposition 9.8]{San81} and \cite[page 98]{Vos98}}]\label{thCTS87}
Let $k$ be a field 
with ${\rm char}\, k=0$
and $K/k$ be a finite Galois extension 
with Galois group $G={\rm Gal}(K/k)$. 
Let $T$ be an algebraic $k$-torus which splits over $K$ and 
$X$ be a smooth $k$-compactification of $T$. 
Then we have 
\begin{align*}
\Sha^2_\omega(G,\widehat{T})\simeq 
H^1(G,{\rm Pic}\, X_K)\simeq {\rm Br}(X)/{\rm Br}(k)
\end{align*}
where 
${\rm Br}(X)$ is the \'etale cohomological Brauer Group of $X$ 
$($it is the same as the Azumaya-Brauer group of $X$ 
for such $X$, see \cite[page 199]{CTS87}$)$. 
\end{theorem}

In other words, for the $G$-lattice $M=\widehat{T}$, 
we have 
$H^1(k,{\rm Pic}\, \overline{X})\simeq H^1(G,{\rm Pic}\, X_K)\simeq 
H^1(G,[M]^{fl})\simeq \Sha^2_\omega(G,M)\simeq {\rm Br}(X)/{\rm Br}(k)$. 
We also see  
${\rm Br}_{\rm nr}(k(X)/k)={\rm Br}(X)\subset {\rm Br}(k(X))$ 
(see Saltman \cite[Proposition 10.5]{Sal99}, 
Colliot-Th\'{e}l\`{e}ne \cite[Theorem 5.11]{CTS07}, 
Colliot-Th\'{e}l\`{e}ne and Skorobogatov 
\cite[Proposition 6.2.7]{CTS21}).

Let $T=R^{(1)}_{K/k}(\bG_m)$ be the norm one torus of $K/k$,
i.e. the kernel of the norm map $R_{K/k}(\bG_m)\rightarrow \bG_m$ where 
$R_{K/k}$ is the Weil restriction (see \cite[page 37, Section 3.12]{Vos98}). 
It is biregularly isomorphic to the norm hypersurface 
$f(x_1,\ldots,x_n)=1$ where 
$f\in k[x_1,\ldots,x_n]$ is the polynomial of total 
degree $n$ defined by the norm map $N_{K/k}:K^\times\to k^\times$ 
and has the 
Chevalley module $\widehat{T}\simeq J_{G/H}$ as its character module 
where $J_{G/H}=(I_{G/H})^\circ={\rm Hom}_\bZ(I_{G/H},\bZ)$ 
is the dual lattice of $I_{G/H}={\rm Ker}\, \varepsilon$ and 
$\varepsilon : \bZ[G/H]\rightarrow \bZ$, 
$\sum_{\overline{g}\in G/H} a_{\overline{g}}\,\overline{g}\mapsto
 \sum_{\overline{g}\in G/H} a_{\overline{g}}$ 
is the augmentation map, 
i.e. the function field $k(T)\simeq L(J_{G/H})^G$ 
(see \cite[Section 4.8]{Vos98}). 

Let $K/k$ be a separable field extension of degree $n$ 
and $L/k$ be the Galois closure of $K/k$. 
Let $G={\rm Gal}(L/k)$ and $H={\rm Gal}(L/K)$ with $[G:H]=n$.
We have the exact sequence 
$0\rightarrow \bZ\rightarrow \bZ[G/H]\rightarrow J_{G/H}\rightarrow 0$ 
and rank $J_{G/H}=n-1$. 
Write $J_{G/H}=\oplus_{1\leq i\leq n-1}\bZ u_i$. 
We define the action of $G$ on $L(x_1,\ldots,x_{n-1})$ by 
$\sigma(x_i)=\prod_{j=1}^{n-1} x_j^{a_{i,j}} (1\leq i\leq n-1)$ 
for any $\sigma\in G$, when $\sigma (u_i)=\sum_{j=1}^{n-1} a_{i,j} u_j$ 
$(a_{i,j}\in\bZ)$. 
Then the invariant field $L(x_1,\ldots,x_{n-1})^G=L(J_{G/H})^G\simeq k(T)$ (see \cite[Section 1]{EM75}). 

The rationality problem for norm one tori is investigated 
by \cite{EM75}, \cite{CTS77}, \cite{Hur84}, \cite{CTS87}, 
\cite{LeB95}, \cite{CK00}, \cite{LL00}, \cite{Flo}, \cite{End11}, 
\cite{HY17}, \cite{HHY20}, \cite{HY21}, \cite{HY24}, \cite{HY2}. 

When $K/k$ is a finite Galois extension, 
we have that: 

\begin{theorem}[{Voskresenskii \cite[Theorem 7]{Vos70}, Colliot-Th\'{e}l\`{e}ne and Sansuc \cite[Proposition 1]{CTS77}}]
Let $k$ be a field and 
$K/k$ be a finite Galois extension with Galois group $G={\rm Gal}(K/k)$. 
Let $T=R^{(1)}_{K/k}(\bG_m)$ be the norm one torus of $K/k$ 
and $X$ be a smooth $k$-compactification of $T$. 
Then 
$H^1(H,{\rm Pic}\, X_K)\simeq H^3(H,\bZ)$ for any subgroup $H$ of $G$. 
In particular, 
$H^1(k,{\rm Pic}\, \overline{X})\simeq
H^1(G,{\rm Pic}\, X_K)\simeq H^3(G,\bZ)$ which is isomorphic to 
the Schur multiplier $M(G)$ of $G$.
\end{theorem}
In other words, for the $G$-lattice $J_G\simeq \widehat{T}$, 
$H^1(H,[J_G]^{fl})\simeq H^3(H,\bZ)$ for any subgroup $H$ of $G$ 
and $H^1(G,[J_G]^{fl})\simeq H^3(G,\bZ)\simeq H^2(G,\bQ/\bZ)$; 
the Schur multiplier of $G$. 
By the exact sequence $0\to\bZ\to\bZ[G]\to J_G\to 0$, 
we also have $\delta:H^1(G,J_G)\simeq H^2(G,\bZ)\simeq G^{ab}\simeq 
G/[G,G]$ where $\delta$ is the connecting homomorphism and 
$G^{ab}$ is the abelianization of $G$. 


It is easy to see that all the $1$-dimensional algebraic $k$-tori $T$, 
i.e. the trivial torus $\bG_m$ and the norm one torus 
$R^{(1)}_{K/k}(\bG_m)$ of $K/k$ with $[K:k]=2$, are $k$-rational. 
Voskresenskii \cite{Vos67} showed that 
all the $2$-dimensional algebraic $k$-tori $T$ are $k$-rational. 
Kunyavskii \cite{Kun90} gave a 
rational (stably rational, retract rational) classification of 
$3$-dimensional algebraic $k$-tori. 
Hoshi and Yamasaki \cite[Theorem 1.9 and Theorem 1.12]{HY17} classified stably/retract rational 
algebraic $k$-tori of dimension $4$ and $5$ (see also \cite[Section 1]{HKY22}).

A necessary and sufficient condition for the classification 
of stably/retract rational norm one tori $T=R^{(1)}_{K/k}(\bG_m)$ 
is also known in the following cases:\\
{\rm (i)} $[K:k]=n\leq 15$;\\
{\rm (ii)} $n=2^d$;\\
{\rm (iii)} $n=p$ a prime 
except for the case when $G={\rm PSL}_2(\bF_{2^e})\leq S_p$  
with Fermat primes $p=2^e+1\geq 17$;\\
{\rm (iv)} $n=2p$ with $G\leq S_{2p}$ primitive. 

Case (i) was proven in Hoshi and Yamasaki \cite{HY21} for $n\leq 10$ 
except for the stable rationality of 
$G=9T27\simeq \PSL_2(\bF_8)$ and of 
$G=10T11\simeq A_5\times C_2$ which were established by 
Hoshi and Yamasaki \cite{HY2} and Hasegawa, Hoshi and Yamasaki \cite{HHY20} respectively. 
Cases (ii) and (iii) were proven in Hoshi and Yamasaki \cite{HY21}. 
Moreover, in \cite{HHY20}, case (i) for $n=12, 14, 15$ and case (iv) were proven. 

\section{Hasse norm principle and norm one tori}\label{S3}

Let $k$ be a global field, 
i.e. a number field (a finite extension of $\bQ$) 
or a function field of an algebraic curve over 
$\bF_q$ (a finite extension of $\bF_q(t))$.

\begin{definition}
Let $T$ be an algebraic $k$-torus 
and $T(k)$ be the group of $k$-rational points of $T$. 
Then $T(k)$ 
embeds into $\prod_{v\in V_k} T(k_v)$ by the diagonal map 
where 
$V_k$ is the set of all places of $k$ and 
$k_v$ is the completion of $k$ at $v$. 
Let $\overline{T(k)}$ be the closure of $T(k)$  
in the product $\prod_{v\in V_k} T(k_v)$. 
The group 
\begin{align*}
A(T)=\left(\prod_{v\in V_k} T(k_v)\right)/\overline{T(k)}
\end{align*}
is called {\it the kernel of the weak approximation} of $T$. 
We say that {\it $T$ has the weak approximation property} if $A(T)=0$. 
\end{definition}

\begin{definition}
Let $E$ be a principal homogeneous space (= torsor) under $T$.  
{\it Hasse principle holds for $E$} means that 
if $E$ has a $k_v$-rational point for all $k_v$, 
then $E$ has a $k$-rational point. 
The set $H^1(k,T)$ classifies all such torsors $E$ up 
to (non-unique) isomorphism. 
We define {\it the Shafarevich-Tate group} of $T$: 
\begin{align*}
\Sha(T)={\rm Ker}\left\{H^1(k,T)\xrightarrow{\rm res} \bigoplus_{v\in V_k} 
H^1(k_v,T)\right\}.
\end{align*}
Then 
Hasse principle holds for all torsors $E$ under $T$ 
if and only if $\Sha(T)=0$. 
\end{definition}

%
\begin{theorem}[{Voskresenskii \cite[Theorem 5, page 1213]{Vos69}, 
\cite[Theorem 6, page 9]{Vos70}, see also \cite[Section 11.6, Theorem, page 120]{Vos98}}]\label{thV}
Let $k$ be a global field, 
$T$ be an algebraic $k$-torus and $X$ be a smooth $k$-compactification of $T$. 
Then there exists an exact sequence
\begin{align*}
0\to A(T)\to H^1(k,{\rm Pic}\,\overline{X})^{\vee}\to \Sha(T)\to 0
\end{align*}
where $M^{\vee}={\rm Hom}(M,\bQ/\bZ)$ is the Pontryagin dual of $M$. 
In particular, if $T$ is retract $k$-rational, then $ H^1(k,{\rm Pic}\,\overline{X})=0$ and hence 
$A(T)=0$ and $\Sha(T)=0$. 
Moreover, if $L$ is the splitting field of $T$ and $L/k$ 
is an unramified extension, then $A(T)=0$ and 
$H^1(k,{\rm Pic}\,\overline{X})^{\vee}\simeq \Sha(T)$. 
\end{theorem}
For the last assertion, see \cite[Section 11.5]{Vos98}. 
It follows that 
$H^1(k,{\rm Pic}\,\overline{X})=0$ if and only if $A(T)=0$ and $\Sha(T)=0$, 
i.e. $T$ has the weak approximation property and 
Hasse principle holds for all torsors $E$ under $T$. 
Theorem \ref{thV} was generalized 
to the case of linear algebraic groups by Sansuc \cite{San81}.

\begin{definition}
Let $k$ be a 
global field, 
$K/k$ be a finite extension and 
$\bA_K^\times$ be the idele group of $K$. 
We say that {\it the Hasse norm principle holds for $K/k$} 
if $(N_{K/k}(\bA_K^\times)\cap k^\times)/N_{K/k}(K^\times)=1$ 
where $N_{K/k}$ is the norm map. 
\end{definition}

Hasse \cite[Satz, page 64]{Has31} proved that 
the Hasse norm principle holds for any cyclic extension $K/k$ 
but does not hold for bicyclic extension $\bQ(\sqrt{-39},\sqrt{-3})/\bQ$. 
For Galois extensions $K/k$, Tate \cite{Tat67} gave the following theorem:
%


\begin{theorem}[{Tate \cite[page 198]{Tat67}}]\label{thTate}
Let $k$ be a global field, $K/k$ be a finite Galois extension 
with Galois group ${\rm Gal}(K/k)\simeq G$. 
Let $V_k$ be the set of all places of $k$ 
and $G_v$ be the decomposition group of $G$ at $v\in V_k$. 
Then 
\begin{align*}
(N_{K/k}(\bA_K^\times)\cap k^\times)/N_{K/k}(K^\times)\simeq 
{\rm Coker}\left\{\bigoplus_{v\in V_k}\widehat H^{-3}(G_v,\bZ)\xrightarrow{\rm cores}\widehat H^{-3}(G,\bZ)\right\}
\end{align*}
where $\widehat H$ is the Tate cohomology. 
In particular, the Hasse norm principle holds for $K/k$ 
if and only if the restriction map 
$H^3(G,\bZ)\xrightarrow{\rm res}\bigoplus_{v\in V_k}H^3(G_v,\bZ)$ 
is injective. 
\end{theorem}
If $G\simeq C_n$ is cyclic, then 
$\widehat H^{-3}(G,\bZ)\simeq H^3(G,\bZ)\simeq H^1(G,\bZ)=0$ 
and hence Hasse's original theorem follows. 
If there exists a place $v$ of $k$ such that $G_v=G$, then 
the Hasse norm principle also holds for $K/k$. 
For example, the Hasse norm principle holds for $K/k$ with 
$G\simeq V_4$ if and only if 
there exists a place $v$ of $k$ such that $G_v=V_4$ because 
$H^3(V_4,\bZ)\simeq\bZ/2\bZ$ and $H^3(C_2,\bZ)=0$. 
The Hasse norm principle holds for $K/k$ with 
$G\simeq (C_2)^3$ if and only if {\rm (i)} 
there exists a place $v$ of $k$ such that $G_v=G$ 
or {\rm (ii)} 
there exist places $v_1,v_2,v_3$ of $k$ such that 
$G_{v_i}\simeq V_4$ and 
$H^3(G,\bZ)\xrightarrow{\rm res}
H^3(G_{v_1},\bZ)\oplus H^3(G_{v_2},\bZ)\oplus H^3(G_{v_3},\bZ)$ 
is an isomorphism because 
$H^3(G,\bZ)\simeq(\bZ/2\bZ)^{\oplus 3}$ and 
$H^3(V_4,\bZ)\simeq \bZ/2\bZ$. 

Ono \cite{Ono63} established the relationship 
between the Hasse norm principle for $K/k$ 
and Hasse principle for all torsors $E$ under
the norm one torus $R^{(1)}_{K/k}(\bG_m)$: 
\begin{theorem}[{Ono \cite[page 70]{Ono63}, see also Platonov \cite[page 44]{Pla82}, Kunyavskii \cite[Remark 3]{Kun84}, Platonov and Rapinchuk \cite[page 307]{PR94}}]\label{thOno}
Let $k$ be a global field and $K/k$ be a finite extension. 
Then 
\begin{align*}
\Sha(R^{(1)}_{K/k}(\bG_m))\simeq (N_{K/k}(\bA_K^\times)\cap k^\times)/N_{K/k}(K^\times).
\end{align*}
In particular, $\Sha(R^{(1)}_{K/k}(\bG_m))=0$ if and only if 
the Hasse norm principle holds for $K/k$. 
\end{theorem}

For norm one tori $T=R^{(1)}_{K/k}(\bG_m)$, 
recall that 
the function field $k(T)$ may be regarded as $L(M)^G$ 
for the character module $M=J_{G/H}$ and hence we have: 
\begin{align*}
[J_{G/H}]^{fl}=0\,
\ \ \Rightarrow\ \ [J_{G/H}]^{fl}\ \textrm{is\ invertible}
\ \ \Rightarrow\ \  H^1(G,[J_{G/H}]^{fl})=0\,
\ \ \Rightarrow\ \  A(T)=0\ \textrm{and}\ \Sha(T)=0
\end{align*}
where the last implication holds over a global field $k$ 
(see Theorem \ref{thV}, 
see also Colliot-Th\'{e}l\`{e}ne and Sansuc \cite[page 29]{CTS77}). 
The last conditions mean that 
$T$ has the weak approximation property and 
the Hasse norm principle holds for $K/k$ as above.  
In particular, it follows that 
$[J_{G/H}]^{fl}$ is invertible, i.e. $T$ is retract $k$-rational, 
and hence $A(T)=0$ and $\Sha(T)=0$ when $G=pTm\leq S_p$ is a transitive 
subgroup of $S_p$ of prime degree $p$ 
and $H=G\cap S_{p-1}\leq G$ with $[G:H]=p$ (see 
Colliot-Th\'{e}l\`{e}ne and Sansuc \cite[Proposition 9.1]{CTS87} and 
\cite[Lemma 2.17]{HY17}). 
Hence the Hasse norm principle holds for $K/k$ when $[K:k]=p$. 

Applying Theorem \ref{thV} to $T=R^{(1)}_{K/k}(\bG_m)$,  
it follows from Theorem \ref{thOno} that 
$H^1(k,{\rm Pic}\,\overline{X})=0$ if and only if 
$A(T)=0$ and $\Sha(T)=0$, 
i.e. 
$T$ has the weak approximation property and 
the Hasse norm principle holds for $K/k$. 
In the algebraic language, 
the latter condition $\Sha(T)=0$ means that 
for the corresponding norm hypersurface $f(x_1,\ldots,x_n)=b$, 
it has a $k$-rational point 
if and only if it has a $k_v$-rational point 
for any place $v$ of $k$ where 
$f\in k[x_1,\ldots,x_n]$ is the polynomial of total 
degree $n$ defined by the norm map $N_{K/k}:K^\times\to k^\times$ 
and $b\in k^\times$ 
(see \cite[Example 4, page 122]{Vos98}).

By Poitou-Tate duality (see Milne \cite[Theorem 4.20]{Mil86}, 
Platonov and Rapinchuk \cite[Theorem 6.10]{PR94}, 
Neukirch, Schmidt and Wingberg \cite[Theorem 8.6.8]{NSW00}, 
Harari \cite[Theorem 17.13]{Har20}), 
we also have 
\begin{align*}
\Sha(T)^\vee\simeq\Sha^2(G,\widehat{T})
\end{align*}
where  
\begin{align*}
\Sha^i(G,\widehat{T})={\rm Ker}\left\{H^i(G,\widehat{T})\xrightarrow{\rm res} \bigoplus_{v\in V_k} 
H^i(G_v,\widehat{T})\right\}\quad (i\geq 1)
\end{align*}
is {\it the $i$-th Shafarevich-Tate group} 
of $\widehat{T}={\rm Hom}(T,\bG_m)$, 
$G={\rm Gal}(L/k)$ and $L$ is the minimal splitting field of $T$. 
Note that $\Sha(T)\simeq \Sha^1(G,T)\simeq \Sha^2(G,\widehat{T})^\vee$. 
In the special case where 
$T=R^{(1)}_{K/k}(\bG_m)$ and $K/k$ is Galois with $G={\rm Gal}(K/k)$, 
we have $\widehat{T}\simeq J_{G}$ and 
$H^2(G,J_{G})\simeq H^3(G,\bZ)$ and hence 
we get Tate's theorem (Theorem \ref{thTate}). 

The Hasse norm principle for Galois extensions $K/k$ 
was investigated by Gerth \cite{Ger77}, \cite{Ger78} and 
Gurak \cite{Gur78a}, \cite{Gur78b}, \cite{Gur80} 
(see also \cite[pages 308--309]{PR94}). 
Gurak \cite{Gur78a} showed that 
the Hasse norm principle holds for a Galois extension $K/k$ 
if all the Sylow subgroups of ${\rm Gal}(K/k)$ are cyclic. 
Note that this also follows from Theorem \ref{thOno} 
and the retract $k$-rationality of 
$T=R^{(1)}_{K/k}(\bG_m)$ due to Endo and Miyata \cite[Theorem 2.3]{EM75}. 

For non-Galois extension $K/k$ of degree $n$, 
the Hasse norm principle was investigated by 
Bartels \cite{Bar81a} (holds for $n=p$; prime), 
\cite{Bar81b} (holds for $G\simeq D_n$), 
Voskresenskii and Kunyavskii \cite{VK84} (holds for $G\simeq S_n$), 
Kunyavskii \cite{Kun84} $(n=4)$, 
Drakokhrust and Platonov \cite{DP87} $(n=6)$, 
Macedo \cite{Mac20} (holds for $G\simeq A_n$ ($n\neq 4)$), 
Macedo and Newton \cite{MN22} 
($G\simeq A_4$, $S_4$, $A_5$, $S_5$, $A_6$, $A_7$ $($general $n))$, 
Hoshi, Kanai and Yamasaki \cite{HKY22} $(n\leq 15$ $(n\neq 12))$ 
(holds for $G\simeq M_n$ ($n=11,12,22,23,24$; $5$ Mathieu groups)), 
\cite{HKY23} $(n=12)$, 
\cite{HKY} $(G\simeq M_{11}$, $J_1$ $($general $n))$, 
Hoshi and Yamasaki \cite{HY2} 
(holds for $G\simeq {\rm PSL}_2(\bF_7)$ $(n=21)$, 
${\rm PSL}_2(\bF_8)$ $(n=63)$) 
where $G={\rm Gal}(L/k)$ and $L/k$ is the Galois closure of $K/k$. 
Recall that the 
case where $n=p$ 
also follows from Theorem \ref{thOno} and 
the retract $k$-rationality of 
$T=R^{(1)}_{K/k}(\bG_m)$ due to 
Colliot-Th\'{e}l\`{e}ne and Sansuc \cite[Proposition 9.1]{CTS87}. 

\section{{Proof of Theorem \ref{thmain1}}}\label{S4}

Let $K/k$ be a separable field extension of degree $16$ 
and $L/k$ be the Galois closure of $K/k$. 
Let $G={\rm Gal}(L/k)=16Tm$ $(1\leq m\leq 1954)$ 
be a transitive subgroup of $S_{16}$ 
and $H={\rm Gal}(L/K)$ with $[G:H]=16$. 

Let $T=R^{(1)}_{K/k}(\bG_m)$ be the norm one torus of $K/k$ 
of dimension $15$. 
We have the character module $\widehat{T}=J_{G/H}$ of $T$ 
and then 
$H^1(k,{\rm Pic}\,\overline{X})\simeq H^1(G,[J_{G/H}]^{fl})$ 
(see Section \ref{S2}). 
We may assume that 
$H$ is the stabilizer ${\rm Stab}_1(G)$ of $1$ in $G$, 
i.e. $L=k(\theta_1,\ldots,\theta_{16})$ and $K=k(\theta_1)$. 
We first prepare two useful lemmas: 
\begin{lemma}\label{lem4.1}
Let $G=nTm$ be a transitive subgroup of $S_n$ and 
$G_p={\rm Syl}_p(G)$ be a $p$-Sylow subgroup of $G$. 
If $n=p^d$ is a prime power, then $G_p\leq S_n$ is transitive. 
\end{lemma}
%
%
\begin{proof}
We take the stabilizer $H={\rm Stab}_1(G)$ of $1$ in $G$. 
Then ${\rm Syl}_p(H)$ is a $p$-subgroup of $G$ so is contained in 
a Sylow $p$-subgroup $G_p^{\prime}$ of $G$, 
which is conjugate to $G_p$ so is of the form $G_p^{\prime}=x\,G_p\,x^{-1}$ for some $x\in G$. 
Hence ${\rm Syl}_p(H)\leq G_p^\prime$ 
so ${\rm Syl}_p({\rm Stab}_i(G))=x^{-1}{\rm Syl}_p(H)x\leq G_p$ $(1\leq i\leq p^d)$. 
Now ${\rm Syl}_p({\rm Stab}_i(G))={\rm Stab}_i(G)\cap G_p$ (clearly 
${\rm Syl}_p({\rm Stab}_i(G))\leq {\rm Stab}_i(G)\cap G_p$ and the reverse inclusion comes from 
considering indexes in ${\rm Stab}_i(G)$). 
Hence $|{\rm Orb}_{G_p}(i)|=[G_p : {\rm Stab}_i(G_p)]=[G_p : {\rm Stab}_i(G)\cap G_p]=[G_p : {\rm Syl}_p({\rm Stab}_i(G))]$. 
Suppose that $|G|=p^rs$ with $p$ ${\not{\mid}}$ $s$. 
Because $G\leq S_n$ is transitive, 
$|{\rm Stab}_i(G)|=|G|/|{\rm Orb}_i(G)|=p^rs/p^d=p^{r-d}s$ 
and so $[G_p : {\rm Syl}_p({\rm Stab}_i(G))]=p^r/p^{r-d}=p^d=n=[G:H]$. 
This implies that $|{\rm Orb}_{G_p}(i)|=p^d$. 
Hence $G_p\leq S_{p^d}$ is transitive. 
\end{proof}
\begin{lemma}\label{lem4.2}
Let $G=nTm$ be a transitive subgroup of $S_n$ and 
$G_p={\rm Syl}_p(G)$ be a $p$-Sylow subgroup of $G$. 
Let 
$H\leq G$ be a subgroup with $[G:H]=n$ 
and $F=[J_{G/H}]^{fl}$ be a flabby class. 
Let $H^1(G,F)_{(p)}$ be the $p$-primary component of $H^1(G,F)$. 
Then there exists an injection $H^1(G,F)_{(p)}\hookrightarrow H^1(G_p,F|_{G_p})$ where $p\mid |G|$. 
In particular, 
if $H^1(G_p,F|_{G_p})=0$ for all $p\mid n$, then $H^1(G, F)=0$. 
\end{lemma}
\begin{proof} 
We have $H^1(G,F)=\bigoplus_{p\,\mid\, |G|}H^1(G,F)_{(p)}$ where 
$H^1(G,F)_{(p)}$ is the $p$-primary component of $H^1(G,F)$. 
Because $[G:G_p]$ is relatively prime to $p$, 
${\rm Cor}^G_{G_p}\circ{\rm Res}^G_{G_p}=[G:G_p]$ (see e.g. Brown \cite[Proposition 9.5]{Bro82}) 
gives an isomorphism on $H^1(G,F)_{(p)}$.
In particular,  
${\rm Res}^G_{G_p}: H^1(G,F)\rightarrow H^1(G_p,F|_{G_p})$ induces 
an injection $H^1(G,F)_{(p)}\hookrightarrow H^1(G_p,F|_{G_p})$ where $p\mid |G|$. 
On the other hand, ${\rm Cor}^G_H\circ {\rm Res}^G_H=[G:H]=n$ 
and ${\rm Res}^G_H:H^1(G,F)\rightarrow H^1(H,F|_{H})=0$ because 
$[F|_H]=0$ (see also Macedo and Newton \cite[Lemma 2.6]{MN22}). 
Hence $n\cdot H^1(G,F)=0$. 
In other words, $H^1(G,F)_{(p)}=0$ if $p$ $\not{|}$ $n$. 
It follows from $H^1(G_p,F|_{G_p})=0$ that $H^1(G,F)_{(p)}=0$ for all $p\mid n$. 
We conclude that $H^1(G, F)=0$. 
\end{proof}

We made the following GAP (\cite{GAP}) function 
in order to compute a flabby class $F=[J_{G/H}]^{fl}$ efficiently:\\

{\tt FlabbyResolutionNorm1TorusJ(}$d${\tt ,}$m${\tt ).actionF} returns 
the matrix representation of the action of $G$ on a flabby class 
$F=[J_{G/H}]^{fl}$ for the $m$-th transitive subgroup $G=dTm\leq S_n$ 
of degree $n$ where $H$ is the stabilizer of one of the letters in $G$. 
(This function is similar to {\tt FlabbyResolution(Norm1TorusJ(}$d${\tt ,}$m${\tt )).actionF} but 
it may speed up and save memory resources.)\\ 

{\it Proof of Theorem \ref{thmain1}.} 
By Lemma \ref{lem4.1} and Lemma \ref{lem4.2}, 
because if $H^1(G_2,F|_{G_2})=0$, then 
$H^1(G,F)=0$ where $F=[J_{G/H}]^{fl}$ and $G_2={\rm Syl}_2(G)$ 
is a $2$-Sylow subgroup of $G$, 
we should first compute 
$H^1(G,F)\simeq H^1(k,{\rm Pic}\,\overline{X})$ 
for any $2$-group 
$G=16Tm\leq S_{16}$. 
There exist $1427$ transitive $2$-groups $G=16Tm\leq S_{16}$ (out of $1954$). 
In order to compute $H^1(G,F)$ for $2$-groups $G=16Tm$, 
we apply 
{\tt FlabbyResolutionNorm1TorusJ($16,m$).actionF} 
to get a flabby class $F=[J_{G/H}]^{fl}$ 
and {\tt H1($F$)} to get $H^1(G,F)$. 
Then we see that there exist $773$ 
(resp. $628$, $24$, $36$, $1$, $5$) $2$-groups $G=16Tm$ 
such that $H^1(G,F)=0$ (resp. $\bZ/2\bZ$, 
$(\bZ/2\bZ)^{\oplus 2}$, 
$(\bZ/2\bZ)^{\oplus 3}$, 
$(\bZ/2\bZ)^{\oplus 6}$, 
$\bZ/4\bZ$). 
In particular, there exist $694$ $(=628+24+36+1+5)$ 
$2$-groups $G=16Tm$ such that $H^1(G,F)\neq 0$. 

For the $527$ $(=1954-1427)$ non-$2$-groups $G=16Tm$, 
we find that $287$ (resp. $240$) of the $527$ cases 
satisfy 
$H^1({\rm Syl}_2(G),F|_{{\rm Syl}_2(G)})=0$ (resp. $\neq 0$). 
Then we again apply 
{\tt FlabbyResolutionNorm1TorusJ($16,m$).actionF} 
and {\tt H1($F$)} to get $H^1(G,F)$ 
for the $240$ non-$2$-groups $G$ 
with $H^1({\rm Syl}_2(G),F|_{{\rm Syl}_2(G)})\neq 0$ 
(see Example \ref{ex4.3}). 

The last assertion follows from Theorem \ref{thV}.\qed
%
\begin{remark}
When $|H|$ is odd, it follows from 
Macedo and Newton \cite[Corollary 3.4]{MN22} that 
$H^1(G,F)\simeq H^3(G,\bZ)_{(2)}$ 
where $F=[J_{G/H}]^{fl}$ with $[G:H]=16$ 
and $H^3(G,\bZ)_{(2)}$ is the $2$-primary component of 
the Schur multiplier $H^3(G,\bZ)$.
There exist $14$ (resp. $10$, $1$, $1$, $1$, $1$, $4$) cases 
with $|H|=1$ (resp. $3$, $5$, $7$, $9$, $15$, $21$) 
out of $1954$ $G=16Tm$ cases. 
\end{remark}
%
{\it Alternative proof of Theorem \ref{thmain1} when $G=16Tm$ is primitive.} 
We assume that $G\leq S_{16}$ is primitive. 
Then there exist $22$ primitive $G=16Tm$ with 
$m=178$, $415$, $447$, $708$, $711$, $777$, $1030$, $1034$, $1079$, 
$1080$, $1081$, $1294$, $1328$, $1329$, $1508$, $1653$, $1654$, 
$1753$, $1840$, $1906$, $1953$, $1954$. 

Let $T=R^{(1)}_{K/k}(\bG_m)$ be the norm one torus of $K/k$ 
of dimension $15$. 
We have the character module $\widehat{T}=J_{G/H}$ of $T$ 
and then 
$H^1(k,{\rm Pic}\,\overline{X})\simeq H^1(G,[J_{G/H}]^{fl})$ 
(see Section \ref{S2}). 
We may assume that 
$H$ is the stabilizer of $1$ in $G$, 
i.e. $L=k(\theta_1,\ldots,\theta_{16})$ and $K=k(\theta_1)$. 
In order to compute $H^1(G,[J_{G/H}]^{fl})$, 
we apply 
{\tt FlabbyResolutionNorm1TorusJ($16,m$).actionF} and  {\tt H1($F$)} 
to obtain a flabby class $F=[J_{G/H}]^{fl}$ and  $H^1(G,F)$ 
for primitive $G=16Tm\leq S_{16}$ except for 
$G=16T1953\simeq A_{16}$, $16T1954\simeq S_{16}$ (see Example \ref{exH1F}). 

For $G=16T1953\simeq A_{16}$, $16T1954\simeq S_{16}$, 
we did not get an answer by computer calculations 
because it needs much time and memory resources. 
However, 
we already know that 
$H^1(G,F)=0$ for 
$G=16T1953\simeq A_{16}$, $16T1954\simeq S_{16}$  
by Macedo \cite{Mac20} 
and Voskresenskii and Kunyavskii \cite{VK84} respectively 
(see also \cite[Theorem 4, Corollary]{Vos88} 
and \cite[Theorem 1.10 and Theorem 1.11]{HKY22}). 

The last assertion follows from Theorem \ref{thV}.\qed\\

Some related functions for Example \ref{ex4.3} and Example \ref{exH1F}  
are available 
as in \cite{Norm1ToriHNP}. 

\smallskip
\begin{example}[{Computation of $H^1(G,[J_{G/H}]^{fl})$ for $G=16Tm$  $(1\leq m\leq 1954)$}]\label{ex4.3}
~{}\vspace*{-4mm}\\
{\small 

}
\end{example}

\smallskip
\begin{example}[{Computation of $H^1(G,[J_{G/H}]^{fl})$ for primitive $G=16Tm$}]\label{exH1F}
~{}\vspace*{-4mm}\\
{\small 
\begin{verbatim}
gap> Read("FlabbyResolutionFromBase.gap");
gap> NrTransitiveGroups(16); # there exist 1954 transitive groups G=16Tm
1954
gap> NrPrimitiveGroups(16); # there exist 22 primitive transitive groups G=16Tm
22
gap> prim16:=Filtered([1..1954],x->IsPrimitive(TransitiveGroup(16,x)));
[ 178, 415, 447, 708, 711, 777, 1030, 1034, 1079, 1080, 1081, 1294, 1328, 
  1329, 1508, 1653, 1654, 1753, 1840, 1906, 1953, 1954 ]
gap> List(prim16,x->StructureDescription(TransitiveGroup(16,x)));
[ "(C2 x C2 x C2 x C2) : C5", 
  "((C2 x C2 x C2 x C2) : C5) : C2",
  "(C2 x C2 x C2 x C2) : C15", 
  "(((C2 x C2 x C2 x C2) : C3) : C2) : C3",
  "((C2 x C2 x C2 x C2) : C5) : C4", 
  "(((C2 x C2 x C2 x C2) : C5) : C2) : C3", 
  "(A4 x A4) : C4", 
  "((((C2 x C2 x C2 x C2) : C3) : C2) : C3) : C2",
  "((C2 x C2 x C2 x C2) : C15) : C4", 
  "(C2 x C2 x C2 x C2) : A5",
  "(C2 x C2 x C2 x C2) : A5", 
  "(S4 x S4) : C2", 
  "(C2 x C2 x C2 x C2) : S5",
  "(C2 x C2 x C2 x C2) : S5", 
  "(C2 x C2 x C2 x C2) : GL(2,4)",
  "(C2 x C2 x C2 x C2) : (A5 : S3)", 
  "(C2 x C2 x C2 x C2) : A6",
  "(C2 x C2 x C2 x C2) : S6", 
  "(C2 x C2 x C2 x C2) : A7",
  "(C2 x C2 x C2 x C2) : A8", 
  "A16", 
  "S16" ]
gap> for i in prim16{[1..20]} do
> F:=FlabbyResolutionNorm1TorusJ(16,i).actionF;
> Print([[16,i],Length(F.1),Filtered(H1(F),x->x>1)],"\n");od; 
# computing H^1(G,F) for 20 primitive transitive groups G=16Tm except for A16, S16
[ [ 16, 178 ], 65, [ 2, 2 ] ]
[ [ 16, 415 ], 265, [  ] ]
[ [ 16, 447 ], 225, [  ] ]
[ [ 16, 708 ], 297, [ 2 ] ]
[ [ 16, 711 ], 415, [  ] ]
[ [ 16, 777 ], 345, [  ] ]
[ [ 16, 1030 ], 507, [  ] ]
[ [ 16, 1034 ], 651, [  ] ]
[ [ 16, 1079 ], 555, [  ] ]
[ [ 16, 1080 ], 1725, [ 2, 2 ] ]
[ [ 16, 1081 ], 955, [  ] ]
[ [ 16, 1294 ], 1083, [  ] ]
[ [ 16, 1328 ], 1075, [  ] ]
[ [ 16, 1329 ], 2595, [ 2 ] ]
[ [ 16, 1508 ], 1725, [  ] ]
[ [ 16, 1653 ], 3915, [  ] ]
[ [ 16, 1654 ], 4875, [ 2 ] ]
[ [ 16, 1753 ], 4875, [ 2 ] ]
[ [ 16, 1840 ], 16635, [  ] ]
[ [ 16, 1906 ], 26715, [  ] ]
\end{verbatim}
}
\end{example}

%
\section{Drakokhrust and Platonov's method}\label{S5}

Let $k$ be a number field, $K/k$ be a finite extension, 
$\bA_K^\times$ be the idele group of $K$ and 
$L/k$ be the Galois closure of $K/k$. 
Let $G={\rm Gal}(L/k)=nTm$ be a transitive subgroup of $S_n$ 
and $H={\rm Gal}(L/K)$ with $[G:H]=n$. 

For $x,y\in G$, we denote $[x,y]=x^{-1}y^{-1}xy$ the commutator of 
$x$ and $y$, and $[G,G]$ the commutator group of $G$. 
Let $V_k$ be the set of all places of $k$ 
and $G_v$ be the decomposition group of $G$ at $v\in V_k$. 

\begin{definition}[{Drakokhrust and Platonov \cite[page 350]{PD85a}, \cite[page 300]{DP87}}]
Let $k$ be a number field, 
$L\supset K\supset k$ be a tower of finite extensions 
where $L$ is normal over $k$. 

We call the group 
\begin{align*}
{\rm Obs}(K/k)=(N_{K/k}(\bA_K^\times)\cap k^\times)/N_{K/k}(K^\times)
\end{align*}
{\it the total obstruction to the Hasse norm principle for $K/k$} 
and 
\begin{align*}
{\rm Obs}_1(L/K/k)=\left(N_{K/k}(\bA_K^\times)\cap k^\times\right)/\left((N_{L/k}(\bA_L^\times)\cap k^\times)N_{K/k}(K^\times)\right)
\end{align*}
{\it the first obstruction to the Hasse norm principle for $K/k$ 
corresponding to the tower 
$L\supset K\supset k$}. 
\end{definition}

Note that (i) 
${\rm Obs}(K/k)=1$ if and only if 
the Hasse norm principle holds for $K/k$; 
and (ii) ${\rm Obs}_1(L/K/k)
={\rm Obs}(K/k)/(N_{L/k}(\bA_L^\times)\cap k^\times)$. 

Drakokhrust and Platonov gave a formula 
for computing the first obstruction ${\rm Obs}_1(L/K/k)$: 

\begin{theorem}[{Drakokhrust and Platonov \cite[page 350]{PD85a}, \cite[pages 789--790]{PD85b}, \cite[Theorem 1]{DP87}}]\label{thDP2}
Let $k$ be a number field, 
$L\supset K\supset k$ be a tower of finite extensions 
where 
$L$ is normal over $k$.  
Let $G={\rm Gal}(L/k)$ and $H={\rm Gal}(L/K)$. 
Let $G_v$ $($resp. $H_w$$)$ be the decomposition group of $G$ $($resp. $H$$)$ at $v\in V_k$ $($resp. $w\in V_K$$)$.  
Then 
\begin{align*}
{\rm Obs}_1(L/K/k)\simeq 
{\rm Ker}\, \psi_1/\varphi_1({\rm Ker}\, \psi_2)
\end{align*}
where 
\begin{align*}
\begin{CD}
H/[H,H] @>\psi_1 >> G/[G,G]\\
@AA\varphi_1 A @AA\varphi_2 A\\
\displaystyle{\bigoplus_{v\in V_k}\left(\bigoplus_{w\mid v} H_w/[H_w,H_w]\right)} @>\psi_2 >> 
\displaystyle{\bigoplus_{v\in V_k} G_v/[G_v,G_v]}, 
\end{CD}
\end{align*}
$\psi_1$, $\varphi_1$ and $\varphi_2$ are defined 
by the inclusions $H\subset G$, $H_w\subset H$ and $G_v\subset G$ respectively, and 
\begin{align*}
\psi_2(h[H_{w},H_{w}])=x^{-1}hx[G_v,G_v]
\end{align*}
for $h\in H_{w}=H\cap xG_vx^{-1}$ $(x\in G)$. 
\end{theorem}


Let $\psi_2^{v}$ be the restriction of $\psi_2$ to the subgroup 
$\bigoplus_{w\mid v} H_w/[H_w,H_w]$ with respect to $v\in V_k$ 
and $\psi_2^{\rm nr}$ (resp. $\psi_2^{\rm r}$) be 
the restriction of $\psi_2$ to the unramified (resp. the ramified) 
places $v$ of $k$. 
\begin{proposition}[{Drakokhrust and Platonov \cite{DP87}}]\label{propDP}
Let $k$, 
$L\supset K\supset k$, 
$G$ and $H$ be as in Theorem \ref{thDP2}.\\
{\rm (i)} $($\cite[Lemma 1]{DP87}$)$ 
Places $w_i\mid v$ of $K$ are in one-to-one correspondence 
with the set of double cosets in the decomposition 
$G=\cup_{i=1}^{r_v} Hx_iG_v$ where $H_{w_i}=H\cap x_iG_vx_i^{-1}$;\\
{\rm (ii)} $($\cite[Lemma 2]{DP87}$)$ 
If $G_{v_1}\leq G_{v_2}$, then $\varphi_1({\rm Ker}\,\psi_2^{v_1})\subset \varphi_1({\rm Ker}\,\psi_2^{v_2})$;\\
{\rm (iii)} $($\cite[Theorem 2]{DP87}$)$ 
$\varphi_1({\rm Ker}\,\psi_2^{\rm nr})=\Phi^G(H)/[H,H]$ 
where $\Phi^G(H)=\langle [h,x]\mid h\in H\cap xHx^{-1}, x\in G\rangle$;\\
{\rm (iv)} $($\cite[Lemma 8]{DP87}$)$ If $[K:k]=p^r$ $(r\geq 1)$ 
and ${\rm Obs}(K_p/k_p)=1$ where $k_p=L^{G_p}$, $K_p=L^{H_p}$, 
$G_p$ and $H_p\leq H\cap G_p$ are $p$-Sylow subgroups of $G$ and $H$ 
respectively, then ${\rm Obs}(K/k)=1$.
\end{proposition}

Note that the inverse direction of Proposition \ref{propDP} (iv) 
does not hold in general. 
For example, if $n=8$, $G=8T13\simeq A_4\times C_2$ and 
there exists a place $v$ of $k$ such that $G_v\simeq V_4$, 
then ${\rm Obs}(K/k)=1$ but $G_2=8T3\simeq (C_2)^3$ 
and ${\rm Obs}(K_2/k_2)\neq 1$ may occur 
(see Hoshi, Kanai and Yamasaki \cite[Theorem 1.18]{HKY22}, cf. Lemma \ref{lem4.2} also).

\begin{theorem}[{Drakokhrust and Platonov \cite[Theorem 3, Corollary 1]{DP87}}]\label{thDP87}
Let $k$, 
$L\supset K\supset k$, 
$G$ and $H$ be as in Theorem \ref{thDP2}. 
Let $H_i\leq G_i\leq G$ $(1\leq i\leq m)$, 
$H_i\leq H\cap G_i$, 
$k_i=L^{G_i}$ and $K_i=L^{H_i}$. 
If ${\rm Obs}(K_i/k_i)=1$ for all $1\leq i\leq m$ and 
\begin{align*}
\bigoplus_{i=1}^m \widehat{H}^{-3}(G_i,\bZ)\xrightarrow{\rm cores} 
\widehat{H}^{-3}(G,\bZ)
\end{align*}
is surjective, 
then ${\rm Obs}(K/k)={\rm Obs}_1(L/K/k)$. 
In particular, 
if $[K:k]=n$ is square-free, 
then ${\rm Obs}(K/k)={\rm Obs}_1(L/K/k)$.
\end{theorem}

We note that if $L/k$ is an unramified extension, 
then $A(T)=0$ and $H^1(k,{\rm Pic}\, \overline{X})\simeq 
H^1(G,[J_{G/H}]^{fl})\simeq \Sha(T)\simeq 
{\rm Obs}(K/k)$ where $T=R^{(1)}_{K/k}(\bG_m)$ 
(see Theorem \ref{thV} and Theorem \ref{thOno}). 
If, in addition, ${\rm Obs}(K/k)={\rm Obs}_1(L/K/k)$ 
(e.g. $[K:k]=n$ is square-free as in Theorem \ref{thDP87}), 
then ${\rm Obs}(K/k)={\rm Obs}_1(L/K/k)=
{\rm Ker}\, \psi_1/\varphi_1({\rm Ker}\, \psi_2^{\rm nr})\simeq$ 
$((H$ $\cap$ $[G,G])/[H,H])$$/$$(\Phi^G(H)/[H,H])\simeq (H$ $\cap$ $[G,G])/\Phi^G(H)$ 
(see Proposition \ref{propDP} (iii)). 

\begin{theorem}[{Drakokhrust \cite[Theorem 1]{Dra89}, see also Opolka \cite[Satz 3]{Opo80}}]\label{thDra89}
Let $k$, 
$L\supset K\supset k$, 
$G$ and $H$ be as in Theorem \ref{thDP2}. 
Assume that $\widetilde{L}\supset L\supset k$ is 
a tower of Galois extensions with 
$\widetilde{G}={\rm Gal}(\widetilde{L}/k)$ 
and $\widetilde{H}={\rm Gal}(\widetilde{L}/K)$ 
which correspond to a central extension 
$1\to A\to \widetilde{G}\to G\to 1$ with 
$A\cap[\widetilde{G},\widetilde{G}]\simeq M(G)=H^2(G,\bC^\times)$; 
the Schur multiplier of $G$ 
$($this is equivalent to 
the inflation 
$M(G)\to M(\widetilde{G})$ being the zero map, 
see {\rm Beyl and Tappe \cite[Proposition 2.13, page 85]{BT82}}$)$. 
Then 
${\rm Obs}(K/k)={\rm Obs}_1(\widetilde{L}/K/k)$. 
In particular, if $\widetilde{G}$ is a Schur cover of $G$, 
i.e. $A\simeq M(G)$, then ${\rm Obs}(K/k)={\rm Obs}_1(\widetilde{L}/K/k)$. 
\end{theorem}

Indeed, Drakokhrust \cite[Theorem 1]{Dra89} shows that 
${\rm Obs}(K/k)\simeq 
{\rm Ker}\, \widetilde{\psi}_1/\widetilde{\varphi}_1({\rm Ker}\, \widetilde{\psi}_2)$ where the maps $\widetilde{\psi}_1, \widetilde{\psi}_2$ and $\widetilde{\varphi}_1$ are defined as in 
\cite[page 31, the paragraph before Proposition 1]{Dra89}. 
The proof of \cite[Proposition 1]{Dra89} shows that 
this group is the same as ${\rm Obs}_1(\widetilde{L}/K/k)$ 
(see also \cite[Lemma 2, Lemma 3 and Lemma 4]{Dra89}).

For an unramified extension $\widetilde{L}/k$ with 
$\widetilde{G}={\rm Gal}(\widetilde{L}/k)$ a Scur cover of 
$G={\rm Gal}(L/k)$, 
Theorem \ref{thDP2}, Proposition \ref{propDP} (iii) and Theorem \ref{thDra89} 
give an explicit description of $H^1(k,{\rm Pic}\, \overline{X})\simeq H^1(G,{\rm Pic}\, X_K)$ 
$\simeq H^1(G,[\widehat{T}]^{fl})\simeq \Sha^2_\omega(G,\widehat{T})\simeq {\rm Br}(X)/{\rm Br}(k)\simeq 
{\rm Br}_{\rm nr}(k(X)/k)/{\rm Br}(k)$ with $\widehat{T}\simeq J_{G/H}$: 
\begin{align*}
H^1(k,{\rm Pic}\, \overline{X})\simeq 
{\rm Obs}(K/k)\simeq {\rm Obs}_1(\widetilde{L}/K/k)\simeq 
\frac{\widetilde{H}\cap[\widetilde{G},\widetilde{G}]}{\Phi^{\widetilde{G}}(\widetilde{H})}
=\frac{\widetilde{H}\cap[\widetilde{G},\widetilde{G}]}{\langle [h,x]\mid h\in \widetilde{H}\cap x\widetilde{H}x^{-1}, x\in \widetilde{G}\rangle}, 
\end{align*}
see the paragraph after Theorem \ref{thDP87} (resp. Theorem \ref{thCTS87}), 
see also Macedo and Newton \cite[Theorem 5.3]{MN22}. 
\\

In order to prove Theorem \ref{thmain2}, we use 
the functions of GAP (\cite{GAP}) which are given in 
Hoshi, Kanai and Yamasaki \cite[Section 6]{HKY22}, \cite[Section 6]{HKY23}. 
We made the following additional one for the case $G=16T1080$. 
The function {\tt StemExtensions($G$)[$j$].StemExtension} 
(resp. {\tt StemExtensions($G$)[$j$].epi}) returns the same as 
{\tt MinimalStemExtensions($G$)[$j$].MinimalStemExtension} 
(resp. {\tt MinimalStemExtensions($G$)[$j$].epi}) 
given in \cite[Section 6]{HKY22}, \cite[Section 6]{HKY23} 
but for any subgroup $A^\prime$ of $A=M(G)$ instead of 
only the ones with minimal $A/A^\prime$ 
where 
$M(G)=H^2(G,\bC^\times)$ is the Schur multiplier of $G$:\\

{\tt StemExtensions($G$)[$j$].StemExtension} 
(resp. {\tt StemExtensions($G$)[$j$].epi}) 
returns 
the $j$-th stem extension $\overline{G}=\widetilde{G}/A^\prime$, 
i.e. $\overline{A}\leq Z(\overline{G})\cap [\overline{G},\overline{G}]$, 
of $G$ provided by the Schur cover $\widetilde{G}$ of $G$ 
via {\tt SchurCoverG($G$).SchurCover} 
where $A^\prime$ is the $j$-th subgroup of $A=M(G)$
(resp. the surjective map $\overline{\pi}$) 
in the commutative diagram 
\begin{align*}
\begin{CD}
1 @>>> A=M(G) @>>> \widetilde{G}@>\pi>> G@>>> 1\\
  @. @VVV @VVV @|\\
1 @>>> \overline{A}=A/A^\prime @>>> \overline{G}=\widetilde{G}/A^\prime @>\overline{\pi}>> G@>>> 1
\end{CD}
\end{align*}
(see Robinson \cite[Exercises 11.4]{Rob96}). 
This function is based on the built-in function 
{\tt EpimorphismSchurCover} in GAP.\\

The related functions are available 
as in \cite{Norm1ToriHNP}. 

%
\section{{Proof of Theorem \ref{thmain2}}}\label{S6}

We prove Theorem \ref{thmain2} and Corollary \ref{cor1.5}.\\

{\it Proof of Theorem \ref{thmain2}.}\\ 

Let $G={\rm Gal}(L/k)=16Tm\leq S_{16}$ $(1\leq m\leq 1954)$ be 
the $m$-th transitive subgroup of $S_{16}$ and 
$H={\rm Gal}(L/K)\leq G$ with $[G:H]=16$. 
We assume that $G\leq S_{16}$ is primitive and 
$H^1(k,{\rm Pic}\, \overline{X})\neq 0$ 
where $T=R^{(1)}_{K/k}(\bG_m)$ and $X$ is a smooth $k$-compactification of $T$ 
and ${\rm Pic}\,\overline{X}$ is the Picard group of $\overline{X}=X\times_k\overline{k}$ 
as in Theorem \ref{thmain1} (Table $1$-$1$). 
Then we find that such $6$ primitive groups $G=16Tm\leq S_{16}$ with 
$m=178$, $708$, $1080$, $1329$, $1654$, $1753$ (Table $1$-$2$). 
Let $V_k$ be the set of all places of $k$ 
and $G_v$ be the decomposition group of $G$ at $v\in V_k$. 

We may assume that 
$H={\rm Stab}_1(G)$ is the stabilizer of $1$ in $G$, 
i.e. $L=k(\theta_1,\ldots,\theta_{16})$ and $K=L^H=k(\theta_1)$. 
Note that  (the multi-set) $\{{\rm Orb}_{G^\prime}(i)\mid 1\leq i\leq 16\}$ 
$(G^\prime\leq G)$ is invariant under the conjugacy actions of $G$, 
i.e. inner automorphisms of $G$.\\


{\rm (1)} Table $2$-$1$: $G=16Tm$ 
$(m=708$, $1329$, $1654$, $1753)$
with $\Sha(T)\leq \bZ/2\bZ$.\\ 

We split the $4$ cases $m=708$, $1329$, $1654$, $1753$ into $2$ parts (1-1), (1-2)  
according to the method to prove the assertion:\\ 

{\rm (1-1)} $m=708$ ($M(G)\simeq\bZ/2\bZ$ case). 
Applying {\tt FirstObstructionN($G$)} and {\tt FirstObstructionDnr($G$)}, 
we have ${\rm Obs}_1(L/K/k)=1$. 
Hence we just apply Theorem \ref{thDra89}. 
We have the Schur multiplier $M(G)\simeq\bZ/2\bZ$ 
for $G=16T708$. 
We obtain a Schur cover 
$1\to M(G)\simeq\bZ/2\bZ\to \widetilde{G}\xrightarrow{\pi} G\to 1$ 
and ${\rm Obs}(K/k)={\rm Obs}_1(\widetilde{L}/K/k)$. 
By Theorem \ref{thmain1}, we have 
${\rm Ker}\, \widetilde{\psi}_1/\widetilde{\varphi}_1({\rm Ker}\, \widetilde{\psi}_2^{\rm nr})\simeq \bZ/2\bZ$. 
Apply 
{\tt FirstObstructionDr($\widetilde{G},\widetilde{G}_{v_{r,s}},\widetilde{H}$)} to representatives of 
the orbit 
${\rm Orb}_{N_{\widetilde{G}}(\widetilde{H})\backslash \widetilde{G}/N_{\widetilde{G}}(\widetilde{G}_{v_{r,s}})}(\widetilde{G}_{v_{r,s}})$ 
of ${\widetilde{G}}_{v_{r,s}}\leq \widetilde{G}$ under the conjugate action of $\widetilde{G}$ 
which corresponds to the double coset 
$N_{\widetilde{G}}(\widetilde{H})\backslash \widetilde{G}/N_{\widetilde{G}}(\widetilde{G}_{v_{r,s}})$ 
with ${\rm Orb}_{\widetilde{G}/N_{\widetilde{G}}({\widetilde{G}}_{v_r})}({\widetilde{G}}_{v_r})$ 
$=$ $\bigcup_{s=1}^{u_r}{\rm Orb}_{N_{\widetilde{G}}(\widetilde{H})\backslash \widetilde{G}
/N_{\widetilde{G}}(\widetilde{G}_{v_{r,s}})}(\widetilde{G}_{v_{r,s}})$ 
corresponding to the $r$-th subgroup ${\widetilde{G}}_{v_r}\leq \widetilde{G}$ up to conjugacy 
via 
\begin{center}
{\tt ConjugacyClassesSubgroupsNGHOrbitRep(ConjugacyClassesSubgroups($\widetilde{G}$),$\widetilde{H}$)}
\end{center}
instead of applying it to all the representatives of the double coset $\widetilde{H}\backslash \widetilde{G}/\widetilde{G}_{v_{r,s}}$ 
which enable us to speed up and save memory resources. 
Then we can get the minimal elements of the ${\widetilde{G}}_{v_{r,s}}$'s 
with ${\rm Ker}\, \widetilde{\psi}_1/\widetilde{\varphi}_1({\rm Ker}\, \widetilde{\psi}_2)=0$ 
via 
\begin{center}
{\tt MinConjugacyClassesSubgroups($l$)}.
\end{center}
Finally, we get a necessary and sufficient condition for 
${\rm Obs}_1(\widetilde{L}/K/k)$ $=$ $1$ for each case 
(see Hoshi, Kanai and Yamasaki 
\cite[Section 6]{HKY22}, \cite[Section 7]{HKY23} and 
Example \ref{ex16-1}).

The last statement 
follows because the statements can be described in terms of 
the characteristic subgroups of $G$, i.e. 
invariants under the automorphisms of $G$. 
We can check this via 
{\tt IsInvariantUnderAutG($l$)} 
(see Hoshi, Kanai and Yamasaki \cite[Section 7]{HKY23}).\\

{\rm (1-2)} $m=1329$, $1654$, $1753$ (by using minimal stem extensions with $\overline{A}\simeq\bZ/2\bZ$). 
Applying 
{\tt FirstObstructionN($G$)} and {\tt FirstObstructionDnr($G$)}, 
we have ${\rm Obs}_1(L/K/k)=1$. 
For the $3$ cases 
$G=16Tm$ $(m=1329, 1654, 1753)$, 
we obtain that the Schur multipliers 
$M(G)\simeq \bZ/2\bZ\oplus\bZ/4\bZ$, 
$\bZ/2\bZ\oplus\bZ/4\bZ\oplus \bZ/3\bZ$, $(\bZ/2\bZ)^{\oplus 2}$ 
respectively. 
We can take a minimal stem extension 
$\overline{G}=\widetilde{G}/A^\prime$, 
i.e. $\overline{A}\leq Z(\overline{G})\cap [\overline{G},\overline{G}]$, 
of $G$ in the commutative diagram 
\begin{align*}
\begin{CD}
1 @>>> A=M(G) @>>> \widetilde{G}@>\pi>> G@>>> 1\\
  @. @VVV @VVV @|\\
1 @>>> \overline{A}=A/A^\prime @>>> \overline{G}=\widetilde{G}/A^\prime @>\overline{\pi}>> G@>>> 1
\end{CD}
\end{align*}
with $\overline{A}\simeq\bZ/2\bZ$ 
via 
{\tt MinimalStemExtensions($G$)[$j$].MinimalStemExtension}. 
Then we apply Theorem \ref{thDP87} instead of Theorem \ref{thDra89}. 
Applying 
{\tt KerResH3Z($G,H$:iterator)}, 
we can get  ${\rm Ker}\{H^3(\overline{G}_j,\bZ)\xrightarrow{\rm res}\oplus_{i=1}^{m^\prime} H^3(G_i,\bZ)\}=0$ 
for exactly one $j$ with $\overline{A}\simeq\bZ/2\bZ$. 
Because 
$\oplus_{i=1}^{m^\prime} \widehat{H}^{-3}(G_i,\bZ)\xrightarrow{\rm cores} 
\widehat{H}^{-3}(\overline{G}_j,\bZ)$ is surjective, 
it follows from Theorem \ref{thDP87} that 
${\rm Obs}(K/k)={\rm Obs}_1(\overline{L}_j/K/k)$. 
We also checked that 
${\rm Ker}\, \overline{\psi}_1/\overline{\varphi}_1({\rm Ker}\, \overline{\psi}_2^{\rm nr})\simeq\bZ/2\bZ$ for $\overline{G}_j$ and 
${\rm Ker}\, \overline{\psi}_1/\overline{\varphi}_1({\rm Ker}\, \overline{\psi}_2^{\rm nr})=0$ for $\overline{G}_{j^\prime}$ $(j^\prime\neq j)$. 
Hence 
${\rm Obs}(K/k)\neq {\rm Obs}_1(\overline{L}_{j^\prime}/K/k)$ 
(even if $\overline{L}_{j^\prime}/k$ is unramified). 
By applying 
{\tt FirstObstructionDr($\overline{G}_j,\overline{G}_{j,v_{r,s}},\overline{H}_j$)} and {\tt MinConjugacyClassesSubgroups($l$)} 
as in the case (1-1), 
we may find the minimal elements of the $\overline{G}_{j,v_{r,s}}$'s 
with ${\rm Ker}\, \overline{\psi}_1/\overline{\varphi}_1({\rm Ker}\, \overline{\psi}_2)=0$. 
Then we get a necessary and sufficient condition for 
${\rm Obs}_1(\overline{L}_j/K/k)$ $=$ $1$ for each case 
(see Hoshi, Kanai and Yamasaki 
\cite[Section 6]{HKY22}, \cite[Section 7]{HKY23} and 
Example \ref{ex16-1}).\\ 

{\rm (2)} Table $2$-$2$: $G=16Tm$ 
$(m=178$, $1080)$ with $\Sha(T)\leq (\bZ/2\bZ)^{\oplus 2}$.\\

{\rm (2-1)} $m=178$ ($M(G)\simeq (\bZ/2\bZ)^{\oplus 2}$ case). 
Applying {\tt FirstObstructionN($G$)} and {\tt FirstObstructionDnr($G$)}, 
we have ${\rm Obs}_1(L/K/k)=1$. 
If we intend to use {\tt MinimalStemExtensions($G$)[$j$].MinimalStemExtension}, then 
we can not get ${\rm Obs}(K/k)$ via ${\rm Obs}_1(\overline{L}_j/K/k)$ with $[\overline{L}_j:K]=2$ (cf. the case (1-2)). 
Hence we apply Theorem \ref{thDra89}. 
We have the Schur multiplier $M(G)\simeq (\bZ/2\bZ)^{\oplus 2}$ 
for $G=16T178$. 
We get a Schur cover 
$1\to M(G)\simeq (\bZ/2\bZ)^{\oplus 2}\to \widetilde{G}\xrightarrow{\pi} G\to 1$ 
and ${\rm Obs}(K/k)={\rm Obs}_1(\widetilde{L}/K/k)$. 
By Theorem \ref{thmain1}, we have 
${\rm Ker}\, \widetilde{\psi}_1/\widetilde{\varphi}_1({\rm Ker}\, \widetilde{\psi}_2^{\rm nr})\simeq (\bZ/2\bZ)^{\oplus 2}$. 
Applying 
{\tt FirstObstructionDr($\widetilde{G},\widetilde{G}_{v_{r,s}},\widetilde{H}$)} 
to representatives of 
the orbit 
${\rm Orb}_{N_{\widetilde{G}}(\widetilde{H})\backslash \widetilde{G}/N_{\widetilde{G}}(\widetilde{G}_{v_{r,s}})}(\widetilde{G}_{v_{r,s}})$ 
of ${\widetilde{G}}_{v_{r,s}}\leq \widetilde{G}$ under the conjugate action of $\widetilde{G}$ 
which corresponds to the double coset 
$N_{\widetilde{G}}(\widetilde{H})\backslash \widetilde{G}/N_{\widetilde{G}}(\widetilde{G}_{v_{r,s}})$ 
with ${\rm Orb}_{\widetilde{G}/N_{\widetilde{G}}({\widetilde{G}}_{v_r})}({\widetilde{G}}_{v_r})$ 
$=$ $\bigcup_{s=1}^{u_r}{\rm Orb}_{N_{\widetilde{G}}(\widetilde{H})\backslash \widetilde{G}
/N_{\widetilde{G}}(\widetilde{G}_{v_{r,s}})}(\widetilde{G}_{v_{r,s}})$ 
corresponding to the $r$-th subgroup ${\widetilde{G}}_{v_r}\leq \widetilde{G}$ up to conjugacy 
via 
\begin{center}
{\tt ConjugacyClassesSubgroupsNGHOrbitRep(ConjugacyClassesSubgroups($\widetilde{G}$),$\widetilde{H}$)}
\end{center}
instead of applying it to all the representatives of the double coset $\widetilde{H}\backslash \widetilde{G}/\widetilde{G}_{v_{r,s}}$ 
which enable us to speed up and save memory resources. 
Then we can get the minimal elements of the ${\widetilde{G}}_{v_{r,s}}$'s 
with ${\rm Ker}\, \widetilde{\psi}_1/\widetilde{\varphi}_1({\rm Ker}\, \widetilde{\psi}_2)=0$ 
via 
\begin{center}
{\tt MinConjugacyClassesSubgroups($l$)}.
\end{center}
We get a necessary and sufficient condition for 
${\rm Obs}_1(\widetilde{L}/K/k)$ $=$ $1$ for each case 
(see Hoshi, Kanai and Yamasaki 
\cite[Section 6]{HKY22}, \cite[Section 7]{HKY23} and 
Example \ref{ex16-1}).\\ 


{\rm (2-2)} $m=1080$ (by using stem extensions with $\overline{A}\simeq (\bZ/2\bZ)^{\oplus 2}$). 
Applying 
{\tt FirstObstructionN($G$)} and {\tt FirstObstructionDnr($G$)}, 
we have ${\rm Obs}_1(L/K/k)=1$. 
For the case $G=16T1080$, 
we obtain the Schur multiplier 
$M(G)\simeq \bZ/2\bZ\oplus (\bZ/4\bZ)^{\oplus 2}$. 
If we intend to use {\tt MinimalStemExtensions($G$)[$j$].MinimalStemExtension}, then 
we can not get ${\rm Obs}(K/k)$ via ${\rm Obs}_1(\overline{L}_j/K/k)$ with $[\overline{L}_j:K]=2$ (cf. the case (1-2)). 
Hence we take a stem extension 
$\overline{G}=\widetilde{G}/A^\prime$, 
i.e. $\overline{A}\leq Z(\overline{G})\cap [\overline{G},\overline{G}]$, 
of $G$ in the commutative diagram 
\begin{align*}
\begin{CD}
1 @>>> A=M(G) @>>> \widetilde{G}@>\pi>> G@>>> 1\\
  @. @VVV @VVV @|\\
1 @>>> \overline{A}=A/A^\prime @>>> \overline{G}=\widetilde{G}/A^\prime @>\overline{\pi}>> G@>>> 1
\end{CD}
\end{align*}
with $\overline{A}\simeq (\bZ/2\bZ)^{\oplus 2}$
via 
{\tt StemExtensions($G$)[$j$].StemExtension}. 
Then we apply Theorem \ref{thDP87} instead of Theorem \ref{thDra89}. 
Applying 
{\tt KerResH3Z($G,H$:iterator)}, 
we can get  ${\rm Ker}\{H^3(\overline{G}_j,\bZ)\xrightarrow{\rm res}\oplus_{i=1}^{m^\prime} H^3(G_i,\bZ)\}=0$ 
for exactly one $j$ with $\overline{A}\simeq (\bZ/2\bZ)^{\oplus 2}$. 
Because 
$\oplus_{i=1}^{m^\prime} \widehat{H}^{-3}(G_i,\bZ)\xrightarrow{\rm cores} 
\widehat{H}^{-3}(\overline{G}_j,\bZ)$ is surjective, 
it follows from Theorem \ref{thDP87} that 
${\rm Obs}(K/k)={\rm Obs}_1(\overline{L}_j/K/k)$. 
We also checked that 
${\rm Ker}\, \overline{\psi}_1/\overline{\varphi}_1({\rm Ker}\, \overline{\psi}_2^{\rm nr})\simeq (\bZ/2\bZ)^{\oplus 2}$ for $\overline{G}_j$ and 
${\rm Ker}\, \overline{\psi}_1/\overline{\varphi}_1({\rm Ker}\, \overline{\psi}_2^{\rm nr})\simeq \bZ/2\bZ$ or $0$ 
for $\overline{G}_{j^\prime}$ $(j^\prime\neq j)$. 
Hence 
${\rm Obs}(K/k)\neq {\rm Obs}_1(\overline{L}_{j^\prime}/K/k)$ 
(even if $\overline{L}_{j^\prime}/k$ is unramified). 
By applying 
{\tt FirstObstructionDr($\overline{G}_j,\overline{G}_{j,v_{r,s}},\overline{H}_j$)} and {\tt MinConjugacyClassesSubgroups($l$)} 
as in the case (2-1), 
we may find the minimal elements of the $\overline{G}_{j,v_{r,s}}$'s 
with ${\rm Ker}\, \overline{\psi}_1/\overline{\varphi}_1({\rm Ker}\, \overline{\psi}_2)=0$. 
Then we get a necessary and sufficient condition for 
${\rm Obs}_1(\overline{L}_j/K/k)$ $=$ $1$ for each case 
(see Hoshi, Kanai and Yamasaki 
\cite[Section 6]{HKY22}, \cite[Section 7]{HKY23} and 
Example \ref{ex16-1}).
\qed\\

\begin{remark}
In the proof of Theorem \ref{thmain2}, 
for $6$ primitive cases $G=16Tm\leq S_{16}$  
$(m=178$, $708$, $1080$, $1329$, $1654$, $1753)$, 
we showed that there exists exactly  one extension 
$\overline{L}_j/L$ with ${\rm Obs}(K/k)={\rm Obs}_1(\overline{L}_j/K/k)$ and $[\overline{L}_j:L]
=|H^1(G,[J_{G/H}]^{fl})|=|H^1(k,{\rm Pic}\, \overline{X})|$ 
where $\widetilde{L}\supset\overline{L}_j\supset L$ and ${\rm Gal}(\widetilde{L}/L)\simeq M(G)$. 
\end{remark}

\smallskip
\begin{example}[$G=16Tm$ $(m=178, 708, 1080, 1329, 1654, 1753)$]\label{ex16-1}~\\\vspace*{-2mm}

(1) $16T178\simeq (C_2)^4\rtimes C_5$. 
{\small 
\begin{verbatim}
gap> Read("HNP.gap");
gap> Read("FlabbyResolutionFromBase.gap");
gap> Filtered(H1(FlabbyResolutionNorm1TorusJ(16,178).actionF),x->x>1); # H^1(G,[J_{G/H}]^{fl})
[ 2, 2 ]
gap> G:=TransitiveGroup(16,178); # G=16T178
t16n178
gap> H:=Stabilizer(G,1);
Group([ (2,10,7,5,11)(3,15,8,9,12)(4,14,13,6,16) ])
gap> FirstObstructionN(G,H).ker;
[ [  ], [ [ 5 ], [  ] ] ]
gap> FirstObstructionDnr(G,H).Dnr;
[ [  ], [ [ 5 ], [  ] ] ]
gap> GroupCohomology(G,3); # H^3(G,Z)=M(G): Schur multiplier of G
[ 2, 2 ]
gap> cGs:=MinimalStemExtensions(G);; # MinimalStemExtensions does not work well
gap> for cG in cGs do
> bG:=cG.MinimalStemExtension;
> bH:=PreImage(cG.epi,H);
> Print(FirstObstructionN(bG,bH).ker[1]);
> Print(FirstObstructionDnr(bG,bH).Dnr[1]);
> Print("\n");
> od;
[ 2 ][  ]
[ 2 ][  ]
[ 2 ][  ]
gap> ScG:=SchurCoverG(G); # Apply SchurCoverG
rec( SchurCover := <permutation group of size 320 with 2 generators>, 
  epi := [ (2,4,9,12,5)(3,7,15,16,8)(10,19,33,34,20)(11,21,37,38,22)(13,25,43,
        46,26)(14,27,47,48,28)(17,29,49,51,30)(18,31,50,52,32)(23,39,57,58,
        40)(24,41,60,59,42)(35,53,44,61,54)(36,55,45,62,56), 
      (1,2,6,3)(4,10,7,11)(5,13,8,14)(9,17,15,18)(12,23,16,24)(19,26,21,
        28)(20,35,22,36)(25,44,27,45)(29,50,31,49)(30,33,32,37)(34,47,38,
        43)(39,48,41,46)(40,58,42,59)(51,56,52,54)(53,63,55,64)(57,62,60,61) 
     ] -> [ (2,7,11,10,5)(3,8,12,15,9)(4,13,16,14,6), 
      (1,11)(2,8)(3,9)(4,14)(5,15)(6,12)(7,13)(10,16) ] )
gap> tG:=ScG.SchurCover;
<permutation group of size 320 with 2 generators>
gap> StructureDescription(tG);
"((C2 x C2) . (C2 x C2 x C2 x C2)) : C5"
gap> tH:=PreImage(ScG.epi,H);
<permutation group of size 20 with 3 generators>
gap> FirstObstructionN(tG,tH).ker;
[ [ 2, 2 ], [ [ 2, 10 ], [ [ 1, 0 ], [ 0, 5 ] ] ] ]
gap> FirstObstructionDnr(tG,tH).Dnr;
[ [  ], [ [ 2, 10 ], [  ] ] ]
gap> HNPtruefalsefn:=x->FirstObstructionDr(tG,PreImage(ScG.epi,x),tH).Dr[2][2];
function( x ) ... end
gap> Gcs:=ConjugacyClassesSubgroups(G);;
gap> Length(Gcs);
17
gap> GcsH:=ConjugacyClassesSubgroupsNGHOrbitRep(Gcs,H);;
gap> GcsHHNPtf:=List(GcsH,x->List(x,HNPtruefalsefn));;
gap> Collected(List(GcsHHNPtf,Set));
[ [ [ [  ] ], 6 ], [ [ [ [ 0, 5 ] ] ], 2 ], [ [ [ [ 1, 0 ] ] ], 2 ], 
  [ [ [ [ 1, 0 ], [ 0, 5 ] ] ], 5 ], [ [ [ [ 1, 5 ] ] ], 2 ] ]
gap> GcsHNPfalse1:=List(Filtered([1..Length(Gcs)],
> x->[] in GcsHHNPtf[x]),y->Gcs[y]);;
gap> Length(GcsHNPfalse1);
6
gap> GcsHNPfalse2a:=List(Filtered([1..Length(Gcs)],
> x->[[0,5]] in GcsHHNPtf[x]),y->Gcs[y]);;
gap> Length(GcsHNPfalse2a);
2
gap> GcsHNPfalse2b:=List(Filtered([1..Length(Gcs)],
> x->[[1,0]] in GcsHHNPtf[x]),y->Gcs[y]);;
gap> Length(GcsHNPfalse2b);
2
gap> GcsHNPfalse2c:=List(Filtered([1..Length(Gcs)],
> x->[[1,5]] in GcsHHNPtf[x]),y->Gcs[y]);;
gap> Length(GcsHNPfalse2c);
2
gap> GcsHNPtrue:=List(Filtered([1..Length(Gcs)],
> x->[[1,0],[0,5]] in GcsHHNPtf[x]),y->Gcs[y]);;
gap> Length(GcsHNPtrue);
5
gap> Collected(List(GcsHNPfalse1,x->StructureDescription(Representative(x))));
[ [ "1", 1 ], [ "C2", 3 ], [ "C2 x C2", 1 ], [ "C5", 1 ] ]
gap> Collected(List(GcsHNPfalse2a,x->StructureDescription(Representative(x))));
[ [ "C2 x C2", 2 ] ]
gap> Collected(List(GcsHNPfalse2b,x->StructureDescription(Representative(x))));
[ [ "C2 x C2", 2 ] ]
gap> Collected(List(GcsHNPfalse2c,x->StructureDescription(Representative(x))));
[ [ "C2 x C2", 2 ] ]
gap> Collected(List(GcsHNPtrue,x->StructureDescription(Representative(x))));
[ [ "(C2 x C2 x C2 x C2) : C5", 1 ], [ "C2 x C2 x C2", 3 ], 
  [ "C2 x C2 x C2 x C2", 1 ] ]
gap> IsInvariantUnderAutG(GcsHNPfalse1);
true
gap> IsInvariantUnderAutG(GcsHNPfalse2a);
false
gap> IsInvariantUnderAutG(GcsHNPfalse2b);
false
gap> IsInvariantUnderAutG(GcsHNPfalse2c);
false
gap> IsInvariantUnderAutG([GcsHNPfalse2a[1]]);
false
gap> IsInvariantUnderAutG([GcsHNPfalse2a[2]]);
false
gap> IsInvariantUnderAutG([GcsHNPfalse2b[1]]);
false
gap> IsInvariantUnderAutG([GcsHNPfalse2b[2]]);
false
gap> IsInvariantUnderAutG([GcsHNPfalse2c[1]]);
false
gap> IsInvariantUnderAutG([GcsHNPfalse2c[2]]);
false
gap> List(GcsHNPfalse2a,x->Orbits(Representative(x),[1..16]));
[ [ [ 1, 10, 13, 6 ], [ 2, 9, 14, 5 ], [ 3, 8, 15, 4 ], [ 7, 12, 11, 16 ] ], 
  [ [ 1, 5, 12, 8 ], [ 2, 6, 15, 11 ], [ 3, 7, 14, 10 ], [ 4, 16, 9, 13 ] ] ]
gap> List(GcsHNPfalse2b,x->Orbits(Representative(x),[1..16]));
[ [ [ 1, 11, 7, 13 ], [ 2, 8, 4, 14 ], [ 3, 9, 5, 15 ], [ 6, 12, 16, 10 ] ], 
  [ [ 1, 3, 13, 15 ], [ 2, 16, 14, 12 ], [ 4, 6, 8, 10 ], [ 5, 7, 9, 11 ] ] ]
gap> List(GcsHNPfalse2c,x->Orbits(Representative(x),[1..16]));
[ [ [ 1, 10, 7, 12 ], [ 2, 9, 4, 15 ], [ 3, 8, 5, 14 ], [ 6, 13, 16, 11 ] ], 
  [ [ 1, 12, 13, 16 ], [ 2, 15, 14, 3 ], [ 4, 9, 8, 5 ], [ 6, 11, 10, 7 ] ] ]
gap> List(Elements(GcsHNPfalse2a[1]),x->Orbit(x,1));
[ [ 1, 14, 2, 13 ], [ 1, 16, 5, 4 ], [ 1, 4, 11, 14 ], [ 1, 6, 7, 16 ], 
  [ 1, 10, 13, 6 ] ]
gap> List(Elements(GcsHNPfalse2a[2]),x->Orbit(x,1));
[ [ 1, 2, 15, 12 ], [ 1, 8, 10, 3 ], [ 1, 3, 11, 9 ], [ 1, 5, 12, 8 ], 
  [ 1, 7, 9, 15 ] ]
gap> List(Elements(GcsHNPfalse2b[1]),x->Orbit(x,1));
[ [ 1, 2, 7, 4 ], [ 1, 5, 2, 6 ], [ 1, 10, 5, 14 ], [ 1, 11, 7, 13 ], 
  [ 1, 11, 10, 16 ] ]
gap> List(Elements(GcsHNPfalse2b[2]),x->Orbit(x,1));
[ [ 1, 14, 12, 3 ], [ 1, 3, 13, 15 ], [ 1, 4, 9, 12 ], [ 1, 6, 15, 8 ], 
  [ 1, 16, 8, 9 ] ]
gap> List(Elements(GcsHNPfalse2c[1]),x->Orbit(x,1));
[ [ 1, 2, 11, 8 ], [ 1, 10, 2, 9 ], [ 1, 5, 7, 3 ], [ 1, 5, 11, 15 ], 
  [ 1, 10, 7, 12 ] ]
gap> List(Elements(GcsHNPfalse2c[2]),x->Orbit(x,1));
[ [ 1, 6, 3, 4 ], [ 1, 4, 8, 13 ], [ 1, 14, 9, 6 ], [ 1, 12, 13, 16 ], 
  [ 1, 15, 16, 14 ] ]
gap> Collected(Concatenation(List(Elements(GcsHNPfalse2a[1]),x->Orbit(x,1))));
[ [ 1, 5 ], [ 2, 1 ], [ 4, 2 ], [ 5, 1 ], [ 6, 2 ], [ 7, 1 ], [ 10, 1 ], 
  [ 11, 1 ], [ 13, 2 ], [ 14, 2 ], [ 16, 2 ] ]
gap> Filtered(last,x->x[2]=1);
[ [ 2, 1 ], [ 5, 1 ], [ 7, 1 ], [ 10, 1 ], [ 11, 1 ] ]
gap> Collected(Concatenation(List(Elements(GcsHNPfalse2a[2]),x->Orbit(x,1))));
[ [ 1, 5 ], [ 2, 1 ], [ 3, 2 ], [ 5, 1 ], [ 7, 1 ], [ 8, 2 ], [ 9, 2 ], 
  [ 10, 1 ], [ 11, 1 ], [ 12, 2 ], [ 15, 2 ] ]
gap> Filtered(last,x->x[2]=1);
[ [ 2, 1 ], [ 5, 1 ], [ 7, 1 ], [ 10, 1 ], [ 11, 1 ] ]
gap> Collected(Concatenation(List(Elements(GcsHNPfalse2b[1]),x->Orbit(x,1))));
[ [ 1, 5 ], [ 2, 2 ], [ 4, 1 ], [ 5, 2 ], [ 6, 1 ], [ 7, 2 ], [ 10, 2 ], 
  [ 11, 2 ], [ 13, 1 ], [ 14, 1 ], [ 16, 1 ] ]
gap> Filtered(last,x->x[2]=1);
[ [ 4, 1 ], [ 6, 1 ], [ 13, 1 ], [ 14, 1 ], [ 16, 1 ] ]
gap> Collected(Concatenation(List(Elements(GcsHNPfalse2b[2]),x->Orbit(x,1))));
[ [ 1, 5 ], [ 3, 2 ], [ 4, 1 ], [ 6, 1 ], [ 8, 2 ], [ 9, 2 ], [ 12, 2 ], 
  [ 13, 1 ], [ 14, 1 ], [ 15, 2 ], [ 16, 1 ] ]
gap> Filtered(last,x->x[2]=1);
[ [ 4, 1 ], [ 6, 1 ], [ 13, 1 ], [ 14, 1 ], [ 16, 1 ] ]
gap> Collected(Concatenation(List(Elements(GcsHNPfalse2c[1]),x->Orbit(x,1))));
[ [ 1, 5 ], [ 2, 2 ], [ 3, 1 ], [ 5, 2 ], [ 7, 2 ], [ 8, 1 ], [ 9, 1 ], 
  [ 10, 2 ], [ 11, 2 ], [ 12, 1 ], [ 15, 1 ] ]
gap> Filtered(last,x->x[2]=1);
[ [ 3, 1 ], [ 8, 1 ], [ 9, 1 ], [ 12, 1 ], [ 15, 1 ] ]
gap> Collected(Concatenation(List(Elements(GcsHNPfalse2c[2]),x->Orbit(x,1))));
[ [ 1, 5 ], [ 3, 1 ], [ 4, 2 ], [ 6, 2 ], [ 8, 1 ], [ 9, 1 ], [ 12, 1 ], 
  [ 13, 2 ], [ 14, 2 ], [ 15, 1 ], [ 16, 2 ] ]
gap> Filtered(last,x->x[2]=1);
[ [ 3, 1 ], [ 8, 1 ], [ 9, 1 ], [ 12, 1 ], [ 15, 1 ] ]
gap> GcsHNPfalse1C2xC2:=Filtered(GcsHNPfalse1, # G(4,2)=C2xC2
> x->IdSmallGroup(Representative(x))=[4,2]);
[ Group( [ ( 1,12)( 2,15)( 3,14)( 4, 9)( 5, 8)( 6,11)( 7,10)(13,16), 
      ( 1,11)( 2, 8)( 3, 9)( 4,14)( 5,15)( 6,12)( 7,13)(10,16) ] )^G ]
gap> List(Elements(GcsHNPfalse1C2xC2[1]),x->Orbit(x,1));
[ [ 1, 2, 3, 16 ], [ 1, 15, 10, 4 ], [ 1, 5, 9, 13 ], [ 1, 12, 11, 6 ], 
  [ 1, 7, 8, 14 ] ]
gap> Collected(Concatenation(List(Elements(GcsHNPfalse1C2xC2[1]),x->Orbit(x,1))));
[ [ 1, 5 ], [ 2, 1 ], [ 3, 1 ], [ 4, 1 ], [ 5, 1 ], [ 6, 1 ], [ 7, 1 ], 
  [ 8, 1 ], [ 9, 1 ], [ 10, 1 ], [ 11, 1 ], [ 12, 1 ], [ 13, 1 ], [ 14, 1 ], 
  [ 15, 1 ], [ 16, 1 ] ]
\end{verbatim}
}~\\\vspace*{-4mm}

(2) $16T708\simeq (((C_2)^4\rtimes C_3)\rtimes C_2)\rtimes C_3$. 
{\small 
\begin{verbatim}
gap> Read("HNP.gap");
gap> Read("FlabbyResolutionFromBase.gap");
gap> Filtered(H1(FlabbyResolutionNorm1TorusJ(16,708).actionF),x->x>1); # H^1(G,[J_{G/H}]^{fl})
[ 2 ]
gap> G:=TransitiveGroup(16,708); # G=16T708
t16n708
gap> H:=Stabilizer(G,1);
Group([ (2,16)(4,14)(5,13)(6,12)(7,15)(8,10), (2,9,14,3,13,11)(4,16,5)
(6,12,15,7,8,10) ])
gap> FirstObstructionN(G,H).ker;
[ [  ], [ [ 6 ], [  ] ] ]
gap> FirstObstructionDnr(G,H).Dnr;
[ [  ], [ [ 6 ], [  ] ] ]
gap> GroupCohomology(G,3); # H^3(G,Z)=M(G): Schur multiplier of G
[ 2 ]
gap> ScG:=SchurCoverG(G); # Apply SchurCoverG
rec( SchurCover := <permutation group of size 576 with 4 generators>, 
  Tid := [ 24, 1482 ], 
  epi := [ (2,3)(5,8)(7,9)(10,14)(11,15)(12,16)(13,17)(18,19)(22,23), 
      (1,3,2)(4,8,5)(6,9,7)(10,15,18)(11,14,19)(12,16,20)(13,17,21)(22,23,24),
      (1,4,6)(2,5,7)(3,8,9)(12,13,22)(16,17,23)(20,21,24), 
      (1,10)(2,11)(4,6)(5,7)(8,12)(9,13)(14,24)(16,17)(18,23)(20,21) ] -> 
    [ (3,16)(4,11)(5,9)(6,10)(7,8)(12,15), 
      (2,3,16)(4,14,11)(5,13,9)(6,15,8)(7,12,10), 
      (2,13,14)(3,9,11)(4,16,5)(6,8,15)(7,12,10), 
      (1,11)(2,8)(3,9)(4,14)(5,15)(6,12)(7,13)(10,16) ] )
gap> tG:=ScG.SchurCover;
<permutation group of size 576 with 4 generators>
gap> StructureDescription(tG);
"(((C2 x C2 x C2) : (C2 x C2)) : (C3 x C3)) : C2"
gap> tH:=PreImage(ScG.epi,H);
Group([ (1,2)(4,5)(6,7)(10,11)(14,18)(15,19)(16,20)(17,21)(23,24), (1,8,6,3,4,
9)(2,5,7)(10,19)(11,18)(12,21,22,20,13,24)(14,15)(16,17,23), (1,24)(2,23)
(3,22)(4,20)(5,16)(6,21)(7,17)(8,12)(9,13)(10,14)(11,18)(15,19) ])
gap> FirstObstructionN(tG,tH).ker;
[ [ 2 ], [ [ 2, 6 ], [ [ 0, 3 ] ] ] ]
gap> FirstObstructionDnr(tG,tH).Dnr;
[ [  ], [ [ 2, 6 ], [  ] ] ]
gap> HNPtruefalsefn:=x->FirstObstructionDr(tG,PreImage(ScG.epi,x),tH).Dr[1]=[2];
function( x ) ... end
gap> Gcs:=ConjugacyClassesSubgroups(G);;
gap> Length(Gcs);
46
gap> GcsH:=ConjugacyClassesSubgroupsNGHOrbitRep(Gcs,H);;
gap> GcsHHNPtf:=List(GcsH,x->List(x,HNPtruefalsefn));;
gap> Collected(List(GcsHHNPtf,Set));
[ [ [ true ], 24 ], [ [ false ], 22 ] ]
gap> GcsHNPfalse:=List(Filtered([1..Length(Gcs)],
> x->false in GcsHHNPtf[x]),y->Gcs[y]);;
gap> Length(GcsHNPfalse);
22
gap> GcsHNPtrue:=List(Filtered([1..Length(Gcs)],
> x->true in GcsHHNPtf[x]),y->Gcs[y]);;
gap> Length(GcsHNPtrue);
24
gap> Collected(List(GcsHNPfalse,x->StructureDescription(Representative(x))));
[ [ "1", 1 ], [ "A4", 2 ], [ "C2", 3 ], [ "C2 x C2", 4 ], [ "C3", 3 ], 
  [ "C3 x C3", 1 ], [ "C3 x S3", 1 ], [ "C4", 1 ], [ "C6", 2 ], [ "D8", 2 ], 
  [ "S3", 1 ], [ "S4", 1 ] ]
gap> Collected(List(GcsHNPtrue,x->StructureDescription(Representative(x))));
[ [ "(((C2 x C2 x C2 x C2) : C3) : C2) : C3", 1 ], 
  [ "((C2 x C2 x C2 x C2) : C2) : C3", 1 ], 
  [ "((C2 x C2 x C2 x C2) : C3) : C2", 1 ], [ "(C2 x C2 x C2 x C2) : C2", 1 ],
  [ "(C2 x C2 x C2 x C2) : C3", 2 ], [ "(C4 x C2) : C2", 1 ], [ "A4", 3 ], 
  [ "A4 x A4", 1 ], [ "C2 x A4", 2 ], [ "C2 x C2", 2 ], [ "C2 x C2 x A4", 1 ],
  [ "C2 x C2 x C2", 3 ], [ "C2 x C2 x C2 x C2", 1 ], [ "C2 x D8", 1 ], 
  [ "C3 x A4", 1 ], [ "C4 x C2", 1 ], [ "C6 x C2", 1 ] ]
gap> GcsHNPtrueMin:=MinConjugacyClassesSubgroups(GcsHNPtrue);;
gap> Collected(List(GcsHNPtrueMin,x->StructureDescription(Representative(x))));
[ [ "C2 x C2", 2 ], [ "C2 x C2 x C2", 1 ], [ "C4 x C2", 1 ] ]
gap> IsInvariantUnderAutG(GcsHNPtrueMin);
true
gap> GcsHNPfalseC2xC2:=Filtered(GcsHNPfalse, # G(4,2)=C2xC2
> x->IdSmallGroup(Representative(x))=[4,2]);;
gap> Length(GcsHNPfalseC2xC2);
4
gap> GcsHNPtrueC2xC2:=Filtered(GcsHNPtrue, # G(4,2)=C2xC2
> x->IdSmallGroup(Representative(x))=[4,2]);;
gap> Length(GcsHNPtrueC2xC2);
2
gap> Collected(List(List(GcsHNPfalseC2xC2,Representative),
> x->SortedList(List(Orbits(x),Length))));
[ [ [ 2, 2, 2, 2, 4, 4 ], 1 ], [ [ 4, 4, 4, 4 ], 3 ] ]
gap> Collected(List(List(GcsHNPtrueC2xC2,Representative),
> x->SortedList(List(Orbits(x),Length))));
[ [ [ 4, 4, 4, 4 ], 2 ] ]
gap> DG:=DerivedSubgroup(G);
Group([ (1,14,11)(2,15,9)(3,13,10)(5,6,7)(8,16,12), (1,6,15)(2,7,13)(3,5,14)
(4,12,16)(9,10,11), (1,8,6)(2,9,4)(3,11,7)(5,16,10)(12,14,13) ])
gap> StructureDescription(DG);
"(C2 x C2 x C2 x C2) : C3"
gap> List(List(GcsHNPfalseC2xC2,Representative),x->Intersection(x,DG));
[ C2 x C2, C2 x C2, Group([ (1,14)(2,13)(3,12)(4,11)(5,10)(6,9)(7,8)
  (15,16) ]), C2 x C2 ]
gap> List(List(GcsHNPtrueC2xC2,Representative),x->Intersection(x,DG));
[ C2 x C2, C2 x C2 ]
gap> ZG:=Centre(G);
Group(())
gap> CharacteristicSubgroups(G);
[ t16n708, <permutation group of size 96 with 6 generators>, 
  <permutation group of size 144 with 6 generators>, Group([ (2,13,14)(3,9,11)
  (4,16,5)(6,8,15)(7,12,10), (1,11)(2,8)(3,9)(4,14)(5,15)(6,12)(7,13)
  (10,16), (1,2)(3,16)(4,7)(5,6)(8,11)(9,10)(12,15)(13,14), (1,9)(2,10)(3,11)
  (4,12)(5,13)(6,14)(7,15)(8,16), (1,14)(2,13)(3,12)(4,11)(5,10)(6,9)(7,8)
  (15,16) ]), Group([ (1,6,15)(2,7,13)(3,5,14)(4,12,16)(9,10,11), (1,11)(2,8)
  (3,9)(4,14)(5,15)(6,12)(7,13)(10,16), (1,2)(3,16)(4,7)(5,6)(8,11)(9,10)
  (12,15)(13,14), (1,9)(2,10)(3,11)(4,12)(5,13)(6,14)(7,15)(8,16), (1,14)
  (2,13)(3,12)(4,11)(5,10)(6,9)(7,8)(15,16) ]), Group([ (1,11)(2,8)(3,9)(4,14)
  (5,15)(6,12)(7,13)(10,16), (1,2)(3,16)(4,7)(5,6)(8,11)(9,10)(12,15)
  (13,14), (1,9)(2,10)(3,11)(4,12)(5,13)(6,14)(7,15)(8,16), (1,14)(2,13)(3,12)
  (4,11)(5,10)(6,9)(7,8)(15,16) ]), Group(()) ]
gap> List(last,Order);
[ 288, 96, 144, 48, 48, 16, 1 ]
gap> ChG:=CharacteristicSubgroups(G);;
gap> List(ChG,Order);
[ 288, 96, 144, 48, 48, 16, 1 ]
gap> Collected(List(List(GcsHNPfalseC2xC2,Representative),
> x->List(ChG,y->Order(Intersection(y,x)))));
[ [ [ 4, 4, 2, 2, 2, 2, 1 ], 1 ], [ [ 4, 4, 4, 4, 4, 4, 1 ], 3 ] ]
gap> Collected(List(List(GcsHNPtrueC2xC2,Representative),
> x->List(ChG,y->Order(Intersection(y,x)))));
[ [ [ 4, 4, 4, 4, 4, 4, 1 ], 2 ] ]
gap> Collected(List(List(GcsHNPfalseC2xC2,Representative),
> x->List(ChG,y->IdSmallGroup(Intersection(y,x)))));
[ [ [ [ 4, 2 ], [ 4, 2 ], [ 2, 1 ], [ 2, 1 ], [ 2, 1 ], [ 2, 1 ], [ 1, 1 ] ], 
      1 ], 
  [ [ [ 4, 2 ], [ 4, 2 ], [ 4, 2 ], [ 4, 2 ], [ 4, 2 ], [ 4, 2 ], [ 1, 1 ] ], 
      3 ] ]
gap> Collected(List(List(GcsHNPtrueC2xC2,Representative),
> x->List(ChG,y->IdSmallGroup(Intersection(y,x)))));
[ [ [ [ 4, 2 ], [ 4, 2 ], [ 4, 2 ], [ 4, 2 ], [ 4, 2 ], [ 4, 2 ], [ 1, 1 ] ], 
      2 ] ]
gap> List(GcsHNPfalseC2xC2,x->Length(Elements(x)));
# The number of conjugates of each V4
[ 3, 9, 18, 3 ]
gap> List(GcsHNPtrueC2xC2,x->Length(Elements(x)));
# The number of conjugates of each V4
[ 2, 18 ]
gap> IsInvariantUnderAutG([GcsHNPtrueC2xC2[1]]);
true
gap> IsInvariantUnderAutG([GcsHNPtrueC2xC2[2]]);
true
gap> IsInvariantUnderAutG([GcsHNPfalseC2xC2[1]]);
true
gap> IsInvariantUnderAutG([GcsHNPfalseC2xC2[2]]);
true
gap> IsInvariantUnderAutG([GcsHNPfalseC2xC2[3]]);
true
gap> IsInvariantUnderAutG([GcsHNPfalseC2xC2[4]]);
true
gap> List(List(GcsHNPfalseC2xC2,Representative),
> x->List(ChG,y->Order(Intersection(y,x))));
[ [ 4, 4, 4, 4, 4, 4, 1 ], [ 4, 4, 4, 4, 4, 4, 1 ], [ 4, 4, 2, 2, 2, 2, 1 ], 
  [ 4, 4, 4, 4, 4, 4, 1 ] ]
gap> List(List(GcsHNPtrueC2xC2,Representative),
> x->List(ChG,y->Order(Intersection(y,x))));
[ [ 4, 4, 4, 4, 4, 4, 1 ], [ 4, 4, 4, 4, 4, 4, 1 ] ]
\end{verbatim}
}~\\\vspace*{-4mm}

(3) $16T1080\simeq (C_2)^4\rtimes A_5$. 
{\small 
\begin{verbatim}
gap> Read("HNP.gap");
gap> Read("FlabbyResolutionFromBase.gap");
gap> Filtered(H1(FlabbyResolutionNorm1TorusJ(16,1080).actionF),x->x>1); # H^1(G,[J_{G/H}]^{fl})
[ 2, 2 ]
gap> G:=TransitiveGroup(16,1080); # G=16T1080
t16n1080
gap> H:=Stabilizer(G,1);
Group([ (2,13,10,8,12)(3,5,15,14,6)(4,7,11,16,9), (2,4,12,14,5)(3,10,6,7,9)
(8,13,16,15,11) ])
gap> FirstObstructionN(G,H).ker;
[ [  ], [ [  ], [  ] ] ]
gap> FirstObstructionDnr(G,H).Dnr;
[ [  ], [ [  ], [  ] ] ]
gap> GroupCohomology(G,3); # H^3(G,Z)=M(G): Schur multiplier of G
[ 2, 4, 4 ]
gap> cGs:=MinimalStemExtensions(G);; # MiniamStemExtensions does not work well
gap> for cG in cGs do
> bG:=cG.MinimalStemExtension;
> bH:=PreImage(cG.epi,H);
> Print(FirstObstructionN(bG,bH).ker[1]);
> Print(FirstObstructionDnr(bG,bH).Dnr[1]);
> Print("\n");
> od;
[ 2 ][  ]
[ 2 ][  ]
[ 2 ][  ]
[  ][  ]
[  ][  ]
[  ][  ]
[  ][  ]
gap> cGs:=StemExtensions(G);; # Apply StemExtensions with |A^-|=4
gap> cGs4:=Filtered(cGs,x->Order(x.StemExtension)=4*Order(G));;
gap> Length(cGs4);
19
gap> for cG in cGs4 do
> bG:=cG.StemExtension;
> bH:=PreImage(cG.epi,H);
> Print(FirstObstructionN(bG,bH).ker[1]);
> Print(FirstObstructionDnr(bG,bH).Dnr[1]);
> Print("\n");
> od;
[ 2 ][  ]
[ 2 ][  ]
[ 2, 2 ][  ]
[ 2 ][  ]
[ 2 ][  ]
[ 2 ][  ]
[ 2 ][  ]
[ 2 ][  ]
[ 4 ][ 2 ]
[ 4 ][ 2 ]
[ 2 ][  ]
[ 4 ][ 2 ]
[ 2 ][  ]
[ 2 ][  ]
[ 4 ][ 2 ]
[ 2 ][  ]
[ 4 ][ 2 ]
[ 2 ][  ]
[ 4 ][ 2 ]
gap> cG:=cGs4[3];; # there exists i=3 such that Obs(K/k)=Obs_1(L^-_i/K/k)
gap> bG:=cG.StemExtension;
Group([ (1,9,12,22,13)(2,4,5,7,11)(6,15,20,16,8)(14,19,24,21,17), (1,2,7,9,3)
(4,13,12,5,10)(6,17,24,15,18)(8,20,19,23,14), (1,14)(3,10)(4,6)(5,9)(7,20)
(12,24)(15,19)(18,23) ])
gap> bH:=PreImage(cG.epi,H);
Group([ (1,10,22,5,7)(3,11,9,12,4)(6,23,16,19,24)(14,18,21,15,20), (1,23,12,
17,11,6,10,20,2,21)(3,24,13,16,4,18,7,8,22,14)(5,19)(9,15), (1,4)(2,13)(3,10)
(5,9)(6,14)(7,12)(8,17)(11,22)(15,19)(16,21)(18,23)(20,24) ])
gap> KerResH3Z(bG,bH:iterator); # KerResH3Z=0 => Obs(K/k)=Obs_1(L^-_3/K/k)
[ [  ], [ [ 2, 2, 2 ], [  ] ] ]
gap> FirstObstructionN(bG,bH).ker;
[ [ 2, 2 ], [ [ 2, 2 ], [ [ 1, 0 ], [ 0, 1 ] ] ] ]
gap> FirstObstructionDnr(bG,bH).Dnr;
[ [  ], [ [ 2, 2 ], [  ] ] ]
gap> HNPtruefalsefn:=x->FirstObstructionDr(bG,PreImage(cG.epi,x),bH).Dr[2][2];
function( x ) ... end
gap> Gcs:=ConjugacyClassesSubgroups(G);;
gap> Length(Gcs);
77
gap> GcsH:=ConjugacyClassesSubgroupsNGHOrbitRep(Gcs,H);;
gap> GcsHHNPtf:=List(GcsH,x->List(x,HNPtruefalsefn));;
gap> Collected(List(GcsHHNPtf,Set));
[ [ [ [  ] ], 52 ], [ [ [ [ 0, 1 ] ] ], 6 ], [ [ [ [ 1, 0 ] ] ], 6 ], 
  [ [ [ [ 1, 0 ], [ 0, 1 ] ] ], 7 ], [ [ [ [ 1, 1 ] ] ], 6 ] ]
gap> GcsHNPfalse1:=List(Filtered([1..Length(Gcs)],
> x->[] in GcsHHNPtf[x]),y->Gcs[y]);;
gap> Length(GcsHNPfalse1);
52
gap> GcsHNPfalse2a:=List(Filtered([1..Length(Gcs)],
> x->[[0,1]] in GcsHHNPtf[x]),y->Gcs[y]);;
gap> Length(GcsHNPfalse2a);
6
gap> GcsHNPfalse2b:=List(Filtered([1..Length(Gcs)],
> x->[[1,0]] in GcsHHNPtf[x]),y->Gcs[y]);;
gap> Length(GcsHNPfalse2b);
6
gap> GcsHNPfalse2c:=List(Filtered([1..Length(Gcs)],
> x->[[1,1]] in GcsHHNPtf[x]),y->Gcs[y]);;
gap> Length(GcsHNPfalse2c);
6
gap> GcsHNPtrue:=List(Filtered([1..Length(Gcs)],
> x->[[1,0],[0,1]] in GcsHHNPtf[x]),y->Gcs[y]);;
gap> Length(GcsHNPtrue);
7
gap> Collected(List(GcsHNPfalse1,x->StructureDescription(Representative(x))));
[ [ "((C2 x C2 x C2 x C2) : C3) : C2", 1 ], 
  [ "((C2 x C2 x C2 x C2) : C5) : C2", 1 ], [ "(C2 x C2 x C2 x C2) : C2", 1 ],
  [ "(C2 x C2 x C2 x C2) : C3", 1 ], [ "(C2 x C2 x C2 x C2) : C5", 1 ], 
  [ "(C4 x C2) : C2", 6 ], [ "1", 1 ], [ "A4", 5 ], [ "A5", 1 ], [ "C2", 2 ], 
  [ "C2 x C2", 8 ], [ "C2 x C2 x C2", 2 ], [ "C2 x C2 x C2 x C2", 1 ], 
  [ "C2 x D8", 3 ], [ "C2 x Q8", 1 ], [ "C3", 1 ], [ "C4", 3 ], 
  [ "C4 x C2", 3 ], [ "C5", 1 ], [ "D10", 1 ], [ "D8", 3 ], [ "Q8", 1 ], 
  [ "S3", 1 ], [ "S4", 3 ] ]
gap> Collected(List(GcsHNPfalse2a,x->StructureDescription(Representative(x))));
[ [ "(C4 x C4) : C2", 1 ], [ "(C4 x C4) : C3", 1 ], [ "A4", 1 ], [ "A5", 1 ], 
  [ "C2 x C2", 1 ], [ "C4 x C4", 1 ] ]
gap> Collected(List(GcsHNPfalse2b,x->StructureDescription(Representative(x))));
[ [ "(C4 x C4) : C2", 1 ], [ "(C4 x C4) : C3", 1 ], [ "A4", 1 ], [ "A5", 1 ], 
  [ "C2 x C2", 1 ], [ "C4 x C4", 1 ] ]
gap> Collected(List(GcsHNPfalse2c,x->StructureDescription(Representative(x))));
[ [ "(C4 x C4) : C2", 1 ], [ "(C4 x C4) : C3", 1 ], [ "A4", 1 ], [ "A5", 1 ], 
  [ "C2 x C2", 1 ], [ "C4 x C4", 1 ] ]
gap> Collected(List(GcsHNPtrue,x->StructureDescription(Representative(x))));
[ [ "(((C2 x C2 x C2 x C2) : C2) : C2) : C3", 1 ], 
  [ "((C2 x C2 x C2 x C2) : C2) : C2", 1 ], [ "(C2 x C2 x C2 x C2) : A5", 1 ],
  [ "(C2 x C2 x C2 x C2) : C2", 1 ], [ "(C2 x C2 x C2 x C2) : C3", 1 ], 
  [ "C2 x C2 x C2", 1 ], [ "C2 x C2 x C2 x C2", 1 ] ]
gap> IsInvariantUnderAutG(GcsHNPtrue);
true
gap> GcsHNPtrueMin:=MinConjugacyClassesSubgroups(GcsHNPtrue);
[ Group( [ ( 1, 3)( 2,16)( 4, 6)( 5, 7)( 8,10)( 9,11)(12,14)(13,15), 
      ( 1, 3)( 2,16)( 8,11)( 9,10)(12,13)(14,15), 
      ( 4, 5)( 6, 7)( 8,11)( 9,10)(12,14)(13,15) ] )^G ]
gap> Collected(List(GcsHNPtrueMin,x->StructureDescription(Representative(x))));
[ [ "C2 x C2 x C2", 1 ] ]
gap> GcsHNPfalse1C2xC2xC2:=Filtered(GcsHNPfalse1, # G(8,5)=C2xC2xC2
> x->IdSmallGroup(Representative(x))=[8,5]);;
gap> Length(GcsHNPfalse1C2xC2xC2);
2
gap> GcsHNPfalse2aC2xC2xC2:=Filtered(GcsHNPfalse2a, # G(8,5)=C2xC2xC2
> x->IdSmallGroup(Representative(x))=[8,5]);;
gap> Length(GcsHNPfalse2aC2xC2xC2);
0
gap> GcsHNPfalse2bC2xC2xC2:=Filtered(GcsHNPfalse2b, # G(8,5)=C2xC2xC2
> x->IdSmallGroup(Representative(x))=[8,5]);;
gap> Length(GcsHNPfalse2bC2xC2xC2);
0
gap> GcsHNPfalse2cC2xC2xC2:=Filtered(GcsHNPfalse2c, # G(8,5)=C2xC2xC2
> x->IdSmallGroup(Representative(x))=[8,5]);;
gap> Length(GcsHNPfalse2cC2xC2xC2);
0
gap> GcsHNPtrueC2xC2xC2:=Filtered(GcsHNPtrue, # G(8,5)=C2xC2xC2
> x->IdSmallGroup(Representative(x))=[8,5]);;
gap> Length(GcsHNPtrueC2xC2xC2);
1
gap> IsInvariantUnderAutG(GcsHNPfalse1C2xC2xC2);
true
gap> IsInvariantUnderAutG(GcsHNPtrueC2xC2xC2);
true
gap> ChG:=CharacteristicSubgroups(G);
[ Group(()), Group([ (1,2)(3,16)(4,7)(5,6)(8,11)(9,10)(12,15)(13,14), (1,5)
  (2,6)(3,7)(4,16)(8,12)(9,13)(10,14)(11,15), (1,7)(2,4)(3,5)(6,16)(8,14)
  (9,15)(10,12)(11,13), (1,8)(2,11)(3,10)(4,13)(5,12)(6,15)(7,14)(9,16) ]), 
  t16n1080 ]
gap> List(ChG,Order);
[ 1, 16, 960 ]
gap> E16:=ChG[2];;
gap> List(GcsHNPfalse1C2xC2xC2,
> x->Order(Intersection(Representative(x),E16)));
[ 4, 8 ]
gap> List(GcsHNPtrueC2xC2xC2,
> x->Order(Intersection(Representative(x),E16)));
[ 2 ]
gap> IsInvariantUnderAutG(GcsHNPfalse1);
false
gap> IsInvariantUnderAutG(GcsHNPfalse2a);
false
gap> IsInvariantUnderAutG(GcsHNPfalse2b);
false
gap> IsInvariantUnderAutG(GcsHNPfalse2c);
false
gap> GcsHNPfalse2aMin:=MinConjugacyClassesSubgroups(GcsHNPfalse2a);
[ Group( [ ( 1, 3)( 2,16)( 8,11)( 9,10)(12,13)(14,15), 
      ( 1, 3)( 2,16)( 4, 7)( 5, 6)( 8, 9)(10,11) ] )^G, 
  Group( [ ( 1, 3)( 2,16)( 4, 6)( 5, 7)( 8,10)( 9,11)(12,14)(13,15), 
      ( 1,16)( 2, 3)( 4, 5)( 6, 7)( 8, 9)(10,11)(12,13)(14,15), 
      ( 1,13, 2,14)( 3,15,16,12)( 4,10, 7, 9)( 5,11, 6, 8), 
      ( 1, 5,16, 4)( 2, 6, 3, 7)( 8,15, 9,14)(10,13,11,12) ] )^G ]
gap> GcsHNPfalse2bMin:=MinConjugacyClassesSubgroups(GcsHNPfalse2b);
[ Group( [ ( 1, 3)( 2,16)( 8,11)( 9,10)(12,13)(14,15), 
      ( 4, 5)( 6, 7)( 8,11)( 9,10)(12,14)(13,15) ] )^G, 
  Group( [ ( 1, 3)( 2,16)( 4, 6)( 5, 7)( 8,10)( 9,11)(12,14)(13,15), 
      ( 1,16)( 2, 3)( 4, 5)( 6, 7)( 8, 9)(10,11)(12,13)(14,15), 
      ( 1, 9,16, 8)( 2,10, 3,11)( 4,14, 5,15)( 6,12, 7,13), 
      ( 1,15, 3,13)( 2,12,16,14)( 4,11, 6, 9)( 5,10, 7, 8) ] )^G ]
gap> GcsHNPfalse2cMin:=MinConjugacyClassesSubgroups(GcsHNPfalse2c);
[ Group( [ ( 1, 3)( 2,16)( 8,11)( 9,10)(12,13)(14,15), 
      ( 1, 2)( 3,16)( 4, 6)( 5, 7)(12,13)(14,15) ] )^G, 
  Group( [ ( 1, 3)( 2,16)( 4, 6)( 5, 7)( 8,10)( 9,11)(12,14)(13,15), 
      ( 1,16)( 2, 3)( 4, 5)( 6, 7)( 8, 9)(10,11)(12,13)(14,15), 
      ( 1, 7, 3, 5)( 2, 4,16, 6)( 8,15,10,13)( 9,14,11,12), 
      ( 1,11, 2, 8)( 3, 9,16,10)( 4,15, 7,12)( 5,14, 6,13) ] )^G ]
gap> Collected(List(GcsHNPfalse2aMin,x->StructureDescription(Representative(x))));
[ [ "C2 x C2", 1 ], [ "C4 x C4", 1 ] ]
gap> Collected(List(GcsHNPfalse2bMin,x->StructureDescription(Representative(x))));
[ [ "C2 x C2", 1 ], [ "C4 x C4", 1 ] ]
gap> Collected(List(GcsHNPfalse2cMin,x->StructureDescription(Representative(x))));
[ [ "C2 x C2", 1 ], [ "C4 x C4", 1 ] ]
gap> GcsHNPfalse1C2xC2:=Filtered(GcsHNPfalse1, # G(4,2)=C2xC2
> x->IdSmallGroup(Representative(x))=[4,2]);;
gap> Length(GcsHNPfalse1C2xC2);
8
gap> GcsHNPfalse2aC2xC2:=Filtered(GcsHNPfalse2a, # G(4,2)=C2xC2
> x->IdSmallGroup(Representative(x))=[4,2]);;
gap> Length(GcsHNPfalse2aC2xC2);
1
gap> GcsHNPfalse2bC2xC2:=Filtered(GcsHNPfalse2b, # G(4,2)=C2xC2
> x->IdSmallGroup(Representative(x))=[4,2]);;
gap> Length(GcsHNPfalse2bC2xC2);
1
gap> GcsHNPfalse2cC2xC2:=Filtered(GcsHNPfalse2c, # G(4,2)=C2xC2
> x->IdSmallGroup(Representative(x))=[4,2]);;
gap> Length(GcsHNPfalse2cC2xC2);
1
gap> IsInvariantUnderAutG(GcsHNPfalse1C2xC2);
false
gap> IsInvariantUnderAutG(GcsHNPfalse2aC2xC2);
false
gap> IsInvariantUnderAutG(GcsHNPfalse2bC2xC2);
false
gap> IsInvariantUnderAutG(GcsHNPfalse2cC2xC2);
false
gap> Collected(List(List(GcsHNPfalse1C2xC2,Representative),
> x->SortedList(List(Orbits(x),Length))));
[ [ [ 2, 2, 2, 2, 4, 4 ], 3 ], [ [ 4, 4, 4 ], 1 ], [ [ 4, 4, 4, 4 ], 4 ] ]
gap> Collected(List(List(GcsHNPfalse2aC2xC2,Representative),
> x->SortedList(List(Orbits(x),Length))));
[ [ [ 2, 2, 2, 2, 2, 2, 4 ], 1 ] ]
gap> Collected(List(List(GcsHNPfalse2bC2xC2,Representative),
> x->SortedList(List(Orbits(x),Length))));
[ [ [ 2, 2, 2, 2, 2, 2, 4 ], 1 ] ]
gap> Collected(List(List(GcsHNPfalse2cC2xC2,Representative),
> x->SortedList(List(Orbits(x),Length))));
[ [ [ 2, 2, 2, 2, 2, 2, 4 ], 1 ] ]
gap> GcsHNPfalse2a16:=Filtered(GcsHNPfalse2a,x->Order(Representative(x)) mod 16=0);;
gap> GcsHNPfalse2b16:=Filtered(GcsHNPfalse2b,x->Order(Representative(x)) mod 16=0);;
gap> GcsHNPfalse2c16:=Filtered(GcsHNPfalse2c,x->Order(Representative(x)) mod 16=0);;
gap> Length(GcsHNPfalse2a16);
3
gap> Length(GcsHNPfalse2b16);
3
gap> Length(GcsHNPfalse2c16);
3
gap> GcsHNPfalse2a16sC2xC2:=List(GcsHNPfalse2a16, # G(4,2)=C2xC2
> x->Filtered(ConjugacyClassesSubgroups(Representative(x)),
> y->IdSmallGroup(Representative(y))=[4,2]));;
gap> GcsHNPfalse2b16sC2xC2:=List(GcsHNPfalse2b16, # G(4,2)=C2xC2
> x->Filtered(ConjugacyClassesSubgroups(Representative(x)),
> y->IdSmallGroup(Representative(y))=[4,2]));;
gap> GcsHNPfalse2c16sC2xC2:=List(GcsHNPfalse2c16, # G(4,2)=C2xC2
> x->Filtered(ConjugacyClassesSubgroups(Representative(x)),
> y->IdSmallGroup(Representative(y))=[4,2]));;
gap> List(GcsHNPfalse2a16sC2xC2,Length);
[ 1, 7, 1 ]
gap> List(GcsHNPfalse2b16sC2xC2,Length);
[ 1, 7, 1 ]
gap> List(GcsHNPfalse2c16sC2xC2,Length);
[ 1, 7, 1 ]
gap> List(GcsHNPfalse2a16sC2xC2,x->Collected(List(List(x,Representative),
> y->SortedList(List(Orbits(y),Length)))));
[ [ [ [ 4, 4, 4, 4 ], 1 ] ], 
  [ [ [ 2, 2, 2, 2, 4, 4 ], 3 ], [ [ 4, 4, 4, 4 ], 4 ] ], 
  [ [ [ 4, 4, 4, 4 ], 1 ] ] ]
gap> List(GcsHNPfalse2b16sC2xC2,x->Collected(List(List(x,Representative),
> y->SortedList(List(Orbits(y),Length)))));
[ [ [ [ 4, 4, 4, 4 ], 1 ] ], 
  [ [ [ 2, 2, 2, 2, 4, 4 ], 3 ], [ [ 4, 4, 4, 4 ], 4 ] ], 
  [ [ [ 4, 4, 4, 4 ], 1 ] ] ]
gap> List(GcsHNPfalse2c16sC2xC2,x->Collected(List(List(x,Representative),
> y->SortedList(List(Orbits(y),Length)))));
[ [ [ [ 4, 4, 4, 4 ], 1 ] ], 
  [ [ [ 2, 2, 2, 2, 4, 4 ], 3 ], [ [ 4, 4, 4, 4 ], 4 ] ], 
  [ [ [ 4, 4, 4, 4 ], 1 ] ] ]
gap> GcsHNPfalse1C4xC4:=Filtered(GcsHNPfalse1, # G(16,2)=C4xC4
> x->IdSmallGroup(Representative(x))=[16,2]);;
gap> Length(GcsHNPfalse1C4xC4);
0
gap> GcsHNPfalse2aC4xC4:=Filtered(GcsHNPfalse2a, # G(16,2)=C4xC4
> x->IdSmallGroup(Representative(x))=[16,2]);;
gap> Length(GcsHNPfalse2aC4xC4);
1
gap> GcsHNPfalse2bC4xC4:=Filtered(GcsHNPfalse2b, # G(16,2)=C4xC4
> x->IdSmallGroup(Representative(x))=[16,2]);;
gap> Length(GcsHNPfalse2bC4xC4);
1
gap> GcsHNPfalse2cC4xC4:=Filtered(GcsHNPfalse2c, # G(16,2)=C4xC4
> x->IdSmallGroup(Representative(x))=[16,2]);;
gap> Length(GcsHNPfalse2cC4xC4);
1
gap> IsInvariantUnderAutG(GcsHNPfalse2aC4xC4);
false
gap> IsInvariantUnderAutG(GcsHNPfalse2bC4xC4);
false
gap> IsInvariantUnderAutG(GcsHNPfalse2cC4xC4);
false
gap> G2a2c:=Elements(GcsHNPfalse2aC2xC2[1]);;
gap> G2b2c:=Elements(GcsHNPfalse2bC2xC2[1]);;
gap> G2c2c:=Elements(GcsHNPfalse2cC2xC2[1]);;
gap> Length(G2a2c);
20
gap> Length(G2b2c);
20
gap> Length(G2c2c);
20
gap> G2a4c:=Elements(GcsHNPfalse2aC4xC4[1]);;
gap> G2b4c:=Elements(GcsHNPfalse2bC4xC4[1]);;
gap> G2c4c:=Elements(GcsHNPfalse2cC4xC4[1]);;
gap> Length(G2a4c);
5
gap> Length(G2b4c);
5
gap> Length(G2c4c);
5
gap> NS16G:=Normalizer(SymmetricGroup(16),G);
Group([ (1,6,2,9,8)(3,16,15,5,13)(7,11,14,12,10), (1,5,13)(2,4,15)(3,6,12)
(7,14,16)(8,11,10), (2,14)(3,8)(4,15)(5,9)(7,16)(11,12), (2,10)(3,4)(5,13)
(7,8)(11,12)(15,16) ])
gap> Order(NS16G);
5760
gap> Collected(List(G2a2c,x->SortedList(List(G2a4c,
> y->Order(Intersection(Normalizer(NS16G,x),Normalizer(NS16G,y)))))));
[ [ [ 24, 24, 24, 24, 96 ], 20 ] ]
gap> Collected(List(G2a2c,x->SortedList(List(G2b4c,
> y->Order(Intersection(Normalizer(NS16G,x),Normalizer(NS16G,y)))))));
[ [ [ 12, 12, 12, 12, 48 ], 20 ] ]
gap> Collected(List(G2a2c,x->SortedList(List(G2c4c,
> y->Order(Intersection(Normalizer(NS16G,x),Normalizer(NS16G,y)))))));
[ [ [ 12, 12, 12, 12, 48 ], 20 ] ]
gap> Collected(List(G2b2c,x->SortedList(List(G2a4c,
> y->Order(Intersection(Normalizer(NS16G,x),Normalizer(NS16G,y)))))));
[ [ [ 12, 12, 12, 12, 48 ], 20 ] ]
gap> Collected(List(G2b2c,x->SortedList(List(G2b4c,
> y->Order(Intersection(Normalizer(NS16G,x),Normalizer(NS16G,y)))))));
[ [ [ 24, 24, 24, 24, 96 ], 20 ] ]
gap> Collected(List(G2b2c,x->SortedList(List(G2c4c,
> y->Order(Intersection(Normalizer(NS16G,x),Normalizer(NS16G,y)))))));
[ [ [ 12, 12, 12, 12, 48 ], 20 ] ]
gap> Collected(List(G2c2c,x->SortedList(List(G2a4c,
> y->Order(Intersection(Normalizer(NS16G,x),Normalizer(NS16G,y)))))));
[ [ [ 12, 12, 12, 12, 48 ], 20 ] ]
gap> Collected(List(G2c2c,x->SortedList(List(G2b4c,
> y->Order(Intersection(Normalizer(NS16G,x),Normalizer(NS16G,y)))))));
[ [ [ 12, 12, 12, 12, 48 ], 20 ] ]
gap> Collected(List(G2c2c,x->SortedList(List(G2c4c,
> y->Order(Intersection(Normalizer(NS16G,x),Normalizer(NS16G,y)))))));
[ [ [ 24, 24, 24, 24, 96 ], 20 ] ]
\end{verbatim}
}~\\\vspace*{-4mm}

(4) $16T1329\simeq (C_2)^4\rtimes S_5$. 
{\small 
\begin{verbatim}
gap> Read("HNP.gap");
gap> Read("FlabbyResolutionFromBase.gap");
gap> Filtered(H1(FlabbyResolutionNorm1TorusJ(16,1329).actionF),x->x>1); # H^1(G,[J_{G/H}]^{fl})
[ 2 ]
gap> G:=TransitiveGroup(16,1329); # G=16T1329
t16n1329
gap> H:=Stabilizer(G,1);
Group([ (2,5,10,11,7)(3,9,15,12,8)(4,6,14,16,13), (2,12,5,4,14)(3,6,9,10,7)
(8,16,11,13,15), (3,16)(5,6)(8,14)(9,12)(10,15)(11,13) ])
gap> FirstObstructionN(G,H).ker;
[ [  ], [ [ 2 ], [  ] ] ]
gap> FirstObstructionDnr(G,H).Dnr;
[ [  ], [ [ 2 ], [  ] ] ]
gap> GroupCohomology(G,3); # H^3(G,Z)=M(G): Schur multiplier of G
[ 2, 4 ]
gap> cGs:=MinimalStemExtensions(G);;
gap> for cG in cGs do
> bG:=cG.MinimalStemExtension;
> bH:=PreImage(cG.epi,H);
> Print(FirstObstructionN(bG,bH).ker[1]);
> Print(FirstObstructionDnr(bG,bH).Dnr[1]);
> Print("\n");
> od;
[  ][  ]
[  ][  ]
[ 2 ][  ]
gap> cG:=cGs[3];; # there exists i=3 such that Obs(K/k)=Obs_1(L^-_i/K/k)
gap> bG:=cG.MinimalStemExtension;
<permutation group of size 3840 with 5 generators>
gap> bH:=PreImage(cG.epi,H);
<permutation group of size 240 with 4 generators>
gap> KerResH3Z(bG,bH:iterator); # KerResH3Z=0 => Obs(K/k)=Obs_1(L^-_3/K/k)
[ [  ], [ [ 2, 2, 2 ], [  ] ] ]
gap> HNPtruefalsefn:=x->FirstObstructionDr(bG,PreImage(cG.epi,x),bH).Dr[1]=[2];
function( x ) ... end
gap> Gcs:=ConjugacyClassesSubgroups(G);;
gap> Length(Gcs);
129
gap> GcsH:=ConjugacyClassesSubgroupsNGHOrbitRep(Gcs,H);;
gap> GcsHNPtf:=List(GcsH,x->List(x,HNPtruefalsefn));;
gap> Collected(List(GcsHNPtf,Set));
[ [ [ true ], 25 ], [ [ false ], 104 ] ]
gap> GcsHNPfalse:=List(Filtered([1..Length(Gcs)],
> x->false in GcsHNPtf[x]),y->Gcs[y]);;
gap> Length(GcsHNPfalse);
104
gap> GcsHNPtrue:=List(Filtered([1..Length(Gcs)],
> x->true in GcsHNPtf[x]),y->Gcs[y]);;
gap> Length(GcsHNPtrue);
25
gap> Collected(List(GcsHNPfalse,x->StructureDescription(Representative(x))));
[ [ "(((C2 x C2 x C2 x C2) : C2) : C3) : C2", 1 ], 
  [ "((C2 x C2 x C2 x C2) : C2) : C2", 1 ], 
  [ "((C2 x C2 x C2 x C2) : C2) : C3", 1 ], 
  [ "((C2 x C2 x C2 x C2) : C3) : C2", 2 ], 
  [ "((C2 x C2 x C2 x C2) : C5) : C2", 1 ], 
  [ "((C2 x C2 x C2 x C2) : C5) : C4", 1 ], 
  [ "((C2 x C2 x C2) : C4) : C2", 1 ], [ "((C4 x C4) : C2) : C2", 1 ], 
  [ "((C4 x C4) : C3) : C2", 1 ], [ "(C2 x C2 x C2 x C2) : C2", 2 ], 
  [ "(C2 x C2 x C2 x C2) : C3", 1 ], [ "(C2 x C2 x C2 x C2) : C5", 1 ], 
  [ "(C2 x C2 x C2) : (C2 x C2)", 1 ], [ "(C2 x C2 x C2) : C4", 4 ], 
  [ "(C2 x Q8) : C2", 1 ], [ "(C4 x C2) : C2", 8 ], [ "(C4 x C4) : C2", 2 ], 
  [ "(C4 x C4) : C3", 1 ], [ "(C4 x C4) : C4", 1 ], [ "(C8 : C2) : C2", 1 ], 
  [ "1", 1 ], [ "A4", 5 ], [ "A5", 2 ], [ "C2", 3 ], [ "C2 x A4", 1 ], 
  [ "C2 x C2", 9 ], [ "C2 x C2 x C2", 4 ], [ "C2 x C2 x C2 x C2", 1 ], 
  [ "C2 x D8", 5 ], [ "C2 x Q8", 1 ], [ "C2 x S4", 1 ], [ "C3", 1 ], 
  [ "C4", 4 ], [ "C4 x C2", 6 ], [ "C4 x C4", 1 ], [ "C5", 1 ], 
  [ "C5 : C4", 1 ], [ "C6", 1 ], [ "C8", 1 ], [ "C8 : C2", 1 ], [ "D10", 1 ], 
  [ "D12", 1 ], [ "D8", 7 ], [ "Q16", 1 ], [ "Q8", 2 ], [ "QD16", 1 ], 
  [ "S3", 2 ], [ "S4", 5 ], [ "S5", 1 ] ]
gap> Collected(List(GcsHNPtrue,x->StructureDescription(Representative(x))));
[ [ "((((C4 x C4) : C2) : C2) : C3) : C2", 1 ], 
  [ "(((C2 x C2 x C2 x C2) : C2) : C2) : C2", 1 ], 
  [ "(((C4 x C4) : C2) : C2) : C3", 1 ], 
  [ "((C2 x C2 x C2 x C2) : C2) : C2", 1 ], 
  [ "((C2 x C2 x C2 x C2) : C3) : C2", 1 ], 
  [ "((C2 x C2 x C2) : C4) : C2", 1 ], [ "((C8 : C2) : C2) : C2", 1 ], 
  [ "(C2 x C2 x C2 x C2) : A5", 1 ], [ "(C2 x C2 x C2 x C2) : C2", 2 ], 
  [ "(C2 x C2 x C2 x C2) : C3", 1 ], [ "(C2 x C2 x C2 x C2) : S5", 1 ], 
  [ "(C4 x C2) : C2", 1 ], [ "(C4 x C4) : C2", 1 ], [ "(C4 x C4) : C3", 1 ], 
  [ "A4", 1 ], [ "A5", 1 ], [ "C2 x C2", 1 ], [ "C2 x C2 x C2", 1 ], 
  [ "C2 x C2 x C2 x C2", 1 ], [ "C2 x D8", 1 ], [ "C4 x C4", 1 ], 
  [ "D8", 1 ], [ "S4", 1 ], [ "S5", 1 ] ]
gap> GcsHNPtrueMin:=MinConjugacyClassesSubgroups(GcsHNPtrue);;
gap> Collected(List(GcsHNPtrueMin,x->StructureDescription(Representative(x))));
[ [ "C2 x C2", 1 ], [ "C4 x C4", 1 ], [ "D8", 1 ] ]
gap> IsInvariantUnderAutG(GcsHNPtrueMin);
false
gap> GcsHNPfalseC2xC2:=Filtered(GcsHNPfalse, # G(4,2)=C2xC2
> x->IdSmallGroup(Representative(x))=[4,2]);;
gap> Length(GcsHNPfalseC2xC2);
9
gap> GcsHNPtrueC2xC2:=Filtered(GcsHNPtrue, # G(4,2)=C2xC2
> x->IdSmallGroup(Representative(x))=[4,2]);;
gap> Length(GcsHNPtrueC2xC2);
1
gap> IsInvariantUnderAutG(GcsHNPtrueC2xC2);
true
gap> ChG:=CharacteristicSubgroups(G);
[ Group(()), Group([ (1,2)(3,16)(4,7)(5,6)(8,11)(9,10)(12,15)(13,14), (1,8)
  (2,11)(3,10)(4,13)(5,12)(6,15)(7,14)(9,16), (1,9)(2,10)(3,11)(4,12)(5,13)
  (6,14)(7,15)(8,16), (1,4)(2,7)(3,6)(5,16)(8,13)(9,12)(10,15)(11,14) ]), 
  <permutation group of size 960 with 7 generators>, t16n1329 ]
gap> List(last,Order);
[ 1, 16, 960, 1920 ]
gap> Collected(List(GcsHNPfalseC2xC2,x->List(ChG,y->Order(Intersection(Representative(x),y)))));
[ [ [ 1, 1, 2, 4 ], 1 ], [ [ 1, 1, 4, 4 ], 2 ], [ [ 1, 2, 2, 4 ], 1 ], 
  [ [ 1, 2, 4, 4 ], 2 ], [ [ 1, 4, 4, 4 ], 3 ] ]
gap> Collected(List(GcsHNPtrueC2xC2,x->List(ChG,y->Order(Intersection(Representative(x),y)))));
[ [ [ 1, 1, 4, 4 ], 1 ] ]
gap> Collected(List(GcsHNPfalseC2xC2,x->Length(Elements(x))));
# The number of conjuugates of each V4 
[ [ 5, 1 ], [ 10, 1 ], [ 20, 3 ], [ 30, 1 ], [ 60, 2 ], [ 120, 1 ] ]
gap> Collected(List(GcsHNPtrueC2xC2,x->Length(Elements(x))));
# The number of conjuugates of each V4 
[ [ 40, 1 ] ]
gap> GcsHNPfalseD4:=Filtered(GcsHNPfalse, # G(8,3)=D4
> x->IdSmallGroup(Representative(x))=[8,3]);;
gap> Length(GcsHNPfalseD4);
7
gap> GcsHNPtrueD4:=Filtered(GcsHNPtrue, # G(8,3)=D4
> x->IdSmallGroup(Representative(x))=[8,3]);;
gap> Length(GcsHNPtrueD4);
1
gap> IsInvariantUnderAutG(GcsHNPtrueD4);
false
gap> Collected(List(List(GcsHNPfalseD4,Representative),
> x->SortedList(List(Orbits(x),Length))));
[ [ [ 2, 4, 8 ], 1 ], [ [ 4, 4, 4, 4 ], 2 ], [ [ 4, 4, 8 ], 4 ] ]
gap> Collected(List(List(GcsHNPtrueD4,Representative),
> x->SortedList(List(Orbits(x),Length))));
[ [ [ 2, 2, 4, 4, 4 ], 1 ] ]
gap> GcsHNPfalseC4xC4:=Filtered(GcsHNPfalse, # G(16,2)=C4xC4
> x->IdSmallGroup(Representative(x))=[16,2]);;
gap> Length(GcsHNPfalseC4xC4);
1
gap> GcsHNPtrueC4xC4:=Filtered(GcsHNPtrue, # G(16,2)=C4xC4
> x->IdSmallGroup(Representative(x))=[16,2]);;
gap> Length(GcsHNPtrueC4xC4);
1
gap> IsInvariantUnderAutG(GcsHNPtrueC4xC4);
true
gap> Collected(List(GcsHNPfalseC4xC4,x->List(ChG,y->Order(Intersection(Representative(x),y)))));
[ [ [ 1, 4, 16, 16 ], 1 ] ]
gap> Collected(List(GcsHNPtrueC4xC4,x->List(ChG,y->Order(Intersection(Representative(x),y)))));
[ [ [ 1, 4, 16, 16 ], 1 ] ]
gap> Collected(List(GcsHNPfalseC4xC4,x->Length(Elements(x))));
# The number of conjugates of each C4xC4
[ [ 5, 1 ] ]
gap> Collected(List(GcsHNPtrueC4xC4,x->Length(Elements(x))));
# The number of conjugates of each C4xC4
[ [ 10, 1 ] ]
\end{verbatim}
}~\\\vspace*{-4mm}

(5) $16T1654\simeq (C_2)^4\rtimes A_6$. 
{\small 
\begin{verbatim}
gap> Read("HNP.gap");
gap> Read("FlabbyResolutionFromBase.gap");
gap> Filtered(H1(FlabbyResolutionNorm1TorusJ(16,1654).actionF),x->x>1); # H^1(G,[J_{G/H}]^{fl})
[ 2 ]
gap> G:=TransitiveGroup(16,1654); # G=16T1654
t16n1654
gap> H:=Stabilizer(G,1);
Group([ (2,10,11)(4,12,13)(6,14,15)(8,9,16), (3,12,16,15)(4,10,8,5)(6,7,9,11)
(13,14) ])
gap> FirstObstructionN(G,H).ker;
[ [  ], [ [  ], [  ] ] ]
gap> FirstObstructionDnr(G,H).Dnr;
[ [  ], [ [  ], [  ] ] ]
gap> GroupCohomology(G,3); # H^3(G,Z)=M(G): Schur multiplier of G
[ 2, 4, 3 ]
gap> cGs:=MinimalStemExtensions(G);;
gap> for cG in cGs do
> bG:=cG.MinimalStemExtension;
> bH:=PreImage(cG.epi,H);
> Print(FirstObstructionN(bG,bH).ker[1]);
> Print(FirstObstructionDnr(bG,bH).Dnr[1]);
> Print("\n");
> od;
[  ][  ]
[ 2 ][  ]
[  ][  ]
[  ][  ]
gap> cG:=cGs[2];; # there exists i=2 such that Obs(K/k)=Obs_1(L^-_i/K/k)
gap> bG:=cG.MinimalStemExtension;
Group([ (1,3,2)(4,10,7)(5,9,6)(8,12,11), (2,7,4)(5,11,8), (2,5)(3,6) ])
gap> bH:=PreImage(cG.epi,H);
Group([ (1,4,7)(8,11,9), (1,4,3,2)(5,9,8,6)(7,10)(11,12), (1,9)(2,5)(3,6)(4,8)
(7,11)(10,12) ])
gap> KerResH3Z(bG,bH:iterator); # KerResH3Z=0 => Obs(K/k)=Obs_1(L^-_2/K/k)
[ [  ], [ [ 2, 6 ], [  ] ] ]
gap> HNPtruefalsefn:=x->FirstObstructionDr(bG,PreImage(cG.epi,x),bH).Dr[1]=[2];
function( x ) ... end
gap> Gcs:=ConjugacyClassesSubgroups(G);;
gap> Length(Gcs);
139
gap> GcsH:=ConjugacyClassesSubgroupsNGHOrbitRep(Gcs,H);;
gap> GcsHNPtf:=List(GcsH,x->List(x,HNPtruefalsefn));;
gap> Collected(List(GcsHNPtf,Set));
[ [ [ true ], 30 ], [ [ false ], 109 ] ]
gap> GcsHNPfalse:=List(Filtered([1..Length(Gcs)],
> x->false in GcsHNPtf[x]),y->Gcs[y]);;
gap> Length(GcsHNPfalse);
109
gap> GcsHNPtrue:=List(Filtered([1..Length(Gcs)],
> x->true in GcsHNPtf[x]),y->Gcs[y]);;
gap> Length(GcsHNPtrue);
30
gap> Collected(List(GcsHNPfalse,x->StructureDescription(Representative(x))));
[ [ "(((C2 x C2 x C2 x C2) : C2) : C2) : C3", 1 ], 
  [ "((C2 x C2 x C2 x C2) : C2) : C2", 1 ], 
  [ "((C2 x C2 x C2 x C2) : C3) : C2", 1 ], 
  [ "((C2 x C2 x C2 x C2) : C5) : C2", 1 ], 
  [ "((C2 x C2 x C2) : (C2 x C2)) : C3", 1 ], 
  [ "((C2 x C2 x C2) : C4) : C2", 1 ], [ "((C4 x C4) : C3) : C2", 1 ], 
  [ "((C8 : C2) : C2) : C2", 1 ], [ "(A4 x A4) : C2", 1 ], 
  [ "(A4 x A4) : C4", 1 ], [ "(C2 x C2 x C2 x C2) : A5", 1 ], 
  [ "(C2 x C2 x C2 x C2) : C2", 1 ], [ "(C2 x C2 x C2 x C2) : C3", 1 ], 
  [ "(C2 x C2 x C2 x C2) : C5", 1 ], [ "(C2 x C2 x C2) : (C2 x C2)", 1 ], 
  [ "(C2 x C2 x C2) : C4", 3 ], [ "(C2 x Q8) : C2", 1 ], 
  [ "(C2 x Q8) : C4", 1 ], [ "(C2 x S4) : C2", 1 ], [ "(C3 x A4) : C2", 1 ], 
  [ "(C3 x C3) : C2", 1 ], [ "(C3 x C3) : C4", 1 ], [ "(C4 x C2) : C2", 6 ], 
  [ "(C4 x C4) : C2", 2 ], [ "(C4 x C4) : C3", 1 ], [ "(C6 x C2) : C2", 1 ], 
  [ "(C8 : C2) : C2", 1 ], [ "1", 1 ], [ "A4", 7 ], [ "A4 : C4", 1 ], 
  [ "A4 x A4", 1 ], [ "A5", 3 ], [ "A6", 1 ], [ "C2", 2 ], 
  [ "C2 . (((C2 x C2 x C2 x C2) : C3) : C2) = (((C2 x C2 x C2) : (C2 x C2)) : C3) . C2", 1 ], 
  [ "C2 . S4 = SL(2,3) . C2", 1 ], [ "C2 x A4", 2 ], 
  [ "C2 x C2", 7 ], [ "C2 x C2 x A4", 1 ], [ "C2 x C2 x C2", 3 ], 
  [ "C2 x C2 x C2 x C2", 1 ], [ "C2 x D8", 3 ], [ "C2 x Q8", 1 ], 
  [ "C2 x S4", 1 ], [ "C3", 2 ], [ "C3 : C4", 1 ], [ "C3 x A4", 1 ], 
  [ "C3 x C3", 1 ], [ "C4", 3 ], [ "C4 x C2", 4 ], [ "C4 x C4", 1 ], 
  [ "C5", 1 ], [ "C6", 1 ], [ "C6 x C2", 1 ], [ "C8", 1 ], [ "C8 : C2", 1 ], 
  [ "D10", 1 ], [ "D12", 1 ], [ "D8", 4 ], [ "Q16", 1 ], [ "Q8", 2 ], 
  [ "QD16", 1 ], [ "S3", 2 ], [ "S4", 6 ], [ "SL(2,3)", 1 ] ]
gap> Collected(List(GcsHNPtrue,x->StructureDescription(Representative(x))));
[ [ "((((C2 x C2 x C2 x C2) : C2) : C2) : C3) : C2", 1 ], 
  [ "((((C2 x C2 x C2) : (C2 x C2)) : C3) : C2) : C2", 1 ], 
  [ "(((C2 x C2 x C2 x C2) : C2) : C2) : C3", 1 ], 
  [ "(((C2 x C2 x C2) : (C2 x C2)) : C3) : C2", 1 ], 
  [ "(((C4 x C4) : C2) : C2) : C2", 1 ], 
  [ "((C2 x C2 x C2 x C2) : C2) : C2", 2 ], 
  [ "((C2 x C2 x C2 x C2) : C3) : C2", 1 ], 
  [ "((C2 x C2 x C2) : C4) : C2", 1 ], [ "(C2 x C2 x C2 x C2) : A5", 1 ], 
  [ "(C2 x C2 x C2 x C2) : A6", 1 ], [ "(C2 x C2 x C2 x C2) : C2", 2 ], 
  [ "(C2 x C2 x C2 x C2) : C3", 1 ], [ "(C4 x C2) : C2", 1 ], 
  [ "(C4 x C4) : C2", 1 ], [ "(C4 x C4) : C3", 1 ], [ "A4", 1 ], [ "A5", 2 ], 
  [ "A6", 1 ], [ "C2 x C2", 1 ], [ "C2 x C2 x C2", 1 ], 
  [ "C2 x C2 x C2 x C2", 1 ], [ "C2 x D8", 1 ], [ "C2 x S4", 1 ], 
  [ "C4 x C4", 1 ], [ "D8", 1 ], [ "S4", 2 ] ]
gap> GcsHNPtrueMin:=MinConjugacyClassesSubgroups(GcsHNPtrue);;
gap> Collected(List(GcsHNPtrueMin,x->StructureDescription(Representative(x))));
[ [ "C2 x C2", 1 ], [ "C4 x C4", 1 ], [ "D8", 1 ] ]
gap> IsInvariantUnderAutG(GcsHNPtrueMin);
false
gap> GcsHNPfalseC2xC2:=Filtered(GcsHNPfalse, # G(4,2)=C2xC2
> x->Order(Representative(x))=4 and IdSmallGroup(Representative(x))=[4,2]);;
gap> Length(GcsHNPfalseC2xC2);
7
gap> GcsHNPtrueC2xC2:=Filtered(GcsHNPtrue, # G(4,2)=C2xC2
> x->Order(Representative(x))=4 and IdSmallGroup(Representative(x))=[4,2]);;
gap> Length(GcsHNPtrueC2xC2);
1
gap> IsInvariantUnderAutG(GcsHNPtrueC2xC2);
true
gap> NS16G:=Normalizer(SymmetricGroup(16),G);
Group([ (1,16,5,15,13,12,9,3)(2,4,10,7,14,8,6,11), (1,8,4,5,12,16)
(2,14,15,6,10,11)(3,7)(9,13), (2,9)(5,14)(7,12)(11,16) ])
gap> Collected(List(GcsHNPfalseC2xC2,
> x->IdSmallGroup(Normalizer(NS16G,Representative(x)))));
[ [ [ 64, 202 ], 1 ], [ [ 96, 226 ], 1 ], [ [ 128, 1755 ], 1 ], 
  [ [ 192, 1538 ], 2 ], [ [ 576, 8653 ], 1 ], [ [ 768, 1090134 ], 1 ] ]
gap> Collected(List(GcsHNPtrueC2xC2,
> x->IdSmallGroup(Normalizer(NS16G,Representative(x)))));
[ [ [ 96, 229 ], 1 ] ]
gap> GcsHNPfalseD4:=Filtered(GcsHNPfalse,
> x->Order(Representative(x))=8 and IdSmallGroup(Representative(x))=[8,3]);;
gap> Length(GcsHNPfalseD4);
4
gap> GcsHNPtrueD4:=Filtered(GcsHNPtrue,
> x->Order(Representative(x))=8 and IdSmallGroup(Representative(x))=[8,3]);;
gap> Length(GcsHNPtrueD4);
1
gap> IsInvariantUnderAutG(GcsHNPtrueD4);
false
gap> Collected(List(List(GcsHNPfalseD4,Representative),
> x->SortedList(List(Orbits(x),Length))));
[ [ [ 2, 4, 8 ], 1 ], [ [ 4, 4, 4, 4 ], 1 ], [ [ 4, 4, 8 ], 2 ] ]
gap> Collected(List(List(GcsHNPtrueD4,Representative),
> x->SortedList(List(Orbits(x),Length))));
[ [ [ 2, 2, 4, 4, 4 ], 1 ] ]
gap> GcsHNPfalseC4xC4:=Filtered(GcsHNPfalse, # G(16,2)=C4xC4
> x->Order(Representative(x))=16 and IdSmallGroup(Representative(x))=[16,2]);;
gap> Length(GcsHNPfalseC4xC4);
1
gap> GcsHNPtrueC4xC4:=Filtered(GcsHNPtrue, # G(16,2)=C4xC4
> x->Order(Representative(x))=16 and IdSmallGroup(Representative(x))=[16,2]);;
gap> Length(GcsHNPtrueC4xC4);
1
gap> IsInvariantUnderAutG(GcsHNPtrueC4xC4);
true
gap> Collected(List(GcsHNPfalseC4xC4,x->Length(Elements(x))));
# The number of conjugates of each C4xC4
[ [ 15, 1 ] ]
gap> Collected(List(GcsHNPtrueC4xC4,x->Length(Elements(x))));
# The number of conjugates of each C4xC4
[ [ 30, 1 ] ]
\end{verbatim}
}~\\\vspace*{-4mm}

(6) $16T1753\simeq (C_2)^4\rtimes S_6$. 
{\small 
\begin{verbatim}
gap> Read("HNP.gap");
gap> Read("FlabbyResolutionFromBase.gap");
gap> Filtered(H1(FlabbyResolutionNorm1TorusJ(16,1753).actionF),x->x>1); # H^1(G,[J_{G/H}]^{fl})
[ 2 ]
gap> G:=TransitiveGroup(16,1753);
t16n1753
gap> H:=Stabilizer(G,1);
Group([ (2,5,3,4)(6,7)(8,10,15,12)(9,11,14,13), (3,12,16,15)(4,9,8,6)
(5,7,10,11)(13,14) ])
gap> FirstObstructionN(G,H).ker;
[ [  ], [ [ 2 ], [  ] ] ]
gap> FirstObstructionDnr(G,H).Dnr;
[ [  ], [ [ 2 ], [  ] ] ]
gap> GroupCohomology(G,3); # H^3(G,Z)=M(G): Schur multiplier of G
[ 2, 2 ]
gap> cGs:=MinimalStemExtensions(G);;
gap> for cG in cGs do
> bG:=cG.MinimalStemExtension;
> bH:=PreImage(cG.epi,H);
> Print(FirstObstructionN(bG,bH).ker[1]);
> Print(FirstObstructionDnr(bG,bH).Dnr[1]);
> Print("\n");
> od;
[  ][  ]
[ 2 ][  ]
[  ][  ]
gap> cG:=cGs[2];; # there exists i=2 such that Obs(K/k)=Obs_1(L^-_i/K/k)
gap> bG:=cG.MinimalStemExtension;
<permutation group of size 23040 with 6 generators>
gap> bH:=PreImage(cG.epi,H);
Group([ (1,5,12,9)(3,14,8,6)(10,20,17,16)(13,19,18,15), (1,2,5,6)(3,9,12,8)
(4,14)(7,13,16,15)(10,18,19,17)(11,20), (1,15)(2,7)(3,10)(4,11)(5,13)(6,16)
(8,17)(9,18)(12,19)(14,20) ])
gap> KerResH3Z(bG,bH:iterator); # KerResH3Z=0 => Obs(K/k)=Obs_1(L^-_2/K/k)
[ [  ], [ [ 2, 2 ], [  ] ] ]
gap> HNPtruefalsefn:=x->FirstObstructionDr(bG,PreImage(cG.epi,x),bH).Dr[1]=[2];
function( x ) ... end
gap> Gcs:=ConjugacyClassesSubgroups(G);;
gap> Length(Gcs);
498
gap> GcsH:=ConjugacyClassesSubgroupsNGHOrbitRep(Gcs,H);;
gap> GcsHNPtf:=List(GcsH,x->List(x,HNPtruefalsefn));;
gap> Collected(List(GcsHNPtf,Set));
[ [ [ true ], 143 ], [ [ false ], 355 ] ]
gap> GcsHNPfalse:=List(Filtered([1..Length(Gcs)],
> x->false in GcsHNPtf[x]),y->Gcs[y]);;
gap> Length(GcsHNPfalse);
355
gap> GcsHNPtrue:=List(Filtered([1..Length(Gcs)],
> x->true in GcsHNPtf[x]),y->Gcs[y]);;
gap> Length(GcsHNPtrue);
143
gap> Collected(List(GcsHNPfalse,x->StructureDescription(Representative(x))));
[ [ "(((((C2 x C2 x C2 x C2) : C3) : C2) : C3) : C2) : C2", 1 ], 
  [ "((((C2 x C2 x C2) : (C2 x C2)) : C3) : C2) : C2", 1 ], 
  [ "((((C4 x C2 x C2) : C2) : C2) : C3) : C2", 1 ], 
  [ "(((C2 x C2 x C2 x C2) : C2) : C2) : C2", 1 ], 
  [ "(((C2 x C2 x C2 x C2) : C2) : C2) : C3", 1 ], 
  [ "(((C2 x C2 x C2 x C2) : C2) : C3) : C2", 1 ], 
  [ "(((C2 x C2 x C2 x C2) : C3) : C2) : C3", 1 ], 
  [ "(((C2 x C2 x C2) : (C2 x C2)) : C3) : C2", 2 ], 
  [ "(((C2 x Q8) : C2) : C2) : C2", 1 ], 
  [ "(((C4 x C2 x C2) : C2) : C2) : C2", 1 ], 
  [ "(((C4 x C2 x C2) : C2) : C2) : C3", 1 ], 
  [ "(((C4 x C4) : C2) : C3) : C2", 1 ], [ "((A4 x A4) : C2) : C2", 1 ], 
  [ "((C2 x C2 x C2 x C2) : C2) : C2", 3 ], 
  [ "((C2 x C2 x C2 x C2) : C2) : C3", 1 ], 
  [ "((C2 x C2 x C2 x C2) : C3) : C2", 2 ], 
  [ "((C2 x C2 x C2 x C2) : C5) : C2", 1 ], 
  [ "((C2 x C2 x C2 x C2) : C5) : C4", 1 ], 
  [ "((C2 x C2 x C2) : (C2 x C2)) : C2", 2 ], 
  [ "((C2 x C2 x C2) : (C2 x C2)) : C3", 1 ], 
  [ "((C2 x C2 x C2) : C4) : C2", 4 ], [ "((C4 x C2 x C2) : C2) : C2", 2 ], 
  [ "((C4 x C2) : C2) : C4", 2 ], [ "((C4 x C4) : C2) : C2", 5 ], 
  [ "((C4 x C4) : C2) : C3", 1 ], [ "((C4 x C4) : C3) : C2", 2 ], 
  [ "((C8 : C2) : C2) : C2", 3 ], [ "(A4 x A4) : C2", 1 ], 
  [ "(A4 x A4) : C4", 1 ], [ "(C2 x C2 x C2 x C2) : A5", 1 ], 
  [ "(C2 x C2 x C2 x C2) : C2", 7 ], [ "(C2 x C2 x C2 x C2) : C3", 1 ], 
  [ "(C2 x C2 x C2 x C2) : C5", 1 ], [ "(C2 x C2 x C2 x C2) : S5", 1 ], 
  [ "(C2 x C2 x C2) : (C2 x C2)", 4 ], [ "(C2 x C2 x C2) : C4", 10 ], 
  [ "(C2 x Q8) : C2", 3 ], [ "(C2 x Q8) : C4", 2 ], [ "(C2 x S4) : C2", 3 ], 
  [ "(C3 x A4) : C2", 1 ], [ "(C3 x C3) : C2", 1 ], [ "(C3 x C3) : C4", 1 ], 
  [ "(C4 x C2 x C2) : C2", 5 ], [ "(C4 x C2) : C2", 21 ], 
  [ "(C4 x C4) : C2", 10 ], [ "(C4 x C4) : C3", 1 ], [ "(C4 x D8) : C2", 1 ], 
  [ "(C6 x C2) : C2", 2 ], [ "(C8 : C2) : C2", 3 ], [ "(D8 x D8) : C2", 1 ], 
  [ "(S3 x S3) : C2", 1 ], [ "1", 1 ], [ "A4", 6 ], [ "A4 : C4", 1 ], 
  [ "A4 x A4", 1 ], [ "A4 x S3", 1 ], [ "A5", 3 ], [ "A6", 1 ], [ "C12", 1 ], [ "C2", 5 ], 
  [ "C2 . (((C2 x C2 x C2 x C2) : C3) : C2) = (((C2 x C2 x C2) : (C2 x C2)) : C3) . C2", 1 ], 
  [ "C2 . S4 = SL(2,3) . C2", 1 ], 
  [ "C2 x ((C4 x C2) : C2)", 1 ], [ "C2 x A4", 6 ], [ "C2 x C2", 18 ], 
  [ "C2 x C2 x A4", 2 ], [ "C2 x C2 x C2", 15 ], [ "C2 x C2 x C2 x C2", 3 ], 
  [ "C2 x C2 x D8", 2 ], [ "C2 x C2 x S3", 2 ], [ "C2 x C2 x S4", 2 ], 
  [ "C2 x D8", 22 ], [ "C2 x Q8", 1 ], [ "C2 x S4", 9 ], [ "C3", 2 ], 
  [ "C3 : C4", 1 ], [ "C3 x A4", 1 ], [ "C3 x C3", 1 ], [ "C3 x D8", 1 ], 
  [ "C3 x S3", 2 ], [ "C3 x S4", 1 ], [ "C4", 6 ], [ "C4 : C4", 3 ], 
  [ "C4 x A4", 1 ], [ "C4 x C2", 13 ], [ "C4 x C2 x C2", 3 ], 
  [ "C4 x C4", 2 ], [ "C4 x D8", 1 ], [ "C4 x S3", 1 ], [ "C4 x S4", 1 ], 
  [ "C5", 1 ], [ "C5 : C4", 1 ], [ "C6", 3 ], [ "C6 x C2", 2 ], [ "C8", 2 ], 
  [ "C8 : (C2 x C2)", 2 ], [ "C8 : C2", 3 ], [ "D10", 1 ], [ "D12", 6 ], 
  [ "D16", 2 ], [ "D24", 1 ], [ "D8", 20 ], [ "D8 x A4", 1 ], 
  [ "D8 x D8", 1 ], [ "D8 x S3", 1 ], [ "GL(2,3)", 1 ], [ "Q16", 2 ], 
  [ "Q8", 3 ], [ "QD16", 4 ], [ "S3", 5 ], [ "S3 x S3", 2 ], [ "S4", 11 ], 
  [ "S4 x A4", 1 ], [ "S4 x D8", 1 ], [ "S4 x S3", 1 ], [ "S4 x S4", 1 ], 
  [ "S5", 2 ], [ "S6", 1 ], [ "SL(2,3)", 1 ] ]
gap> Collected(List(GcsHNPtrue,x->StructureDescription(Representative(x))));
[ [ "(((((C2 x C2 x C2 x C2) : C2) : C2) : C3) : C2) : C2", 1 ], 
  [ "(((((C2 x C2 x C2) : (C2 x C2)) : C3) : C2) : C2) : C2", 1 ], 
  [ "((((C2 x C2 x C2 x C2) : C2) : C2) : C3) : C2", 2 ], 
  [ "((((C2 x C2 x C2) : (C2 x C2)) : C3) : C2) : C2", 1 ], 
  [ "((((C4 x C2 x C2) : C2) : C2) : C3) : C2", 1 ], 
  [ "((((C4 x C4) : C2) : C2) : C2) : C3", 1 ], 
  [ "(((C2 x C2 x C2 x C2) : C2) : C2) : C2", 2 ], 
  [ "(((C2 x C2 x C2 x C2) : C2) : C2) : C3", 1 ], 
  [ "(((C2 x C2 x C2) : (C2 x C2)) : C3) : C2", 1 ], 
  [ "(((C2 x C2 x C2) : C4) : C2) : C2", 1 ], 
  [ "(((C4 x C2 x C2) : C2) : C2) : C2", 1 ], 
  [ "(((C4 x C4) : C2) : C2) : C2", 2 ], 
  [ "((C2 x ((C2 x C2 x C2) : (C2 x C2))) : C2) : C3", 1 ], 
  [ "((C2 x C2 x C2 x C2) : C2) : C2", 6 ], 
  [ "((C2 x C2 x C2 x C2) : C3) : C2", 2 ], 
  [ "((C2 x C2 x C2) : (C2 x C2)) : C2", 1 ], 
  [ "((C2 x C2 x C2) : C4) : C2", 1 ], [ "((C4 x C2 x C2) : C2) : C2", 2 ], 
  [ "((C4 x C4) : C2) : C2", 1 ], [ "((C4 x C4) : C2) : C3", 1 ], 
  [ "((C8 : C2) : C2) : C2", 1 ], [ "((D8 x D8) : C2) : C2", 1 ], 
  [ "(C2 x (((C2 x C2 x C2) : (C2 x C2)) : C3)) : C2", 1 ], 
  [ "(C2 x ((C2 x C2 x C2 x C2) : C2)) : C2", 1 ], 
  [ "(C2 x ((C2 x C2 x C2) : (C2 x C2))) : C2", 1 ], 
  [ "(C2 x ((C2 x C2 x C2) : C4)) : C2", 1 ], 
  [ "(C2 x ((C4 x C2) : C2)) : C2", 1 ], [ "(C2 x C2 x C2 x C2) : A5", 1 ], 
  [ "(C2 x C2 x C2 x C2) : A6", 1 ], [ "(C2 x C2 x C2 x C2) : C2", 8 ], 
  [ "(C2 x C2 x C2 x C2) : C3", 1 ], [ "(C2 x C2 x C2 x C2) : S5", 1 ], 
  [ "(C2 x C2 x C2 x C2) : S6", 1 ], [ "(C2 x Q8) : C2", 1 ], 
  [ "(C2 x SL(2,3)) : C2", 1 ], [ "(C4 x C2 x C2) : C2", 4 ], 
  [ "(C4 x C2) : C2", 5 ], [ "(C4 x C4) : C2", 2 ], [ "(C4 x C4) : C3", 1 ], 
  [ "(C8 x C2) : C2", 2 ], [ "A4", 1 ], [ "A5", 1 ], [ "A6", 1 ], 
  [ "C2 x ((((C2 x C2 x C2) : (C2 x C2)) : C3) : C2)", 1 ], 
  [ "C2 x (((C2 x C2 x C2 x C2) : C3) : C2)", 1 ], 
  [ "C2 x (((C2 x C2 x C2) : (C2 x C2)) : C2)", 1 ], 
  [ "C2 x (((C2 x C2 x C2) : (C2 x C2)) : C3)", 1 ], 
  [ "C2 x (((C8 : C2) : C2) : C2)", 1 ], 
  [ "C2 x ((C2 x C2 x C2 x C2) : C2)", 2 ], 
  [ "C2 x ((C2 x C2 x C2 x C2) : C3)", 1 ], 
  [ "C2 x ((C2 x C2 x C2) : (C2 x C2))", 1 ], 
  [ "C2 x ((C2 x C2 x C2) : C4)", 3 ], [ "C2 x ((C4 x C2) : C2)", 6 ], 
  [ "C2 x ((C8 : C2) : C2)", 1 ], [ "C2 x (C8 : C2)", 1 ], [ "C2 x A4", 3 ], 
  [ "C2 x C2", 2 ], [ "C2 x C2 x A4", 1 ], [ "C2 x C2 x C2", 8 ], 
  [ "C2 x C2 x C2 x C2", 5 ], [ "C2 x C2 x C2 x C2 x C2", 1 ], 
  [ "C2 x C2 x D8", 3 ], [ "C2 x C2 x S4", 1 ], [ "C2 x D8", 7 ], 
  [ "C2 x Q8", 1 ], [ "C2 x S4", 5 ], [ "C2 x SL(2,3)", 1 ], [ "C4 x C2", 2 ],
  [ "C4 x C2 x C2", 3 ], [ "C4 x C4", 1 ], [ "C4 x D8", 1 ], 
  [ "C8 : (C2 x C2)", 1 ], [ "C8 x C2", 1 ], [ "D8", 2 ], [ "S4", 3 ], 
  [ "S5", 1 ], [ "S6", 1 ] ]
gap> GcsHNPtrueMin:=MinConjugacyClassesSubgroups(GcsHNPtrue);;
gap> Collected(List(GcsHNPtrueMin,x->StructureDescription(Representative(x))));
[ [ "(C4 x C2) : C2", 1 ], [ "C2 x C2", 2 ], [ "C2 x Q8", 1 ], 
  [ "C4 x C2", 1 ], [ "C4 x C4", 1 ], [ "C8 x C2", 1 ], [ "D8", 2 ] ]
gap> Collected(List(GcsHNPtrueMin,x->IdSmallGroup(Representative(x))));
[ [ [ 4, 2 ], 2 ], [ [ 8, 2 ], 1 ], [ [ 8, 3 ], 2 ], [ [ 16, 2 ], 1 ], 
  [ [ 16, 3 ], 1 ], [ [ 16, 5 ], 1 ], [ [ 16, 12 ], 1 ] ]
gap> IsInvariantUnderAutG(GcsHNPtrueMin);
false
gap> GcsHNPfalseC2xC2:=Filtered(GcsHNPfalse, # G(4,2)=C2xC2
> x->Order(Representative(x))=4 and IdSmallGroup(Representative(x))=[4,2]);;
gap> Length(GcsHNPfalseC2xC2);
18
gap> GcsHNPtrueC2xC2:=Filtered(GcsHNPtrue, # G(4,2)=C2xC2
> x->Order(Representative(x))=4 and IdSmallGroup(Representative(x))=[4,2]);;
gap> Length(GcsHNPtrueC2xC2);
2
gap> IsInvariantUnderAutG(GcsHNPtrueC2xC2);
false
gap> Collected(List(List(GcsHNPfalseC2xC2,Representative),
> x->SortedList(List(Orbits(x),Length))));
[ [ [ 2, 2, 2, 2, 2, 2, 2, 2 ], 1 ], [ [ 2, 2, 2, 2, 2, 2, 4 ], 2 ], 
  [ [ 2, 2, 2, 2, 4 ], 1 ], [ [ 2, 2, 2, 2, 4, 4 ], 5 ], 
  [ [ 2, 2, 2, 4, 4 ], 2 ], [ [ 2, 2, 4, 4 ], 1 ], [ [ 2, 2, 4, 4, 4 ], 1 ], 
  [ [ 4, 4, 4 ], 1 ], [ [ 4, 4, 4, 4 ], 4 ] ]
gap> Collected(List(List(GcsHNPtrueC2xC2,Representative),
> x->SortedList(List(Orbits(x),Length))));
[ [ [ 2, 2, 2, 2, 2, 2, 2, 2 ], 1 ], [ [ 2, 2, 2, 2, 2, 2, 4 ], 1 ] ]
gap> NS16G:=Normalizer(SymmetricGroup(16),G);
Group([ (1,10,11,6,9,2,3,14)(4,13,8,15,12,5,16,7), (1,12,6,4,7,8)
(2,3,15,9,13,10)(5,11,14) ])
gap> Collected(List(List(GcsHNPfalseC2xC2,Representative),
> x->[SortedList(List(Orbits(x),Length)),Order(Normalizer(NS16G,x))]));
[ [ [ [ 2, 2, 2, 2, 2, 2, 2, 2 ], 768 ], 1 ], 
  [ [ [ 2, 2, 2, 2, 2, 2, 4 ], 32 ], 1 ], 
  [ [ [ 2, 2, 2, 2, 2, 2, 4 ], 192 ], 1 ], [ [ [ 2, 2, 2, 2, 4 ], 64 ], 1 ], 
  [ [ [ 2, 2, 2, 2, 4, 4 ], 32 ], 1 ], [ [ [ 2, 2, 2, 2, 4, 4 ], 64 ], 2 ], 
  [ [ [ 2, 2, 2, 2, 4, 4 ], 128 ], 2 ], [ [ [ 2, 2, 2, 4, 4 ], 32 ], 1 ], 
  [ [ [ 2, 2, 2, 4, 4 ], 96 ], 1 ], [ [ [ 2, 2, 4, 4 ], 64 ], 1 ], 
  [ [ [ 2, 2, 4, 4, 4 ], 64 ], 1 ], [ [ [ 4, 4, 4 ], 192 ], 1 ], 
  [ [ [ 4, 4, 4, 4 ], 64 ], 1 ], [ [ [ 4, 4, 4, 4 ], 256 ], 1 ], 
  [ [ [ 4, 4, 4, 4 ], 576 ], 1 ], [ [ [ 4, 4, 4, 4 ], 768 ], 1 ] ]
gap> Collected(List(List(GcsHNPtrueC2xC2,Representative),
> x->[SortedList(List(Orbits(x),Length)),Order(Normalizer(NS16G,x))]));
[ [ [ [ 2, 2, 2, 2, 2, 2, 2, 2 ], 64 ], 1 ], 
  [ [ [ 2, 2, 2, 2, 2, 2, 4 ], 96 ], 1 ] ]
gap> GcsHNPfalseC4xC2:=Filtered(GcsHNPfalse, # G(8,2)=C4xC2
> x->Order(Representative(x))=8 and IdSmallGroup(Representative(x))=[8,2]);;
gap> Length(GcsHNPfalseC4xC2);
13
gap> GcsHNPtrueC4xC2:=Filtered(GcsHNPtrue, # G(8,2)=C4xC2
> x->Order(Representative(x))=8 and IdSmallGroup(Representative(x))=[8,2]);;
gap> Length(GcsHNPtrueC4xC2);
2
gap> IsInvariantUnderAutG(GcsHNPtrueC4xC2);
false
gap> Collected(List(List(GcsHNPfalseC4xC2,Representative),
> x->SortedList(List(Orbits(x),Length))));
[ [ [ 2, 2, 4, 8 ], 2 ], [ [ 2, 4, 8 ], 1 ], [ [ 4, 4, 4, 4 ], 1 ], 
  [ [ 4, 4, 8 ], 4 ], [ [ 8, 8 ], 5 ] ]
gap> Collected(List(List(GcsHNPtrueC4xC2,Representative),
> x->SortedList(List(Orbits(x),Length))));
[ [ [ 2, 2, 4, 4, 4 ], 1 ], [ [ 8, 8 ], 1 ] ]
gap> Collected(List(List(GcsHNPfalseC4xC2,Representative),
> x->[SortedList(List(Orbits(x),Length)),Order(Normalizer(NS16G,x))]));
[ [ [ [ 2, 2, 4, 8 ], 64 ], 2 ], [ [ [ 2, 4, 8 ], 32 ], 1 ], 
  [ [ [ 4, 4, 4, 4 ], 128 ], 1 ], [ [ [ 4, 4, 8 ], 32 ], 1 ], 
  [ [ [ 4, 4, 8 ], 64 ], 1 ], [ [ [ 4, 4, 8 ], 128 ], 2 ], 
  [ [ [ 8, 8 ], 128 ], 3 ], [ [ [ 8, 8 ], 256 ], 2 ] ]
gap> Collected(List(List(GcsHNPtrueC4xC2,Representative),
> x->[SortedList(List(Orbits(x),Length)),Order(Normalizer(NS16G,x))]));
[ [ [ [ 2, 2, 4, 4, 4 ], 32 ], 1 ], [ [ [ 8, 8 ], 32 ], 1 ] ]
gap> GcsHNPfalseD4:=Filtered(GcsHNPfalse, # G(8,3)=D4
> x->Order(Representative(x))=8 and IdSmallGroup(Representative(x))=[8,3]);;
gap> Length(GcsHNPfalseD4);
20
gap> GcsHNPtrueD4:=Filtered(GcsHNPtrue, # G(8,3)=D4
> x->Order(Representative(x))=8 and IdSmallGroup(Representative(x))=[8,3]);;
gap> Length(GcsHNPtrueD4);
2
gap> IsInvariantUnderAutG(GcsHNPtrueD4);
false
gap> Collected(List(List(GcsHNPfalseD4,Representative),
> x->SortedList(List(Orbits(x),Length))));
[ [ [ 2, 2, 4, 8 ], 2 ], [ [ 2, 4, 4, 4 ], 2 ], [ [ 2, 4, 8 ], 2 ], 
  [ [ 4, 4, 4, 4 ], 5 ], [ [ 4, 4, 8 ], 6 ], [ [ 8, 8 ], 3 ] ]
gap> Collected(List(List(GcsHNPtrueD4,Representative),
> x->SortedList(List(Orbits(x),Length))));
[ [ [ 2, 2, 4, 4, 4 ], 2 ] ]
gap> GcsHNPfalseC4xC4:=Filtered(GcsHNPfalse, # G(16,2)=C4xC4
> x->Order(Representative(x))=16 and IdSmallGroup(Representative(x))=[16,2]);;
gap> Length(GcsHNPfalseC4xC4);
2
gap> GcsHNPtrueC4xC4:=Filtered(GcsHNPtrue, # G(16,2)=C4xC4
> x->Order(Representative(x))=16 and IdSmallGroup(Representative(x))=[16,2]);;
gap> Length(GcsHNPtrueC4xC4);
1
gap> IsInvariantUnderAutG(GcsHNPtrueC4xC4);
true
gap> Collected(List(GcsHNPfalseC4xC4,x->Length(Elements(x))));
[ [ 15, 1 ], [ 45, 1 ] ]
gap> Collected(List(GcsHNPtrueC4xC4,x->Length(Elements(x))));
[ [ 30, 1 ] ]
gap> GcsHNPfalseG16_3:=Filtered(GcsHNPfalse, # G(16,3)=(C4xC2):C2
> x->Order(Representative(x))=16 and IdSmallGroup(Representative(x))=[16,3]);;
gap> Length(GcsHNPfalseG16_3);
15
gap> GcsHNPtrueG16_3:=Filtered(GcsHNPtrue, # G(16,3)=(C4xC2):C2
> x->Order(Representative(x))=16 and IdSmallGroup(Representative(x))=[16,3]);;
gap> Length(GcsHNPtrueG16_3);
3
gap> IsInvariantUnderAutG(GcsHNPtrueG16_3);
true
gap> Collected(List(GcsHNPfalseG16_3,
> x->Collected(List(Elements(x),y->Order(Intersection(Representative(x),y))))));
# The number of conjugates of each G16_3
[ [ [ [ 1, 24 ], [ 2, 10 ], [ 4, 10 ], [ 16, 1 ] ], 1 ], 
  [ [ [ 1, 24 ], [ 2, 18 ], [ 4, 2 ], [ 16, 1 ] ], 1 ], 
  [ [ [ 1, 48 ], [ 2, 28 ], [ 4, 11 ], [ 8, 2 ], [ 16, 1 ] ], 1 ], 
  [ [ [ 1, 48 ], [ 2, 28 ], [ 4, 12 ], [ 8, 1 ], [ 16, 1 ] ], 1 ], 
  [ [ [ 1, 48 ], [ 2, 32 ], [ 4, 4 ], [ 8, 5 ], [ 16, 1 ] ], 1 ], 
  [ [ [ 1, 48 ], [ 2, 32 ], [ 4, 7 ], [ 8, 2 ], [ 16, 1 ] ], 2 ], 
  [ [ [ 1, 48 ], [ 2, 36 ], [ 4, 3 ], [ 8, 2 ], [ 16, 1 ] ], 1 ], 
  [ [ [ 1, 48 ], [ 2, 36 ], [ 4, 4 ], [ 8, 1 ], [ 16, 1 ] ], 2 ], 
  [ [ [ 1, 144 ], [ 2, 28 ], [ 4, 6 ], [ 8, 1 ], [ 16, 1 ] ], 2 ], 
  [ [ [ 4, 40 ], [ 8, 4 ], [ 16, 1 ] ], 1 ], 
  [ [ [ 4, 42 ], [ 8, 2 ], [ 16, 1 ] ], 1 ], 
  [ [ [ 4, 83 ], [ 8, 6 ], [ 16, 1 ] ], 1 ] ]
gap> Collected(List(GcsHNPtrueG16_3,
> x->Collected(List(Elements(x),y->Order(Intersection(Representative(x),y)))))); 
# The number of conjugates of each G16_3
[ [ [ [ 1, 160 ], [ 2, 4 ], [ 4, 14 ], [ 8, 1 ], [ 16, 1 ] ], 2 ], 
  [ [ [ 1, 160 ], [ 2, 16 ], [ 4, 2 ], [ 8, 1 ], [ 16, 1 ] ], 1 ] ]
gap> GcsHNPfalseC8xC2:=Filtered(GcsHNPfalse, # G(16,5)=C8xC2
> x->Order(Representative(x))=16 and IdSmallGroup(Representative(x))=[16,5]);;
gap> Length(GcsHNPfalseC8xC2);
0
gap> GcsHNPtrueC8xC2:=Filtered(GcsHNPtrue, # G(16,5)=C8xC2
> x->Order(Representative(x))=16 and IdSmallGroup(Representative(x))=[16,5]);;
gap> Length(GcsHNPtrueC8xC2);
1
gap> GcsHNPfalseC2xQ8:=Filtered(GcsHNPfalse, # G(16,12)=C2xQ8
> x->Order(Representative(x))=16 and IdSmallGroup(Representative(x))=[16,12]);;
gap> Length(GcsHNPfalseC2xQ8);
1
gap> GcsHNPtrueC2xQ8:=Filtered(GcsHNPtrue, # G(16,12)=C2xQ8
> x->Order(Representative(x))=16 and IdSmallGroup(Representative(x))=[16,12]);;
gap> Length(GcsHNPtrueC2xQ8);
1
gap> IsInvariantUnderAutG(GcsHNPtrueC2xQ8);
true
gap> ChG:=CharacteristicSubgroups(G);
[ Group(()), Group([ (1,2)(3,16)(4,7)(5,6)(8,11)(9,10)(12,15)(13,14), (1,6)
  (2,5)(3,4)(7,16)(8,15)(9,14)(10,13)(11,12), (1,10)(2,9)(3,8)(4,15)(5,14)
  (6,13)(7,12)(11,16), (1,3)(2,16)(4,6)(5,7)(8,10)(9,11)(12,14)(13,15) ]), 
  <permutation group of size 5760 with 6 generators>, t16n1753 ]
gap> List(ChG,Order);
[ 1, 16, 5760, 11520 ]
gap> Collected(List(GcsHNPtrueC2xQ8,
> x->List(ChG,y->Order(Intersection(Representative(x),y)))));
[ [ [ 1, 2, 8, 16 ], 1 ] ]
gap> StructureDescription(ChG[2]);
"C2 x C2 x C2 x C2"
\end{verbatim}
}~\\\vspace*{-4mm}
\end{example}

{\it Proof of Corollary \ref{cor1.5}.}
By the definition, we have $0\to \bZ\xrightarrow{\varepsilon^\circ}\bZ[G/H]\to J_{G/H}\to 0$ where 
$J_{G/H}\simeq \widehat{T}={\rm Hom}(T,\bG_m)$. 
Then we get $H^1(G,\bZ[G/H])\to H^1(G,J_{G/H})$ 
$\xrightarrow{\delta}$ $H^2(G,\bZ)\xrightarrow{\varphi} H^2(G,\bZ[G/H])$ 
where $\delta$ is the connecting homomorphism. 
We have $H^2(G,\bZ)\simeq H^1(G,\bQ/\bZ)
={\rm Hom}(G,\bQ/\bZ)\simeq G^{ab}$. 
By Shapiro's lemma, we also have  
$H^1(G,\bZ[G/H])\simeq H^1(H,\bZ)={\rm Hom}(H,\bZ)=0$
and 
$H^2(G,\bZ[G/H])\simeq H^2(H,\bZ)\simeq H^1(H,\bQ/\bZ)={\rm Hom}(H,\bQ/\bZ)=H^{ab}$.
Then we get 
$0\to H^1(G,J_{G/H})\xrightarrow{\delta} 
G^{ab}={\rm Hom}(G,\bQ/\bZ)\xrightarrow{\varphi} H^{ab}={\rm Hom}(H,\bQ/\bZ)$ 
where $\varphi$ is the restriction map from $G$ to $H$. 
If $f\in {\rm Ker}\,\varphi$, then $f|_H=0$. 
Hence $H\leq {\rm Ker}\,f\lhd G$. 
On the other hand, $H={\rm Stab}_1(G)$ is not normal in $G$ 
and $H$ is a maximal subgroup of $G$ because 
$G=16Tm$ is a primitive subgroup of $S_{16}$. 
This implies that ${\rm Ker}\,f=G$ and ${\rm Ker}\,\varphi=0$. 
Hence $\varphi$ is injective and $H^1(G,J_{G/H})=0$. 
Applying Ono's formula $\tau(T)=|H^1(k,\widehat{T})|/|\Sha(T)|$ with 
$H^1(k,\widehat{T})\simeq H^1(G,J_{G/H})=0$, we get $\tau(T)=1/|\Sha(T)|$.\qed 

\section{GAP algorithms}\label{S7}
The GAP algorithms for computing the total obstruction 
${\rm Obs}(K/k)$ and 
the first obstruction ${\rm Obs}_1(L/K/k)$ are given in Hoshi, Kanai and Yamasaki 
\cite[Section 6]{HKY22}, \cite[Section 6]{HKY23}. 
The below is used for the case $G=16T1080$ of the proof of Theorem \ref{thmain2} in Section \ref{S6}. 
The related functions are available 
as in \cite{Norm1ToriHNP}. 
~{}\vspace*{-2mm}\\
{\small 
\begin{verbatim}
AllSubgroups2 := function(G)
    Reset(GlobalMersenneTwister);
    Reset(GlobalRandomSource);
    return AllSubgroups(G);
end;

StemExtensions:= function(G)
    local ScG,ScGg,K,Ks,ans,m,pi,cG,cGg,iso,GG,GGg,Gg,epi,n,i,id;
    ScG:=SchurCoverG(G);
    ScGg:=GeneratorsOfGroup(ScG.SchurCover);
    K:=Kernel(ScG.epi);
    Ks:=AllSubgroups2(K);
    ans:=[];
    for m in Ks do
        pi:=NaturalHomomorphismByNormalSubgroup(ScG.SchurCover,m);
        cG:=Range(pi);
        cGg:=List(ScGg,x->Image(pi,x));
        iso:=IsomorphismPermGroup(Range(pi));
        GG:=Range(iso);
        GGg:=List(cGg,x->Image(iso,x));
        Gg:=List(ScGg,x->Image(ScG.epi,x));
        epi:=GroupHomomorphismByImages(GG,G,GGg,Gg);
        n:=NrMovedPoints(Source(epi));
        if n>=2 and n<=30 and IsTransitive(Source(epi),[1..n]) then
            for i in [1..NrTransitiveGroups(n)] do
                if Order(TransitiveGroup(n,i))=Order(Source(epi)) and
                 IsConjugate(SymmetricGroup(n),
                  TransitiveGroup(n,i),Source(epi)) then
                    id:=[n,i];
                    break;
                fi;
            od;
            Add(ans,rec(StemExtension:=Source(epi), epi:=epi, Tid:=id));
        else
            Add(ans,rec(StemExtension:=Source(epi), epi:=epi));
        fi;
    od;
    return ans;
end;
\end{verbatim}
}


\end{document}